\documentclass{article}
\usepackage{graphicx}
\usepackage[utf8]{inputenc}
\usepackage{amsmath}
\usepackage{amssymb}
\usepackage{amsthm}
\usepackage{enumerate}
\usepackage{cases}
\usepackage{tikz}
\usepackage[ruled,vlined,linesnumbered]{algorithm2e}
\SetKwInput{KwData}{Input}
\SetKwInOut{KwResult}{Output}
\usepackage[top=2.5cm,bottom=2.5cm,left=3cm,right=3cm]{geometry}
\usepackage{multirow}
\usepackage{multicol}
\usepackage{makecell}
\usepackage{subfigure}
\usepackage{bm}
\usepackage{booktabs}
\usepackage{empheq}
\newcolumntype{?}{!{\vrule width 1.5pt}}

\usepackage{hyperref}
\hypersetup{
    colorlinks=true,
    linkcolor=blue,
    filecolor=magenta,      
    urlcolor=cyan,
    citecolor = green,
    }

\newcommand{\tr}{\mathrm{tr}}

\newcommand{\Smnp}{{\mathcal{S}_M(n,p)}}
\newcommand{\bb}[1]{\mathbb{#1}}
\newcommand{\ca}[1]{\mathcal{#1}}
\newcommand{\Rnp}{{\mathbb{R}^{n\times p}}}

\newcommand{\tp}{^\top}
\newcommand{\inner}[1]{\left\langle #1 \right\rangle}
\newcommand{\norm}[1]{\left\Vert#1\right\Vert}
\newcommand{\A}{\mathcal{A}}

\newcommand{\mybrace}[1]{\left(#1\right)}
\newcommand{\sigmamax}[1]{\sigma_{\max}\left(#1\right)}
\newcommand{\sigmamin}[1]{\sigma_{\min}\left(#1\right)}

\usetikzlibrary{intersections}
\newtheorem{theorem}{Theorem}[section]
\newtheorem{definition}[theorem]{Definition}
\newtheorem{proposition}[theorem]{Proposition} 

\newtheorem{assumption}[theorem]{Assumption}
\newtheorem{lemma}[theorem]{Lemma}
\newtheorem{example}[theorem]{Example}

\newtheorem{remark}[theorem]{Remark}

\newtheorem{corollary}[theorem]{Corollary}

\newcommand{\st}{\mathrm{s.~t.}}

\tikzstyle{kuang} = [rectangle, rounded corners, text centered, draw = black]
\begin{document}
     \title{Stochastic optimization over expectation-formulated generalized Stiefel manifold}
\author{Linshuo Jiang\thanks{State Key Laboratory of Scientific and Engineering Computing, Academy of Mathematics and Systems Science, Chinese Academy of Sciences, and University of Chinese Academy of Sciences, China (e-mail: jianglinshuo@lsec.cc.ac.cn)}, ~ 
	Nachuan Xiao\thanks{The Institute of Operations Research and Analytics, National University of Singapore, Singapore. (e-mail: xnc@lsec.cc.ac.cn).},
	~ 
	Xin Liu\thanks{State Key Laboratory of Scientific and Engineering Computing, Academy of Mathematics and Systems Science, Chinese Academy of Sciences, and University of Chinese Academy of Sciences, China (e-mail: liuxin@lsec.cc.ac.cn)}}
     	\maketitle
      \begin{abstract}

In this paper, we consider a class of stochastic optimization problems over the expectation-formulated generalized Stiefel manifold \eqref{sogse}, where the objective function $f$ is continuously differentiable.
We propose a novel constraint dissolving penalty function with a customized penalty term \eqref{cdfcp}, which maintains the same order of differentiability as $f$. Our theoretical analysis establishes the global equivalence between \ref{cdfcp} and  \ref{sogse}, in the sense that they share the same first-order and second-order stationary points under mild conditions. These results on equivalence enable the direct implementation of various stochastic optimization approaches to solve \ref{sogse}. In particular, we develop a stochastic gradient algorithm and its accelerated variant by incorporating an adaptive step size strategy. Furthermore, we prove their $\mathcal{O}(\varepsilon^{-4})$ sample complexity for finding an $\varepsilon$-stationary point of \ref{cdfcp}. Comprehensive numerical experiments show the efficiency and robustness of our proposed algorithms.

      \end{abstract}
 \section{Introduction}\label{intro}
     In this paper, we focus on the following stochastic optimization problem with expectation constraints, 
    	\begin{equation}\tag{SOEGS}\label{sogse}
    	\begin{aligned}
    	\min_{X\in\mathbb{R}^{n\times p}}\;\;\; &f(X) := \bb{E}_{\xi\sim \Xi}[f_{\xi}(X)],\\
    		\st\;\;\; &X^\top \bb{E}_{\theta \sim \Theta}[M_{\theta}] X=I_p,
    	\end{aligned}
    \end{equation}
where $I_p$ refers to the $p$-by-$p$ identity matrix with $p\le n$, $f_{\xi}: \mathbb{R}^{n\times p} \to \mathbb{R}$ is a possibly nonconvex function parameterized by the random vector $\xi$ following some unknown distribution $\Xi$, and $M_{\theta} \in \mathbb{R}^{n\times n}$ is a matrix depending on the random vector $\theta$ following certain unknown distribution $\Theta$. Throughout this paper, we make the following assumptions on problem \ref{sogse}. 
\begin{assumption}[Blanket assumptions]\label{as:blanket}
        \begin{enumerate}[(1)]
            \item  The objective function $f(X)$ is differentiable and its gradient $\nabla f(X)$ is globally $L_g$-Lipschitz continuous over $\Rnp$. 
            \item  For almost every $\xi \in \Xi$, $f_{\xi}$ is differentiable. Moreover, it holds for any $X \in \Rnp$ that 
            $\bb{E}_{\xi\sim \Xi}[\nabla f_{\xi}(X)] = \nabla f(X).$
            \item The matrix $M := \bb{E}_{\theta \sim \Theta}[M_{\theta}]\in\mathbb{R}^{n\times n}$ is positive definite. 
        \end{enumerate}
        
\end{assumption}

It is worth mentioning that  the feasible region of \ref{sogse}, denoted as  $\Smnp:=\{X\in\mathbb{R}^{n\times p}:X\tp M X=I_p\}$, can be regarded as an embedded submanifold of $\Rnp$. Such manifold is usually referred to as the generalized Stiefel manifold \cite{absil_optimization_2008,boumal_manopt_2014,shustin_preconditioned_2021,sato_cholesky_2019}.

The optimization problem \ref{sogse} has a wide range of applications in scientific computing, data science and statistics \cite{bendory_nonconvex_2016,kasai_riemannian_2018,gao_orthogonalization-free_2022}, including the generalized canonical correlation analysis (GCCA) \cite{kettenring_canonical_1971,sorensen_generalized_2021,gao_sparse_2021} and the linear discriminant analysis (LDA) \cite{chen_sparse_2013,luo_linear_2011}. We present a brief introduction to GCCA problem in the following. 

\begin{example}[GCCA problem]\label{pr:1}
 Given $m\in\mathbb{N}$ $(m\ge2)$ random vectors $\theta_1 \in \bb{R}^{n_1}$, $\theta_2 \in \bb{R}^{n_2}$, ... $\theta_m \in \bb{R}^{n_m}$, the GCCA problem aims to find a linear projection of these random vectors to a low-dimensional subspace $\mathbb{R}^p$, where $p \leq \min \{n_i,\;1\le i\le m\}$, such that the overall correlation of the projected vectors is maximized. According to \cite{kettenring_canonical_1971,gao_sparse_2021}, the GCCA problem can be formulated as the following optimization problem, 
\begin{equation}\label{p:GCCA}
    \begin{aligned}   
		&&\min_{X\in\mathbb{R}^{n\times p}}\;\;\;&-\sum_{r=1}^p\sum_{\substack{i,j=1\\i\neq j}}^m c_{ij} g\left({X^{[i]}_r}^\top \mathbb{E}[\mathcal{\theta}_i\mathcal{\theta}_j^\top]X^{[j]}_r\right),\\
   &&\st\;\;\;& X^\top MX=I_p,
   \end{aligned}
\end{equation}
where $n:=\sum_{i = 1}^m n_i$, $X^{[i]}\in\mathbb{R}^{n_i\times p}$ ($1\le i\le m$) is a submatrix of $X$ consisting of rows from $(\sum_{j=1}^{i-1} n_j)+1$ to $\sum_{j=1}^{i} n_j$, and $X^{[i]}_r$ denotes the $r$-th column of $X^{[i]}$ for $r=1,\dots,p$. Moreover, $M\in\mathbb{R}^{n\times n}$ is an expectation-formulated matrix defined by
\[M:= \mathbb{E}\left[\begin{pmatrix}
	 \mathcal{\theta}_1\mathcal{\theta}_1^\top&&&\\ 
	  & \mathcal{\theta}_2\mathcal{\theta}_2^\top&&\\
    &&\ddots&\\
	   &&&\mathcal{\theta}_m\mathcal{\theta}_m^\top\\
   \end{pmatrix}\right],\]
and $g: \bb{R} \to \bb{R}$ refers to the merit function. Popular choices of $g$ include the identity mapping $g(t)=t$ \cite{horst1961generalized} and the absolute value function $g(t)=|t|$ \cite{steel1951minimum}. The latter one is nonsmooth and a widely used smooth approximation for it is the Huber loss function \cite{huber1992robust}, defined as 
             \[g(t):=\left\{\begin{aligned}
           &\frac{1}{2}t^2, &&|t|\le\mu,\\
           &\mu\left(|t|-\frac{1}{2}\mu\right),&&|t|>\mu,
        \end{aligned}\right.\]
        where $\mu>0$ refers to the Huber threshold parameter.

 It is worth mentioning that when we choose $m=2$ and $g(t)=t$, the model \eqref{p:GCCA} yields the following classical formulation of canonical correlation analysis (CCA) \cite{hotelling_cca_1936, hardoon_canonical_2004, arora_cca_2017},
\begin{equation}\label{eq:classcca}
     \begin{aligned}	&&\min_{\substack{X^{[1]}\in\mathbb{R}^{n_1\times p},\\X^{[2]}\in\mathbb{R}^{n_2\times p}}}\;\;\;&-\frac{1}{2}\operatorname{tr}\mybrace{{X^{[1]}}^\top\mathbb{E}\left[\theta_1\theta_2^\top
\right] X^{[2]}},\\
	   &&\st\;\;\;& {X^{[1]}}^\top \mathbb{E}[
         \theta_1\theta_1^\top] X^{[1]}=I_p,\;{X^{[2]}}^\top \mathbb{E}[
         \theta_2\theta_2^\top] X^{[2]}=I_p.
	   \end{aligned}
\end{equation}
\end{example}


\subsection{Related works}\label{existwork}

For nonlinear programming problems where the feasible region admits a differentiable manifold structure,
a great number of existing algorithms fall into the category of Riemannian optimization approaches \cite{absil_optimization_2008,boumal2023introduction},   including Riemannian gradient methods \cite{abrudan_steepest_2008,iannazzo_r_2018}, Riemannian conjugate gradient methods \cite{sato_new_2015,sato_riemannian_2022}, Riemannian trust-region methods \cite{absil_trust_2007,baker_riemannian_2008}, Riemannian Newton methods \cite{adler_newton_2002}, Riemannian quasi-Newton methods \cite{huang_broyden_2015}, etc. Typically, as mentioned in \cite{absil_trust_2007,absil_optimization_2008}, these methods typically operate through an iterative process: they first ``lift'' the objective function from the manifold to the tangent space at the current iterate, which provides a local linear approximation of the underlying manifold. Subsequently, they execute one step of the unconstrained optimization algorithm within this tangent space, followed by ``pulling'' the resulting point back onto the manifold to generate the next iterate. These ``lift to" and ``pull back" operations \cite{absil_trust_2007,absil_optimization_2008} are implemented via exponential mapping or, for computational efficiency, through its relaxed alternatives known as retractions.  Interested readers could refer to the books \cite{absil_optimization_2008,boumal2023introduction} and a recent survey paper \cite{hu2020brief} for more details.

 However, these approaches necessitate the noiseless evaluations of differential geometric components of the underlying manifold, such as the tangent spaces and the retractions. This requirement presents a significant challenge for the generalized Stiefel manifold in problem \ref{sogse} as its expectation formulation makes these geometric components computationally intractable. As a result, all these aforementioned Riemannian optimization approaches cannot be directly applied to solve \ref{sogse}.

 For expectation-constrained optimization problems, several methods have been proposed, including proximal-point-based methods   \cite{ma2019proximally, Boob2022stochastic} and the inexact augmented Lagrangian method \cite{li2022stochastic}. However, these methods rely on critical theoretical assumptions (cf. \cite[Assumption 1B]{ma2019proximally}, \cite[Assumption 3.3]{Boob2022stochastic}, and \cite[Assumption 3]{li2022stochastic}) that fail to hold in the context of \ref{sogse}, hence making them inapplicable to \ref{sogse}.

Very recently,  for equality-constrained optimization problems, \cite{xiao_constraint_2022} presents a novel penalty function, named the constraint dissolving function \eqref{cdf}. When applied to problem \ref{sogse}, a possible choice of \ref{cdf} can be expressed as,
\begin{equation}\label{cdf}\tag{CDF}
   f(\mathcal{A}(X))+\frac{\beta}{4}\|X^\top MX-I_p\|_F^2.
\end{equation}
Here, $\mathcal{A}(X):=X\left(\frac{3}{2}I_p-\frac{1}{2}X^\top MX\right)$ refers to the constraint dissolving operator, the second term is the penalty term, and $\beta>0$ is the penalty parameter. With a sufficiently large parameter $\beta$, \ref{cdf} and \ref{sogse} share the same stationary points and local minimizers within a neighborhood of the feasible region $\ca{S}_{M}(n,p)$, as demonstrated in \cite[Theorem 4.5]{xiao_constraint_2022}. This implies that the iterates should be restricted in this neighborhood to ensure that the algorithm for minimizing \ref{cdf} theoretically converges to the critical points of \ref{sogse}. Moreover, a similar requirement is also imposed in the convergence of the landing method for solving problem \ref{sogse}  \cite[Theorem 2.9]{vary2024optimization}. However, controlling the distance between iterates and the feasible region becomes challenging when $M$ takes expectation formulation, as the noiseless evaluation of the feasibility is unavailable. Therefore, these aforementioned algorithms cannot be applied to solve \ref{sogse} with global convergence guarantees. To the best of our knowledge, the efficient optimization approach for solving \ref{sogse} with global convergence guarantees remains undiscovered.

\subsection{Motivation}

As Assumption \ref{as:blanket} assumes the global Lipschitz continuity of $\nabla f(X)$, the growth rate of $f(\A(X))$ can be estimated by $\limsup_{\norm{X} \to +\infty} |f(\A(X))|/ \norm{X}^6 < +\infty$. Therefore, by choosing a specific penalty term for the constraints $X\tp MX = I_p$, the resulting penalty function can be bounded from below.
Specifically,  we consider the following constraint dissolving penalty function with a customized sixth-order penalty term. 
 \begin{equation}\tag{CDFCP}\label{cdfcp}
 \begin{aligned}    h(X):={}&f(\mathcal{A}(X))+\frac{\beta}{6} \mathrm{tr}\left( X\tp M X((X\tp MX)^2 - 3I_p) \right). 
 \end{aligned}
\end{equation}
As we will show later, due to the introduction of this sixth-order penalty term, all the first-order stationary points of \ref{cdfcp} are restricted within a neighborhood of $\ca{S}_{M}(n, p)$ as the parameter $\beta$ is sufficiently large (see Proposition \ref{pro:sigma_up} in Section \ref{subsec:fo}). Consequently, although the iterates may travel far away from $\ca{S}_{M}(n, p)$ when employing stochastic optimization approaches in minimizing \ref{cdfcp}, these iterates eventually stabilize within a neighborhood of $\ca{S}_{M}(n, p)$.

\subsection{Contributions}\label{contri}
In this paper, we develop an novel penalty function named \ref{cdfcp} for stochastic optimization problems over the expectation-formulated generalized Stiefel manifold \eqref{sogse}. 
First, we prove that the minimization of \ref{cdfcp} is globally equivalent to problem \ref{sogse}. 
Specifically, we show that all first-order stationary point of \ref{cdfcp} are either first-order stationary points of \ref{sogse}, or strict saddle points of \ref{cdfcp}.  Furthermore, we prove that \ref{cdfcp} shares the same second-order stationary points with \ref{sogse}.
Notice that minimizing \ref{cdfcp} can be regarded as a nested optimization problem \cite{wang2017stochastic}.
These theoretical results enable the direct implementations of various existing nested stochastic optimization approaches to solve \ref{sogse} via \ref{cdfcp}.

Particularly, we propose a stochastic gradient algorithm for solving \ref{sogse} by employing the inner function value tracking technique to deal with the nested expectation structure emerging in \ref{cdfcp}. Instead of storing an $n$-by-$n$ matrix to track the expectation matrix $M$, our proposed algorithm only requires the storage of a $p$-by-$p$ matrix to track the value of $X^\top MX$, which significantly reduces its memory cost in solving \ref{sogse}. Additionally, our proposed algorithm only requires matrix multiplications, hence can be efficiently implemented in practice. We establish the $O(\varepsilon^{-4})$ sample complexity of this algorithm for reaching an $\varepsilon$-stationary point of \ref{cdfcp}. By incorporating an adaptive step size strategy, we also develop its accelerated variant, which maintains the same sample complexity while demonstrating faster practical convergence.
Finally, we conduct extensive numerical experiments on the GCCA problem to demonstrate the efficiency of our proposed algorithms. The experimental results illustrates that our proposed algorithms exhibit superior performance compared to existing approaches.

\subsection{Organization}\label{org}

The rest of this paper is organized as follows. Section \ref{sec:prelimi}
 introduces preliminary materials, including the notations used throughout this paper and the optimality conditions for \ref{sogse} and \ref{cdfcp}, respectively. In Section \ref{equivalence}, we analyze the relationships between \ref{sogse} and \ref{cdfcp}. Section \ref{algorithm} presents a stochastic gradient algorithm and its accelerated variant for solving \ref{sogse}, together with a detailed analysis of their sample complexities. In Section \ref{numerical}, we present numerical experiments to demonstrate the efficiency and robustness of our proposed algorithms through comparisons with state-of-the-art stochastic algorithms for solving the GCCA problem. We conclude our paper in the last section.

\section{Preliminaries}\label{sec:prelimi}

In this section, we present the basic notations used throughout this paper and introduce some necessary definitions for our theoretical analysis.

\subsection{Notations}\label{notation}
In this paper, we adopt the following notations.   Given a matrix $A$, we denote its $i,j$-th element as $A(i,j)$. We use $\sigma_{\max}(A)$ and $\sigma_{\min}(A)$ to denote its largest and smallest singular values, respectively, and hence $\kappa(A) := \frac{\sigmamax{M}}{\sigmamin{M}}$  stands for the condition number of $A$.  Furthermore, for any square matrix $A$, we denote its trace as $\operatorname{tr}(A)$, and define its symmetrization as $\Psi(A):= \frac{1}{2}(A + A^\top)$. For two matrices $X,Y$ with the same order, the Euclidean inner product is defined as $\langle X,Y\rangle:=\operatorname{tr}(X^\top Y)$, and $X\odot Y$, $(X)^{\odot 2}$ refer to the element-wise multiplication and the element-wise square, respectively.  $\|X\|$ represents the Frobenius norm of a matrix $X$.  Given a differentiable function $q(X):\Rnp\to\mathbb{R}$, the $i,j$-th entry of its gradient $\nabla q(X)$ is represented by $\nabla_{i,j} q(X)$.
  $\mathbb{S}^p$ stands for the set of all symmetric matrices in $\mathbb{R}^{p\times p}$. We denote $\mathcal{T}_{\Smnp}(X)$ as the tangent space of $\mathcal{S}_M(n,p)$ at $X$, which is given by
\begin{equation}\label{eq:tx}
    \mathcal{T}_{\Smnp}(X):= \left\{D\in\mathbb{R}^{n\times p}| \Psi(D^\top M X) = 0\right\},
\end{equation}
while $\mathcal{N}_{\Smnp}(X)$ refers to the normal space of $\Smnp$ at $X$,
\begin{equation}\label{eq:nx}
    \mathcal{N}_{\Smnp}(X):=\{XS\in\mathbb{R}^{n\times p}|\;S\in\mathbb{S}^p\}. 
\end{equation}
$\Omega$ denotes a neighborhood of $\ca{S}_M(n,p)$ defined as follows
\begin{equation}
    \label{eq:omega}
    \Omega:=\{X\in\Rnp:\sigma_{\max}\left(X\tp MX\right)\le 1\}.
\end{equation}
$\Omega_r$ with $r>0$ represents a class of neighborhoods of $\ca{S}_M(n,p)$ given by 
\begin{equation}\label{eq:omegar}
    \Omega_{r}:=\{X\in\Rnp:\norm{X\tp MX-I_p}\le r\}. 
\end{equation}
For simplicity, we define the function $g(X): \bb{R}^{n\times p} \to \bb{R}$, the mapping $G(X): \bb{R}^{n\times p} \to \bb{R}^{n\times p}$ and $L(X): \bb{R}^{n\times p} \to \bb{R}^{n\times p}$ as follows,
\begin{equation}\label{eq:defgG}
    g(X):=f(\mathcal{A}(X)), \; G(X):=\nabla f(\mathcal{A}(X)),
\end{equation}
\begin{equation}
    \begin{aligned}
        L(X) :={}& \nabla f(X)-MX\Psi({X}^\top \nabla f(X)),
    \end{aligned}
\end{equation}
respectively.

\subsection{Optimality conditions}
We first introduce the optimality conditions for problem \ref{sogse} as follows. 
\begin{definition}\cite{absil_optimization_2008}\label{def:ocso}
We call $X\in\Rnp$ a first-order stationary point of \ref{sogse} if and only if
 \begin{equation}\label{eq:regrad}
 \left\{ \begin{aligned}
     & 0=\nabla f(X)-MX\Psi({X}^\top \nabla f(X)),\\
     &0={X}^\top MX-I_p.
 \end{aligned} \right.   
 \end{equation}
 Furthermore, we call $X\in\Rnp$ a second-order stationary point of \ref{sogse} if and only if $X$ is a first-order stationary point of \ref{sogse} and for any $T\in\mathcal{T}_{\Smnp}(X)$, it holds that 
\[
		\inner{D,M^{-1}\nabla^2 f(X)[T]-T{X}^\top\nabla f(X)}\ge0.
\]
\end{definition}

Then we present the definitions of the optimality conditions for minimizing the penalty function \ref{cdfcp}. 
\begin{definition}\cite{nocedal_numerical_1999}\label{def:occd}
    We call $X\in\Rnp$ a first-order stationary point of  \ref{cdfcp} if and only if $\nabla h(X)=0$. We call $X\in\Rnp$ a second-order stationary point of  \ref{cdfcp} if and only if $X$ is a first-order stationary point of  \ref{cdfcp} and for any $D\in\Rnp$,  
    \[\inner{D,\nabla^2 h(X)[D]}\ge0.\]
Furthermore, we call $X\in\Rnp$ a $c$-saddle point ($c>0$) of \ref{cdfcp}, if and only if $X$ is a first-order stationary point of \ref{cdfcp}, and there is a direction $D\in\Rnp$ such that
\[\inner{D,\nabla^2 h(X)[D]}\le-c\norm{D}^2.\]
\end{definition}

Then, in the following two definitions, we introduce the concept of the $\varepsilon$-first-order stationary points for \ref{sogse} and \ref{cdfcp}, respectively.
\begin{definition}
    We call $X\in\Rnp$ a $\varepsilon$-first-order stationary point of \ref{sogse} if and only if
 \begin{equation}\label{eq:epfos}
 \left\{ \begin{aligned}
     & \|f(X)-MX\Psi({X}^\top \nabla f(X))\|\le\varepsilon,\\
     &\|X\tp MX  - I_p\|\le \varepsilon.
 \end{aligned} \right.   
 \end{equation}
\end{definition}

\begin{definition}
    We call $X\in\Rnp$ a $\varepsilon$-first-order stationary point of \ref{cdfcp} if and only if
$\|\nabla h(X)\|\le \varepsilon$.
\end{definition}

\section{Equivalence}
\label{equivalence}

In this section, we demonstrate the equivalence between \ref{sogse} and \ref{cdfcp} with sufficiently large penalty parameter $\beta$. Specifically, Section \ref{subsec:cdfcp_properties} summarizes some fundamental properties of \ref{cdfcp} that are instrumental for subsequent analysis. Sections \ref{subsec:fo} and \ref{subsec:so} demonstrate the equivalence between \ref{cdfcp} and \ref{sogse} with respect to first-order and second-order stationary points, respectively. Finally, Section \ref{subsec:approxfo} presents their equivalence concerning
$\varepsilon$-first-order stationary points.

\subsection{Basic properties of \ref{cdfcp}}\label{subsec:cdfcp_properties}
We begin our theoretical analysis by providing several basic properties of \ref{cdfcp}. 
The following lemma presents the explicit expression of $\nabla h(X)$. The result in the following lemma straightforwardly follows from the expression of $h(X)$ and thus its proof is omitted for simplicity. 
\begin{lemma}
\label{Le_formulation_gradient_f}
  Suppose Assumption \ref{as:blanket} holds, then for any $X \in \Rnp$, the gradient of $h$ at $X$ can be expressed as 
    \begin{equation}\label{eq:gradh}
        \nabla h(X) = G(X)\left( \frac{3}{2}I_p - \frac{1}{2} X\tp M X \right) - MX\Lambda(X) + \beta MX ( (X\tp M X)^2 - I_p ),
    \end{equation}
    where $\Lambda(X) := \Psi(X\tp G(X))$.
\end{lemma}

Lemma \ref{Le_formulation_gradient_f} illustrates that the computation of $\nabla h(X)$ only contains matrix multiplications and hence can be computed efficiently in practice.

Then the following auxiliary lemma presents some basic properties of $\nabla h(X)$. 
\begin{lemma}
    \label{Le_inner_nablah_XQ}
    Suppose Assumption \ref{as:blanket} holds, then for any $X \in \Rnp$ and any matrix $Q \in \mathbb{S}^p$ satisfying $QX\tp MX = X\tp MX Q$, it holds that 
    \begin{equation}\label{eq:innergrah}
         \inner{XQ, \nabla h(X)} = \mathrm{tr}\left( \left(\beta X\tp MX (X\tp MX + I_p) - \frac{3}{2}\Lambda(X) \right)(X\tp MX - I_p)Q  \right).
    \end{equation}
\end{lemma}
\begin{proof}
    For any $X \in \Rnp$ and any $Q\in \mathbb{S}^p$ that satisfies $QX\tp MX = X\tp MX Q$, it is easy to verify that $X\tp MXQ$ is a symmetric matrix. Then by the expression of $\nabla h(X)$ in 
    \eqref{eq:gradh}, we have
    \[\begin{aligned}
        &\inner{XQ,\nabla h(X)}\\
        \overset{}{=}{}&\left\langle XQ,G(X)\mybrace{\frac{3}{2}I_p-\frac{1}{2}X\tp MX}-MX\Lambda(X)\right.
        \left.+\beta MX\mybrace{\mybrace{X\tp MX}^2-I_p}\right\rangle\\
        ={}&\tr\mybrace{Q X\tp G(X)\mybrace{\frac{3}{2}I_p-\frac{1}{2}X\tp MX}}-\tr\mybrace{Q X\tp MX\Lambda(X)}\\
        &+\beta\tr\mybrace{Q X\tp MX\mybrace{\mybrace{X\tp MX}^2-I_p}}\\
        \overset{(i)}{=}{}&\tr\mybrace{\Lambda(X)\mybrace{\frac{3}{2}I_p-\frac{1}{2}X\tp MX}Q}-\tr\mybrace{\Lambda(X)X\tp MXQ}\\
        &+ \beta\tr\mybrace{X\tp MX\mybrace{X\tp MX+I_p}\mybrace{X\tp MX-I_p}Q}\\
        ={} &\mathrm{tr}\left( \left(\beta X\tp MX \mybrace{X\tp MX + I_p} - \frac{3}{2}\Lambda(X) \right)\mybrace{X\tp MX + I_p}Q  \right),
    \end{aligned}\]
    where the equation $(i)$ follows from the fact that $\mybrace{\frac{3}{2}I_p-\frac{1}{2}X\tp MX}Q$ is symmetric. This completes the proof. 
    
\end{proof}

According to Assumption \ref{as:blanket}, $\nabla f$ is globally $L_g$-Lipschitz continuous over $\Rnp$. Let  $L_0 :=\norm{\nabla f(0)}$, then it is easy to verify that 
\begin{equation}\label{eq:normgraf}
    \norm{\nabla f(X)} \leq L_g \norm{X} + L_0
\end{equation}
holds for any $X \in \Rnp$. 
Building on this, we establish an upper bound for the norm of $G(X)$ in the following lemma.
\begin{lemma}\label{le:normg}
    Suppose Assumption \ref{as:blanket} holds, then for any $X\in\Rnp$,  it holds that 
    \begin{equation}\label{eq:normg}
        \norm{G(X)} \leq  L_g\left(\frac{3}{2}  + \frac{1}{2} \sigma_{\max}(X\tp MX)  \right)\norm{X} + L_0.
    \end{equation}
\end{lemma}
\begin{proof}
By the expression of $G(X)$ in \eqref{eq:defgG}, it directly follows from \eqref{eq:normgraf} that 
    \begin{equation}
        \norm{G(X)}=\norm{\nabla f(\A(X))} \leq L_g\norm{\A(X)} + L_0 \leq L_g\left(\frac{3}{2}  + \frac{1}{2} \sigma_{\max}(X\tp MX)  \right)\norm{X} + L_0. 
    \end{equation}
    This completes the proof. 
\end{proof}

In the following, we introduce an additional assumption regarding the second-order differentiability of the objective function $f$, which is assumed to hold only in the context of the Hessian operator of $h$.

\begin{assumption}[The second-order differentiability of $f$]\label{as:s}
      $\nabla^2 f(X)$ exists at every $X\in\mathbb{R}^{n\times p}$.
\end{assumption}

Given Assumption \ref{as:s}, the twice-differentiability of $h$ 
  is straightforward to establish. Specifically, the following lemma shows the explicit expression of the Hessian operator of $h$. 
\begin{lemma}
    \label{Le_formulation_hessian_h}
     Suppose Assumption \ref{as:blanket} and Assumption \ref{as:s} hold, then for any $X,D \in \Rnp$, it holds that
      \begin{equation}\label{eq:hessh}
          \begin{aligned}
      &\nabla^2 h(X)[D]\\
      ={}& \mathcal{J}_G(X)[D]\left(\frac{3}{2}I_p-\frac{1}{2}X\tp MX\right)-G(X) \Psi(X\tp MD)-MD\Psi(X\tp G(X))\\
      &-MX\Psi(D\tp G(X))-MX\Psi(X\tp \mathcal{J}_G(X)[D])\\
      & +\beta MD((X\tp MX)^2-I_p)
      +4\beta MX\Psi(X\tp MD)X\tp MX,
  \end{aligned}
      \end{equation}
      where $\mathcal{J}_G(\cdot)[\cdot] : \Rnp \times \Rnp \to \Rnp $ is defined as
      \begin{equation}\label{eq:GD}
          \mathcal{J}_G(X)[D] : =\nabla^2 f(\mathcal{A}(X))\left[D\left(\frac{3}{2}I_p-\frac{1}{2}X\tp MX\right)-X\Psi(X\tp MD)\right].
      \end{equation}
\end{lemma}
The proof of Lemma \ref{Le_formulation_hessian_h} directly follows from the expression of $\nabla h(X)$ in \eqref{eq:gradh}, hence is omitted for simplicity. 
In the following lemma, we provide an upper bound for the norm of $\mathcal{J}_G(X)[D]$. 
\begin{lemma}\label{le:normj}
    Suppose Assumption \ref{as:blanket} and Assumption \ref{as:s} hold, then for any $X\in\Omega$, it holds that
    \begin{equation}
        \norm{\mathcal{J}_G(X)[D]} \leq (2+\kappa(M))L_g\norm{D}.
    \end{equation}
\end{lemma}
\begin{proof}

According to the expression  of $\norm{\mathcal{J}_G(X)[D]}$ in \eqref{eq:GD}, we have
    \begin{equation}
    \begin{aligned}
        \norm{\mathcal{J}_G(X)[D]}{}&\overset{(i)}{\le} L_g\norm{D\left(\frac{3}{2}I_p-\frac{1}{2}X\tp MX\right)-X\Psi(X\tp MD)}\\
        &\overset{(ii)}{\le} L_g\left(2\norm{D}+\norm{XX\tp MD}\right)\overset{(iii)}{\le}(2+\kappa(M))L_g\norm{D},
    \end{aligned}
    \end{equation}
where the inequality $(i)$  can be directly obtained  from Assumption \ref{as:blanket} that  $\nabla f$ is $L_g-$Lipschitz continuous, and the inequality $(ii)$ follows from the fact that $\sigmamax{X\tp MX}\le 1$. In addition, the  inequality $(iii)$ holds because 
\[\|XX\tp MD\|\le \sigmamax{XX\tp}\sigmamax{M}\|D\|\le\frac{\sigmamax{X\tp MX}\sigmamax{M}}{\sigmamin{M}}\norm{D}\le \kappa(M)\norm{D}.\]
The proof is completed. 
\end{proof}

Before analyzing the relationships between \ref{cdfcp} and \ref{sogse}, we define the thresholding value $\tilde{\beta}$ of the penalty parameter $\beta$ as 
\begin{equation}
    \tilde{\beta} \geq \frac{12\kappa(M)(3(p+1)L_g+\sigma_{\min}^{\frac{1}{2}}(M)L_0)}{\sigma_{\min}(M)}.  
\end{equation}

\subsection{Relationships on first-order stationary points}\label{subsec:fo}
In this subsection, we discuss the relationships between \ref{cdfcp} and \ref{sogse}  regarding first-order stationary points.
To begin with, the following proposition states that when $\beta\ge\tilde{\beta}$, all first-order stationary points of \ref{cdfcp} are restricted in the neighborhood $\Omega$ of $\Smnp$. 
\begin{proposition}\label{pro:sigma_up}
    Suppose Assumption \ref{as:blanket} holds, and $\beta > \tilde{\beta}$. Then for any first-order stationary point $X \in \Rnp$ of \ref{cdfcp}, it holds that $\sigmamax{X^\top MX}\le 1$, i.e., $X\in\Omega$.
\end{proposition}
\begin{proof}
    We prove this proposition by contradiction. Suppose $\sigma_{\max}(X\tp MX) > 1$, and let $v\in\mathbb{R}^p$ be a unit eigenvector of $X\tp MX$ corresponding to $\sigma_{\max}(X\tp MX)$. Since $X$ is a first-order stationary point of \ref{cdfcp}, we have
    \begin{equation}
        \begin{aligned}
            0 ={}& \inner{Xvv\tp, \nabla h(X)}\\
            \overset{(i)}{=}{}&(\sigma_{\max}(X\tp MX)-1)\left(\beta\sigma_{\max}(X\tp MX)(\sigma_{\max}(X\tp MX)+1)-\frac{3}{2}\tr\mybrace{vv\tp X\tp G(X)}\right)\\
            \overset{(ii)}{\ge}{}&(\sigma_{\max}(X\tp MX)-1)\left(\beta\sigma_{\max}(X\tp MX)(\sigma_{\max}(X\tp MX)+1)-\frac{3}{2}\sigma_{\max}(X^\top G(X))\right),
        \end{aligned}
    \end{equation}
where the equation $(i)$ follows from \eqref{eq:innergrah}, and the inequality $(ii)$ holds since $\tr\mybrace{vv\tp X\tp G(X)}=v\tp X\tp G(X)v\le\sigmamax{X^\top G(X)}$. Consequently, by  $\sigma_{\max}(X\tp MX) > 1$, we obtain
\begin{equation}\label{eq:stacontra1}
     0 \ge\beta\sigma_{\max}(X\tp MX)(\sigma_{\max}(X\tp MX)+1)-\frac{3}{2}\sigma_{\max}(X^\top G(X)).
\end{equation}
For the second term on the right-hand side of the above inequality, we have the following estimation.
    \begin{equation}\label{eq:normlambda}
        \begin{aligned}
            &\sigma_{\max}(X\tp G(X) )\leq \sigma_{\max}(X) \norm{G(X)}\\
            \overset{(i)}{\le}{}& \sigma_{\max}(X)\left(L_g\left(\frac{3}{2}  + \frac{1}{2} \sigma_{\max}(X\tp MX)  \right)\norm{X} + L_0 \right)\\
            \overset{(ii)}{\le}{}& \sigma_{\max}(X)\left( p L_g\left(\frac{3}{2}  + \frac{1}{2} \sigma_{\max}(X\tp MX)  \right)\sigma_{\max}(X) + L_0 \right)\\
            \leq{}&  p L_g\left(\frac{3}{2} + \frac{1}{2} \sigma_{\max}(X\tp MX)  \right)\sigma_{\max}^2(X) + L_0\sigma_{\max}(X)\\
            \overset{(iii)}{\le}{}&  p L_g\left(\frac{3}{2} + \frac{1}{2} \sigma_{\max}(X\tp MX)  \right)\frac{\sigma_{\max}(X\tp MX)}{\sigma_{\min}(M)} + L_0\frac{\sigma_{\max}^{\frac{1}{2}}(X\tp MX)}{\sigma_{\min}^{\frac{1}{2}}(M)} \\
            \leq{}&  p L_g\left(\frac{3}{2} + \frac{1}{2} \sigma_{\max}(X\tp MX)  \right)\frac{\sigma_{\max}(X\tp MX)}{\sigma_{\min}(M)} + L_0\frac{\sigma_{\max}(X\tp MX)}{\sigma_{\min}^{\frac{1}{2}}(M)} \\
            \leq{}& \frac{3 pL_g + 2L_0\sigma_{\min}^{\frac{1}{2}}(M) }{2\sigma_{\min} (M)} \sigma_{\max}(X\tp MX) + \frac{pL_g}{2\sigma_{\min}(M)} \sigma_{\max}^2(X\tp MX),
        \end{aligned}
    \end{equation}
    where the inequality $(i)$ follows from \eqref{eq:normg}, the inequality $(ii)$ follows from the fact that $\|X\|\le p\sigma_{\max}(X)$, and the inequality $(iii)$ utilizes the following estimation of $ \sigma_{\max}^2(X)$,
    \begin{equation}\label{eq:sigmax}
        \sigma_{\max}^2(X)=\sigmamax{X\tp X}\le\frac{\sigmamax{X^\top MX}}{\sigmamin{M}}.
    \end{equation}
    
    Combining \eqref{eq:stacontra1} with \eqref{eq:normlambda}, we immediately obtain that
    \begin{equation}
        \begin{aligned}
        0\ge{}&\mybrace{\beta\sigma_{\max}(X\tp MX)(\sigma_{\max}(X\tp MX)+1)-\frac{3}{2}\sigma_{\max}(X^\top G(X))} \\
            \geq{}& \mybrace{ \sigma_{\max}(X\tp MX) \left(\beta  - \frac{9 pL_g + 6L_0\sigma_{\min}^{\frac{1}{2}}(M)}{4\sigma_{\min}(M)}  \right) + \sigma_{\max}^2(X\tp MX) \left( \beta - \frac{3pL_g}{4\sigma_{\min}(M)}  \right)}\\
            >{}& 0,
        \end{aligned}
    \end{equation}
   which leads to a contradiction. Therefore, we can conclude that $\sigma_{\max}\left( X\tp M X \right) \leq 1$, and complete the proof. 
\end{proof}

The following proposition shows an equivalence relationship between \ref{sogse} and \ref{cdfcp} regarding first-order stationary points over the manifold $\Smnp$, without imposing any conditions on the penalty parameter $\beta$.

\begin{proposition}\label{Prop_first_order_equivalence_feasible}
   Suppose Assumption \ref{as:blanket} holds, then \ref{sogse} and \ref{cdfcp} share the same first-order stationary points over the manifold $\mathcal{S}_M(n,p)$. 
\end{proposition}
\begin{proof}
    For any $X\in\Smnp$, it follows from the feasibility of $X$ (i.e., $X\tp MX = I_p$) and the expression of $\nabla h(X)$ that 
    \begin{equation*}
        \nabla h(X) = \nabla f(X) - MX\Psi(X\tp \nabla f(X)). 
    \end{equation*}

    Then for any $X \in \Smnp$ that is a first-order stationary point of \ref{cdfcp}, from Definition  \ref{def:occd}, it holds that $0 = \nabla h(X) = \nabla f(X) - MX\Psi(X\tp \nabla f(X))$, hence Definition \ref{def:ocso} illustrates that $X$ is a first-order stationary point of \ref{sogse}. 

    Furthermore, for any $X \in \Smnp$ that is a first-order stationary point of \ref{sogse}, it follows from Definition \ref{def:ocso} that $\nabla h(X) = \nabla f(X) - MX\Psi(X\tp \nabla f(X)) = 0$. As a result, together with Definition \ref{def:occd}, we can conclude that $X$ is a first-order stationary point of \ref{cdfcp}. This completes the proof.  

\end{proof}

 For the relationship between first-order stationary points of \ref{sogse} and those of \ref{cdfcp} over the entire space $\Rnp$, the following theorem illustrates that a first-order stationary point of \ref{cdfcp} is either a first-order stationary point of \ref{sogse}, or sufficiently far away from $\Smnp$. 
\begin{theorem}
    \label{The_first_order_equivalence_feasible}
    Suppose Assumption \ref{as:blanket} holds and $\beta \geq \tilde{\beta}$, then for any $X \in \Rnp$ that is a first-order stationary point of \ref{cdfcp}, $X$ is either a first-order stationary point of \ref{sogse}, or satisfies $\sigma_{\min}(X\tp MX) \leq \frac{\tilde{\beta}}{4 \beta}$. 
\end{theorem}
\begin{proof}
    For any $X \in \Rnp$ that is a first-order stationary point of \ref{cdfcp}, it follows from the expression of $\nabla h(X)$ in \eqref{eq:gradh} and Lemma \ref{Le_inner_nablah_XQ} that 
\begin{equation}\label{eq:pro2-1}
    \begin{aligned}
      0{}&=\left\langle X(X\tp MX-I_p),\nabla h(X) \right\rangle\\
      &\overset{(i)}{=}\operatorname{tr}\left(\left(\beta X\tp MX(X\tp MX+I_p)-\frac{3}{2}\Lambda(X)\right)\left(X\tp MX-I_p\right)^2\right)\\
      &\overset{(ii)}{\ge}\left(\beta\sigma_{\min}(X^\top MX)-\frac{3}{2}\sigmamax{X\tp G(X)}\right)\left\|X\tp MX-I_p\right\|^2\\
      &\overset{(iii)}{\ge}\left(\beta\sigma_{\min}(X^\top MX)-\frac{6pL_g+3\sigma_{\min}^{\frac{1}{2}}(M)L_0}{2\sigma_{\min}(M)}\right)\left\|X\tp MX-I_p\right\|^2,
\end{aligned}
\end{equation} 
where the equation $(i)$ directly follows from Lemma \ref{Le_inner_nablah_XQ}, and the inequality $(ii)$ utilizes the expression of $\Lambda(X)$ and the fact that $X\tp MX(X\tp MX+I_p) \succeq X\tp MX$. Moreover, the inequality
$(iii)$ follows from Proposition \ref{pro:sigma_up} and \eqref{eq:normlambda}. Specifically, according to Proposition \ref{pro:sigma_up}, we have $\sigmamax{X\tp MX}\le1$, then by \eqref{eq:normlambda} it holds that $\sigmamax{X\tp G(X)}\le 2pL_g+\sigma_{\min}^{\frac{1}{2}}(M)L_0$. 

Then, when $X \in \Smnp$, Proposition \ref{Prop_first_order_equivalence_feasible} illustrates that $X$ is a first-order stationary point of \ref{sogse}. On the other hand, when $X \notin \Smnp$, it holds that $\|X\tp MX-I_p\|>0$. Therefore, \eqref{eq:pro2-1} illustrates that $\sigma_{\min}(X\tp MX)\le\frac{6pL_g+3\sigma_{\min}^{\frac{1}{2}}(M)L_0}{2\sigma_{
  \min}(M)\beta}$. Together with the fact that the penalty parameter $\beta$ satisfies $\beta\ge\tilde{\beta}\ge\frac{12pL_g+6\sigma_{\min}^{\frac{1}{2}}(M)L_0}{\sigma_{\min}(M)}$, it holds that $\sigma_{\min}(X\tp MX)\le\frac{\tilde{\beta}}{4\beta}$. The proof is completed.
\end{proof}

In the following proposition, we prove that all infeasible first-order stationary points of \ref{cdfcp} are strict saddle points of \ref{cdfcp} when $\beta\ge\tilde{\beta}$. 
\begin{proposition}
    \label{pro:saddle} 
    Suppose Assumptions \ref{as:blanket} and \ref{as:s} hold, and $\beta \geq \tilde{\beta}$. Then for any  $X \in \Rnp \setminus \ca{S}_{n,p}$ that is a first-order stationary point of \ref{cdfcp}, it holds that $X$ is a $\frac{\sigma_{\min}(M)}{4}\beta$-strict saddle point of \ref{cdfcp}.
\end{proposition}
\begin{proof}
  For any  $X \in \Rnp \setminus \ca{S}_{n,p}$ that is a first-order stationary point of \ref{cdfcp}, by Propositions \ref{pro:sigma_up} and \ref{The_first_order_equivalence_feasible}, we have $\sigmamax{X\tp MX}\le 1$ and $\sigmamin{X\tp MX}\le \frac{\tilde{\beta}}{4\beta}\le \frac{1}{4}$.
 Let $v\in\mathbb{R}^p$ be a unit eigenvector of $X\tp MX$ corresponding to $\sigma_{\min}(X\tp MX)$, and let $D:=Xvv\tp$. Then by the expression of $\nabla^2 h(X)[D]$ in \eqref{eq:hessh}, we have
  \[\begin{aligned}
      &\inner{D,\nabla^2 h(X)[D]}\\
      \overset{}{=}{}&\left\langle D, \mathcal{J}_G(X)[D]\left(\frac{3}{2}I_p-\frac{1}{2}X\tp MX\right)\right\rangle-3\sigma_{\min}(X\tp MX)\langle D, G(X)\rangle\\
     & -\sigma_{\min}(X\tp MX)\inner{D,\mathcal{J}_G(X)[D]}+\beta\sigma_{\min}(X\tp MX)\left(5\sigma_{\min}^2(X\tp MX)-1\right)\\
     \le{} & \norm{D}\norm{\mathcal{J}_G(X)[D]}\mybrace{\frac{3}{2}+\frac{1}{2}\sigmamax{X\tp MX}}+3\sigma_{\min}(X\tp MX)\norm{D}
     \norm{G(X)}\\
     &+\beta\sigma_{\min}(X\tp MX)\left(5\sigma_{\min}^2(X\tp MX)-1\right)\\
     \le{} & 3\norm{D}\norm{\mathcal{J}_G(X)[D]}+3\sigma_{\min}(X\tp MX)\norm{D}
     \norm{G(X)}+\beta\sigma_{\min}(X\tp MX)\left(\frac{5}{16}-1\right)\\
     \overset{(i)}{\le}{} & 3(2+\kappa(M))L_g\norm{D}^2+3\sigma_{\min}^\frac{1}{2}(X\tp MX)\norm{D}
     \mybrace{\frac{2L_g p}{\sigma_{\min}^{\frac{1}{2}}(M)}+L_0}-\beta\frac{\sigma_{\min}(X\tp MX)}{2}\\
      \overset{(ii)}{\le} {}& 3(2+\kappa(M))L_g\norm{D}^2+3\sigma_{\max}^{\frac{1}{2}}(M)\|D\|^2\mybrace{\frac{2L_g p}{\sigma_{\min}^{\frac{1}{2}}(M)}+L_0}-\beta\frac{\sigma_{\min}(M)}{2}\norm{D}^2\\
     \le {}& \left((3\kappa(M)(3(p+1)L_g+\sigma_{\min}^{\frac{1}{2}}(M)L_0)-\beta\frac{\sigma_{\min}(M)}{2}\right)\|D\|^2,
  \end{aligned}\]
  where the inequality $(i)$ follows from Lemmas \ref{le:normg} and \ref{le:normj}. Specifically, according to Lemma \ref{le:normg}, we have 
  \[\|G(X)\|\le L_g\mybrace{\frac{3}{2}+\frac{1}{2}\sigmamax{X\tp MX}}\|X\|+L_0\le2L_gp\sigmamax{X}+L_0\le\frac{2L_gp}{\sigma_{\min}^{\frac{1}{2}}(M)}+L_0. \]
 Furthermore, $(ii)$ follows from
  $\sigma_{\min}(M)\norm{D}^2\le \sigma_{\min}(X\tp MX)=\operatorname{tr}(vv\tp X\tp M Xvv\tp)\le\sigma_{\max}(M)\norm{D}^2.$
Therefore, with the fact that $\beta\ge\tilde{\beta}$, we have
   \[\inner{D,\nabla^2 h(X)[D]}\le -\frac{\sigma_{\min}(M)}{4}\beta\norm{D}^2,\]
which implies  that $X$ is a $\frac{\sigma_{\min}(M)}{4}\beta$-strict saddle point of \ref{cdfcp} and thus completes the proof. 
\end{proof}

In the following proposition, we show that when $f$ is a quadratic function, then any infeasible first-order stationary point $X\in\Rnp$ of \ref{cdfcp} satisfies  $\sigma_{\min}(X) = 0$. 
\begin{proposition}
    Suppose the objective function $f(X) := \frac{1}{2} \mathrm{tr}\left( X\tp AX \right)$, where $A \in \mathbb{S}^{n}$, and $\beta\ge \tilde{\beta}$. Then for any first-order stationary point $X \in \Rnp$ of \ref{cdfcp}, the eigenvalue of $X\tp MX$ is either $0$ or $1$. 
\end{proposition}
\begin{proof}
    Suppose $\sigma$ is an eigenvalue of $X^\top MX$, and $v$ is a corresponding unit eigenvector.
    Then as $X$ is a first-order stationary of \ref{cdfcp}, we have
    \begin{equation}\label{pro3}
        \begin{aligned}
        0=\langle Xvv\tp,\nabla h(X)\rangle\overset{(i)}{=}(\sigma-1)\left(\beta\sigma(\sigma+1)-\frac{3}{2}v\tp X\tp AXv\left(\frac{3}{2}-\frac{1}{2}\sigma\right)\right),
    \end{aligned}
    \end{equation}
    where the equation $(i)$ follows from Lemma \ref{Le_inner_nablah_XQ}. Moreover, it follows from Proposition \ref{pro:sigma_up} that $0\le\sigma\le 1$. 

    Then whenever $0<\sigma< 1$, let $u: = \sigma^{-\frac{1}{2}}M^{\frac{1}{2}}Xv$. It is easy to verify that $\norm{u} = 1$. Then we have 
    \begin{equation*}
        \begin{aligned}
            0={}&\beta\sigma(1+\sigma)-\frac{3}{2}v\tp X\tp AXv\left(\frac{3}{2}-\frac{1}{2}\sigma\right)\\
            ={}& \beta\sigma(1+\sigma)-\frac{3}{2} \sigma \left(\frac{3}{2}-\frac{1}{2}\sigma\right) u\tp M^{-\frac{1}{2}} A M^{-\frac{1}{2}} u \\
            \overset{(i)}{\geq}{}& \beta\sigma(1+\sigma) - \frac{\tilde{\beta}}{12} \cdot \left( \frac{3}{2} \sigma \left(\frac{3}{2}-\frac{1}{2}\sigma\right)\right)\\
            \geq{}& \sigma \left( \beta(1+\sigma) - \frac{\tilde{\beta}}{4}  \right)>0,
        \end{aligned}
    \end{equation*}
   which leads to a contradiction. Here, the inequality $(i)$ holds since 
   $u\tp M^{-\frac{1}{2}} A M^{-\frac{1}{2}} u\le\frac{\sigmamax{A}}{\sigmamin{M}}$ and $\sigmamax{A}$ is a Lipschitz continuous constant  of $\nabla f(X)$.
   Therefore, we can conclude that $\sigma = 0$ or $1$. This completes the proof. 
\end{proof}

Then we present the following theorem to illustrate that any first-order stationary point of \ref{cdfcp} is either a first-order stationary point of \ref{sogse} or a strict saddle point of \ref{cdfcp}. These results directly follow from Proposition \ref{Prop_first_order_equivalence_feasible} and  \ref{pro:saddle}, hence their proof is omitted for simplicity. 

\begin{theorem}\label{thm:refir}
Suppose Assumptions \ref{as:blanket} and \ref{as:s} hold, and $\beta\ge\tilde{\beta}$. If $X \in \Rnp$ is a first-order stationary point of \ref{cdfcp}. Then either of the following facts holds
\begin{enumerate}[(1)]
    \item $X \in \Smnp$, hence is a first-order stationary point of \ref{sogse}.
    \item $X \notin \Smnp$, hence is a strict saddle point of \ref{cdfcp}.
\end{enumerate}
\end{theorem}

\subsection{Relationships on second-order stationary points}\label{subsec:so}

In this subsection, we establish the global equivalence between \ref{sogse} and \ref{cdfcp} with respect to the second-order stationary points. 

We begin our analysis with the following auxiliary lemma, which uncovers a crucial characteristic of $\nabla^2 h$. 
\begin{lemma}\label{le:sop_slep}
	Suppose Assumptions \ref{as:blanket} and \ref{as:s} hold. Then for any $X\in\Smnp$ that is a first-order stationary point of \ref{cdfcp}, $T\in\mathcal{T}_{\Smnp}(X)$ and $N\in\mathcal{N}_{\Smnp}(X)$, it holds that
	\begin{equation}\label{eq:tn}
	    \left\langle N,\nabla^2  h(X)[T]\right\rangle=0.
	\end{equation}
	Moreover, suppose $\beta\ge\tilde{\beta}$, then for any non-zero element $N\in\mathcal{N}_{\Smnp}(X)$, it holds that
	\begin{equation}\label{eq:nn}
	    \left\langle N,\nabla^2 h(X)[N] \right\rangle>0.
	\end{equation}
\end{lemma}
\begin{proof}
For any $X\in\Smnp$ that is a first-order stationary point of \ref{cdfcp}, it follows from Proposition \ref{Prop_first_order_equivalence_feasible} that 
\begin{equation}\label{eq:hesshpro1}
    \nabla f(X)-M X\Psi({X}^\top \nabla f(X))=0.
\end{equation}
 For any $T\in\mathcal{T}_{\Smnp}(X)$, by the definition of $\mathcal{T}_{\Smnp}(X)$ in \eqref{eq:tx}, we have $T\tp MX+X\tp MT=0$, which implies that $X\tp MT$ is skew-symmetric. Moreover, for any $N\in\mathcal{N}_{\Smnp}(X)$, by the definition of $\mathcal{N}_{\Smnp}(X)$ in \eqref{eq:nx}, there is a symmetric matrix $S\in\mathbb{S}^p$ such that $N=XS$. Then it follows from the expression of $\nabla^2 h(X)$ in \eqref{eq:hessh} that
 \[\begin{aligned}
		\left\langle N, \nabla^2 h(X)[T]\right\rangle
		={}&\left\langle N,\nabla^2 f(X)[T]-MX\Psi({X}^\top \nabla^2 f(X)[T])\right\rangle-\left\langle N, MT\Psi(\nabla f(X)^\top X)\right\rangle
	      \\&-\left\langle N, MX\Psi(\nabla f(X)^\top T)\right\rangle\\
		={}&\operatorname{tr}(S{X}^\top\nabla^2 f(X)[T])-\operatorname{tr}(S\Psi({X}^\top \nabla^2 f(X)[T]))-\operatorname{tr}(S{X}^\top M T\Psi({X}^\top \nabla f(X)))\\
&-\operatorname{tr}(S{X}^\top MX\Psi(T^\top \nabla f(X)))\\
		\overset{(i)}{=}{}&\operatorname{tr}(ST^\top M X\Psi({X}^\top \nabla f(X)))-\operatorname{tr}(S\Psi(T^\top \nabla f(X)))\\
		\overset{(ii)}{=}{}&\operatorname{tr}(ST^\top \nabla f(X))-\operatorname{tr}(ST^\top \nabla f(X))=0,
	\end{aligned}\]
 which proves the equality \eqref{eq:tn}. Here, the equation $(i)$ follows from the fact that $X\tp MT$ is skew-symmetric, and the equation $(ii)$ directly follows from \eqref{eq:hesshpro1}.

In particular, if $N=XS\neq0$,  which implies that $\|S\|>0$, then it follows again from the expression of $\nabla^2 h(X)$ in \eqref{eq:hessh} that
\[\begin{aligned}
	&\left\langle N,\nabla^2 h(X)[N]\right\rangle \\
 ={}&\left\langle N,\nabla^2 f(X)[N]-MX\Psi({X}^\top \nabla^2 f(X)[N])\right\rangle-\left\langle N,\nabla f(X)S\right\rangle\\
 &-\left\langle N,	MN\Psi({X}^\top \nabla f(X)) \right\rangle -\left\langle N, 	MX\Psi(N^\top\nabla f(X))\right\rangle+4\beta\left\langle N,MN \right\rangle\\
 ={}&\operatorname{tr}(S{X}^\top\nabla^2 f(X)[N])-\operatorname{tr}(S\Psi({X}^\top \nabla^2 f(X)[N]))
 -\operatorname{tr}\left( SX\tp\nabla f(X)S\right)\\
 &-\operatorname{tr}\left( S^2\Psi(\nabla f(X)^\top X)\right)-\operatorname{tr}\left( S\Psi(SX^\top\nabla f(X))\right)+4\beta\operatorname{tr}\left( S^2\right)\\
	={}&4\beta\norm{S}^2-3\operatorname{tr}(S^2{X}^\top\nabla f(X))
	={}4\beta\norm{S}^2-3\sigmamax{X^\top\nabla f(X)}\norm{S}^2\\
	\overset{(i)}{\ge}{}&\mybrace{4\beta-\frac{2pL_g+L_0\sigma_{\min}^{\frac{1}{2}}(M)}{\sigmamin{M}}}\norm{S}^2\overset{(ii)}{>}0,
\end{aligned}\]
where the inequality $(i)$ follows form the upper bound of $\sigmamax{X^\top\nabla f(X)}$ given in \eqref{eq:normlambda}, and the inequality $(ii)$ directly follows from the fact that $\beta\ge\tilde{\beta}>\frac{2pL_g+L_0\sigma_{\min}^{\frac{1}{2}}(M)}{4\sigmamin{M}}$. This completes the proof.
\end{proof}

The following theorem illustrates that \ref{cdfcp} and \ref{sogse} share the same second-order stationary points over $\Rnp$. 
\begin{theorem}\label{th:sopr_b1}
    Suppose Assumptions \ref{as:blanket} and \ref{as:s} hold, and $\beta\ge\tilde{\beta}$. Then \ref{cdfcp} and \ref{sogse} share the same second-order stationary points over $\Rnp$.
\end{theorem}
\begin{proof}
For any $X\in\Smnp$ and $T\in\mathcal{T}_{\Smnp}(X)$, by the definition of $\mathcal{T}_{\Smnp}(X)$ in \eqref{eq:tx}, we have $\Psi(T\tp MX)=0$. Then it follows from the expression of $\nabla^2 h(X)$ in \eqref{eq:hessh} that
\begin{equation}\label{eq:s1}
    \begin{aligned}
\inner{ T,\nabla^2 h(X)[T]}
={}&\inner{T,\nabla^2 f(X)[T]}-\inner{T,MX\Psi({X}^\top \nabla^2 f(X)[T])}\\
&-\inner{T,MT\Psi(X\tp\nabla f(X))}-\inner{T,MX\Psi(T\tp\nabla f(X))}\\
\overset{(i)}{=}{}&\inner{T,\nabla^2 f(X)[T]}-\left\langle T,M T{X}^\top\nabla f(X)\right\rangle,
\end{aligned}
\end{equation}
where the equation $(i)$ utilizes the fact that $\inner{T,MXS}=\operatorname{tr}(T\tp MXS)=\operatorname{tr}(\Psi(T\tp MX)S)=0$ for any $S\in\mathbb{S}^p.$

For any second-order stationary point $X\in\Rnp$ of \ref{cdfcp}, it follows from Definition \ref{def:occd} that $X$ is a first-order stationary of \ref{cdfcp} and $\left\langle D, \nabla^2 h(X)[D] \right\rangle\ge0$ for any $D\in\Rnp$. Additionally, by theorem \ref{thm:refir}, we can deduce that $X$ is a feasible point and thus a first-order stationary point of \ref{sogse}. Then for any $T\in\mathcal{T}_{\Smnp}(X)$, it follows from \eqref{eq:s1} that 
\[\left\langle T, \nabla^2 f(X)[T]
	\right\rangle-\left\langle T,M T{X}^\top\nabla f(X)\right\rangle=\left\langle T,\nabla^2 h(X)[T] \right\rangle\ge0.\]
Hence, according to Definition \ref{def:ocso}, we can conclude that $X$ is a second-order stationary point of \ref{sogse}.

 Furthermore, for any $X\in\Smnp$ that is a second-order stationary point of \ref{sogse}, it follows from Definition \ref{def:ocso} that $X$ is a first-order stationary of \ref{sogse} and $\left\langle T, \nabla^2 f(X)[T]
	\right\rangle-\left\langle T,M T{X}^\top\nabla f(X)\right\rangle \ge0$ for any $T\in\mathcal{T}_{\Smnp}(X)$. Additionally, by theorem \ref{Prop_first_order_equivalence_feasible}, we can deduce that $X$ is also a first-order stationary point of \ref{cdfcp}. Then for any $T\in\mathcal{T}_{\Smnp}(X)$, it follows from \eqref{eq:s1} that 
\[\left\langle T,\nabla^2 h(X)[T] \right\rangle=\left\langle T, \nabla^2 f(X)[T]
	\right\rangle-\left\langle T,M T{X}^\top\nabla f(X)\right\rangle\ge0.\]
 It is easy to verify that $\bb{R}^{n\times p}$ is the direct sum of the subspaces $\mathcal{T}_{\Smnp}(X)$ and $\mathcal{N}_{\Smnp}(X)$. That is to say, for any $D\in\mathbb{R}^{n\times p}$, there is an element $T$ from $\mathcal{T}_{\Smnp}(X)$ and an element $N$ from $\mathcal{N}_{\Smnp}(X)$ such that $D=T+N$. Then it follows from Lemma \ref{le:sop_slep} that
 \[\begin{aligned}
     \left\langle D, \nabla^2 h(X)[D] \right\rangle
     =\left\langle T,\nabla^2 h(X)[T] \right\rangle+2\left\langle N,\nabla^2 h(X)[T] \right\rangle+\left\langle N,\nabla^2 h(X)[N] \right\rangle
     \ge0,
 \end{aligned}\]
 Hence, by Definition \eqref{def:ocso}, $X$ is also a second-order stationary point of \ref{cdfcp}. This completes the proof.

\end{proof}

\subsection{Relationships on $\varepsilon$-first-order stationary points }\label{subsec:approxfo}

In this subsection, we show that \ref{sogse} and \ref{cdfcp} share the same $\mathcal{O}(\varepsilon)$-first-order stationary points in the neighborhood of $\Smnp$. 

To set the stage for our main results, we begin with a proposition that offers a simple analysis of the stationary of \ref{cdfcp}.
\begin{proposition}
     Suppose Assumption \ref{as:blanket} holds, and $\beta\ge \tilde{\beta}$. Then for any $X\in\Omega_{\frac{1}{6}}$, it holds that
     \begin{equation}\label{eq:feabound}
         \|\nabla h(X)\|\ge \frac{1}{2\kappa^\frac{1}{2}(M)}\left\|
   \nabla g(X)\right\|+\frac{\beta\sigma_{\min}^\frac{1}{2}(M)}{2}\left\|X^\top MX-I_p\right\|.
     \end{equation}
\end{proposition}

\begin{proof}
  
According to the definition of $g(X)$ in \eqref{eq:defgG}, we have $h(X)=g(X)+\frac{\beta}{6}\mathrm{tr}( X\tp M X((X\tp MX)^2 - 3I_p))$. Therefore, it is easy to verify that $\nabla h(X)=\nabla g(X)+\beta MX((X\tp MX)^2-I_p)$. Then we have
\begin{equation}\label{eq:sta1}
    \begin{aligned}
        \norm{\nabla h(X)}^2\ge{}&\sigma_{\min}(M)\operatorname{tr}\left(\nabla h(X)^\top M^{-1}\nabla h(X)\right)\\
        \ge{}& \frac{1}{\kappa(M)}\left\|\nabla g(X)\right\|^2+2\sigma_{\min}(M)\beta\left\langle X\left(\left(X\tp MX\right)^2-I_p\right),\nabla g(X)\right\rangle\\
    &+\sigma_{\min}(M)\beta^2\operatorname{tr}\left(X\tp MX\left(\left(X\tp MX\right)^2-I_p\right)^2\right).\\
    \end{aligned}
\end{equation}
It follows from  Lemma \ref{Le_inner_nablah_XQ} that for any $Q\in\mathbb{S}^p$, 
\begin{equation}\label{eq:sta2}
    \begin{aligned}
    \inner{XQ, \nabla g(X)} ={}&\inner{XQ, \nabla h(X)}-\inner{XQ,-\beta MX\left(\left(X\tp MX\right)^2-I_p\right)}\\
    = {}&- \frac{3}{2}\mathrm{tr}\left( \Lambda(X)(X\tp MX - I_p)Q  \right).
\end{aligned}
\end{equation}
Combining \eqref{eq:sta1} with \eqref{eq:sta2}, we have
 \[\begin{aligned}
 \norm{\nabla h(X)}^2
    \ge{}&  \frac{1}{\kappa(M)}\left\|\nabla g(X)\right\|^2-3\sigma_{\min}(M)\beta\operatorname{tr}\left(\Lambda(X)\left(X\tp MX+I_p\right)\left(X\tp MX-I_p\right)^2\right)\\
    &+\sigma_{\min}(M)\beta^2\operatorname{tr}\left(X\tp MX\left(\left(X\tp MX\right)^2-I_p\right)^2\right)\\
     \overset{(i)}{\ge} {}&  \frac{1}{\kappa(M)}\left\|\nabla g(X)\right\|^2-3\sigma_{\min}(M)\beta\sigmamax{\Lambda(X)(X\tp MX+I_p)}\|X\tp MX-I_p\|^2\\
    &+\sigma_{\min}(M)\beta^2\sigmamin{X\tp MX\left(X\tp MX+I_p\right)^2}\|X\tp MX-I_p\|^2\\
    \overset{(ii)}{\ge} {}&  \frac{1}{\kappa(M)}\left\|\nabla g(X)\right\|^2+\sigma_{\min}(M)\beta(\frac{14}{5}\beta-\frac{7}{2}\sigmamax{X\tp G(X)})\left\|X\tp MX-I_p\right\|^2\\
    \overset{(iii)}{\ge}   {}& \frac{1}{\kappa(M)}\left\|\nabla g(X)\right\|^2+\beta^2\sigma_{\min}(M)\left\|X^\top MX-I_p\right\|^2
    \end{aligned}
    \]
where the inequality $(i)$ utilizes the fact that $\tr(AB^2)\le\sigmamax{A}\norm{B}^2$ holds  for any $A\in\mathbb{R}^{p\times p}$, $B\in\mathbb{S}^p$, the inequality $(ii)$ uses the fact that $\frac{5}{6}\le\sigmamin{X\tp MX}\le\sigmamax{X\tp MX}\le\frac{7}{6}$ holds for any $X\in\Omega_{\frac{1}{6}}$, and the inequality $(iii)$ follows from the upper bound of $\sigmamax{X\tp G(X)}$ given in \eqref{eq:normlambda}.
Therefore, we can finally conclude that 
\[\|\nabla h(X)\| \ge \frac{1}{2\kappa^\frac{1}{2}(M)}\left\|
    \nabla g(X)\right\|+\frac{\beta\sigma_{\min}^\frac{1}{2}(M)}{2}\left\|X^\top MX-I_p\right\|,\]
which completes the proof. 
\end{proof}

The following theorem exhibits the equivalence between the stationary of \ref{sogse} and that of \ref{cdfcp}.
\begin{theorem}\label{th:star_bx}
    Suppose Assumption \ref{as:blanket} holds, and $\beta\ge \tilde{\beta}$. Then for any $X\in\Omega_{\frac{1}{6}}$, it holds that 
    \[\norm{L(X)}+7\beta\sigma_{\max}^\frac{1}{2}(M)\norm{X^\top MX-I_p}\ge\|\nabla h(X)\|\ge \frac{1}{2\kappa^{\frac{1}{2}}(M)}\norm{L (X)}+\frac{\sigma^{\frac{1}{2}}_{\min}(M)\beta}{4}\norm{X^\top MX-I_p}.\]
\end{theorem}
\begin{proof}
According to the expression of  $L(X)$ in \eqref{eq:regrad}, for any $X\in\Omega_{\frac{1}{6}}$ we have
\[\begin{aligned}
    &\norm{\nabla g(X)-L(X)}\\    
    \overset{}{=}{}&\left\|\nabla f(\A(X))-f(X)+\frac{1}{2}\nabla f(\A(X))(X^\top MX-I_p)-(MX\Psi\mybrace{X\tp(\nabla f(\A(X))-\nabla f(X))}\right\|\\
    \overset{(i)}{\le}{}& L_g\norm{\A(X)-X}+\frac{1}{2}\norm{\nabla f(\A(X))}\norm{X^\top MX-I_p}+\sigmamax{M}\sigma_{\max}^2(X)L_g\norm{\A(X)-X}\\
    \overset{(ii)}{\le}{}& L_g\frac{\sigmamax{X}}{2}\norm{X^\top MX-I_p}+\frac{1}{2}(L_g\frac{13p\sigmamax{X}}{12}+L_0)\norm{X^\top MX-I_p}\\
    &+\sigmamax{M}\sigma_{\max}^2(X)L_g\frac{\sigmamax{X}}{2}\norm{X^\top MX-I_p}\\
    \overset{(iii)}{\le}{}& \frac{\kappa(M)(3(p+1)L_g+\sigma^{\frac{1}{2}}_{\min}(M)L_0)}{2\sigma^{\frac{1}{2}}_{\min}(M)}\le\frac{\beta\sigma_{\min}^{\frac{1}{2}}(M)}{2}\norm{X^\top MX-I_p}
\end{aligned}\]
 where the inequality $(i)$ follows from the $L_g$-Lipschitz continuity of $\nabla f(X)$. Moreover, the inequality $(ii)$ holds since 
\[\norm{\A(X)-X}=\norm{X\mybrace{\frac{3}{2}I_p-\frac{1}{2}X\tp MX}-X}\le\frac{\sigmamax{X}}{2}\norm{X\tp M X-I_p}\le\frac{\sigmamax{X}}{12},\]
and 
$\norm{\nabla f(\A(X))}\le L_g\norm{\A(X)}+L_0\le L_g(\norm{\A(X)-X}+\norm{X})+L_0\le L_g\frac{13p\sigmamax{X}}{12}+L_0$. In addition, the inequality $(iii)$ follows form $\sigma_{\max}^2(X)\le\frac{\sigmamax{X\tp MX}}{\sigmamin{M}}\le\frac{7}{6\sigmamin{M}}$.

 Then by \eqref{eq:feabound}, we obtain
\[\begin{aligned}
    \|\nabla h(X)\|{}&\ge\frac{1}{2\kappa^{\frac{1}{2}}(M)}\norm{\nabla g(X)}+\frac{\sigma_{\min}^{\frac{1}{2}}(M)\beta}{2}\norm{X^\top MX-I_p}\\
    &\ge\frac{1}{2\kappa^{\frac{1}{2}}(M)}(\norm{L(X)}-\norm{L(X)-\nabla g (X)})+\frac{\sigma^{\frac{1}{2}}_{\min}(M)\beta}{2}\norm{X^\top MX-I_p} \\
  & \ge\frac{1}{2\kappa^{\frac{1}{2}}(M)}\norm{L(X)}-\frac{\beta\sigma^{\frac{1}{2}}_{\min}(M)}{4\kappa^{\frac{1}{2}}(M)}\norm{X^\top MX-I_p}+\frac{\sigma^{\frac{1}{2}}_{\min}(M)\beta}{2}\norm{X^\top MX-I_p} \\
  &\ge\frac{1}{2\kappa^{\frac{1}{2}}(M)}\norm{L(X)}+\frac{\sigma^{\frac{1}{2}}_{\min}(M)\beta}{4}\norm{X^\top MX-I_p}.
\end{aligned}\]
On the other hand, we have 
\[\begin{aligned}
   &\|\nabla h(X)\|\le\norm{\nabla g(X)}+\beta\mybrace{\tr\mybrace{\left(\left(X\tp MX\right)^2-I_p\right)X^\top M^2 X\left(\left(X\tp MX\right)^2-I_p\right)}}^\frac{1}{2}\\
    \le&\norm{L(X)}+\norm{\nabla g(X)-L (X)}
    +\beta\sigma_{\max}^\frac{1}{2}(M)\sigma_{\max}^\frac{1}{2}(X\tp MX)\sigmamax{X\tp MX+I_p}\norm{X^\top MX-I_p}\\
    \le& \norm{L(X)}+7\beta\sigma_{\max}^\frac{1}{2}(M)\norm{X^\top MX-I_p}.
\end{aligned}\]
The proof is completed.
\end{proof}
The following results are direct corollaries of the aforementioned theorem, indicating that \ref{sogse} and \ref{cdfcp} share the same $\mathcal{O}(\varepsilon)$-first-order stationary points in $\Omega_{\frac{1}{6}}$, thus we omit the proof here.
\begin{corollary}
    Suppose Assumption \ref{as:blanket} holds, and $\beta\ge \tilde{\beta}$.  If  $X\in\Omega_{\frac{1}{6}}$ is an $\varepsilon$-first-order stationary point of \ref{cdfcp}, then it is a $\frac{4+2\sigma_{\max}^{\frac{1}{2}}(M)\beta}{\sigma_{\min}^{\frac{1}{2}}(M)\beta}\varepsilon$-first-order stationary point of \ref{sogse}. 
    
    Moreover, when $X\in\Omega_{\frac{1}{6}}$ is an  $\varepsilon$-first-orderstationary point of \ref{sogse}, then it is a $(1+7\beta\sigma_{\max}^{\frac{1}{2}}(M))\varepsilon$-first-order stationary point of \ref{cdfcp}.
\end{corollary}

\section{Stochastic optimization algorithm}\label{algorithm}

In this section, we develop two stochastic optimization algorithms for solving \ref{sogse} based on our proposed penalty function \ref{cdfcp} and establish their global convergence properties.
According to the formulation of $h$ in \ref{cdfcp}, the minimization of \ref{cdfcp} over $\Rnp$ is a nested stochastic optimization problem 
\cite{wang2017stochastic,chen2021solving,wang2017accelerating,ghadimi2020single,zhang2019stochastic}. As a result, the stochastic optimization approaches from these studies can be directly applied to solve \ref{sogse} by minimizing \ref{cdfcp}. However, this direct application requires tracking of $n$-by-$n$ matrices, resulting in high storage demands for large $n$. Towards this issue, our proposed algorithms are derived through an insightful combination of the tracking technique proposed in \cite{chen2021solving} with \ref{cdfcp}, which only need to track a $p$-by-$p$
matrix and thus enjoy lower memory cost compared to existing approaches.

\subsection{Stochastic gradient method}

In this subsection, we aim to develop a stochastic gradient method for solving \ref{sogse}. 

To begin with, we define the following auxiliary functions,
\begin{empheq}[left=\empheqlbrace]{align}
&C(X) := \mathbb{E}[C_{\theta}(X)],&&\text{where } C_{\theta}(X) := X^\top M_{\theta}\,X,\label{eq:c}\\
&H(U,V) := \mathbb{E}[H_{\xi}(U,V)],&&\text{where } H_{\xi}(U,V) := f_\xi\!\Bigl(
    U \bigl(\tfrac{3}{2}I_p - \tfrac{1}{2}V\bigr)
\Bigr)
+ \frac{\beta}{6}\,\operatorname{tr}\!\bigl(V(V^2 - I_p)\bigr).\label{eq:H}
\end{empheq}
Based on these auxiliary functions , we can reformulate the objective function $h$ in \ref{cdfcp} as
\begin{equation}\label{eq:nest}
    h(X)=H(X,C(X))=\bb{E}_{\xi}[H_{\xi} ( X,\bb{E}_{\theta}[C_{\theta}(X)]) ].
\end{equation}
Moreover, from the reformulation of $h$ in \eqref{eq:nest}, the gradient of $h(X)$ can be expressed as 
\[\begin{aligned}
    \nabla h(X)&=\nabla_U H(X,C(X))+\nabla C(X)\nabla_V H(X,C(X))\\     &= \mathbb{E}_{\xi,\theta} \left[\nabla_U H_{\xi}(X,\mathbb{E}_{\theta}[C_{\theta}(X)])+ \nabla C_{\theta}(X) \nabla_V H_{\xi}(X,\mathbb{E}_{\theta}[C_{\theta}(X)])\right].
 \end{aligned}\]
 Due to the compositional structure, constructing an unbiased estimator of $\nabla h(X)$ necessitates noiseless estimation of $C(X)$, which is intractable in practice. Towards this issue, inspired by \cite{chen2021solving}, we introduce a sequence of auxiliary variables $\{Y_k\}_{k\ge 0}$ to track the values of $\{C(X_k)\}_{k\ge 0}$, which can be formulated recursively by the following equation,
 \begin{equation}\label{eq:updatey}
    Y_{k+1}=Y_{k}-b_k(Y_k-C_{\theta_{k+1}}(X_k))+(C_{\theta_{k+1}}(X_{k+1})-C_{\theta_{k+1}}(X_k)),
\end{equation}
where $\theta_{k+1}$ is randomly and independently drawn from the distribution $\Theta$, and $b_k$ refers to the step size for tracking.

 In the following, we introduce an auxiliary function $W_{\xi,\theta}$ defined as follows, 
 \[W_{\xi,\theta}(X,Y):=\nabla_U H_{\xi}(X,Y)+\nabla C_{\theta}(X)\nabla_V H_{\xi}(X,Y),\]
 and let $W(X,Y):=\mathbb{E}[W_{\xi,\theta}(X,Y)]$ be the expectation of $W_{\xi,\theta}(X,Y)$ with respect to $\xi$ and $\theta$. 
 
 From the formulation of $W$, it is easy to verify that 
 \[\nabla h(X_{k})=W(X,C(X_{k}))=\mathbb{E}[W_{\xi,\theta}(X_{k},C(X_{k}))].\]
  Hence, we introduce the following sequence $\{D_k\}_{k\ge 0}$ to track
 $\{\nabla h(X_k)\}_{k\ge 0}$,
    \begin{equation}\label{eq:updated}
 D_{k}:=W_{\xi_{k},\theta_{k}}(X_{k},Y_{k}).
 \end{equation}
 Here $\xi_{k}$ and $\theta_{k}$ are randomly and independently drawn from the distributions $\Xi$ and $\Theta$, respectively. 
 
 The detailed algorithm is presented in Algorithm  \ref{cdfsg}.
 
\begin{algorithm}[H]\label{cdfsg}
\caption{Stochastic gradient method for solving \ref{sogse}.}
\KwData{Initialization $X_0,D_0\in\Rnp,Y_0\in\bb{R}^{p\times p}$, step size sequence $\{\alpha_k\}_{k\ge 0},\{b_k\}_{k\ge 0}$, initial update direction $D_0 = 0$.}

\For{$k=0,1,2,\dots,K-1$}
 {Update the main iterate $X_{k+1}$ by 
    \begin{equation}\label{eq:updatexoa}
        X_{k+1}=X_k-\alpha_kD_{k}.
    \end{equation}

Randomly  and independently draw $\theta_{k+1}$ from $\Theta$, update the tracking iterate $Y_{k+1}$ by \eqref{eq:updatey}.

   Randomly  and independently draw $\xi_{k+1}$ from $\Xi$, update the direction $D_{k+1}$ according to \eqref{eq:updated}.
 }
 \KwResult{$X_{K}$.}
\end{algorithm}

As a stochastic optimization method, Algorithm \ref{cdfsg} only utilizes a small batch of samples in each iteration.
Additionally, instead of storing an $n$-by-$n$ matrix to track the expectation of $M_{\theta}$, we
only needs to store a $p$-by-$p$ matrix to track the values of $\{C(X_k)\}_{k\ge 0}$ in Algorithm \ref{cdfsg}, which results in significantly lower storage costs.
In addition, Algorithm \ref{cdfsg} only requires matrix-matrix multiplications, hence avoiding the bottlenecks of factorizing matrices in existing Riemannian optimization approaches (see, e.g., \cite{absil_optimization_2008}). 
 We present a detailed summary of the storage requirements and computational costs for our algorithm in Table \ref{ta:computation}.

\begin{table}[htbp]
	\centering 
	\begin{tabular}{c|c|c}
 \hline\hline
 Key update& Computational cost& Memory cost\\
		\hline
		$X_{k+1}$ \eqref{eq:updatex}&$\mathcal{O}(np)$&$np$\\
  \hline
		$Y_{k+1}$ \eqref{eq:updatey}&$2O_C+\mathcal{O}(np)$&$p^2$\\
  \hline
		$D_{k+1} \eqref{eq:updated}$&$O_{\nabla f}+O_{\nabla C}+10np^2+p^3+\mathcal{O}(np)$&$np$\\
  \hline
		In total&$O_{\nabla f}+O_{\nabla C}+2O_C+10np^2+p^3+\mathcal{O}(np)$&$2np+p^2$\\
  \hline\hline
	\end{tabular}
\caption{Computational and memory complexity of key updates in Algorithm \ref{cdfsg}}
\label{ta:computation}
\end{table}
Here, $O_{\nabla f}$ represents the computational cost of $\nabla f_{\xi}(X)$. Similarly, $O_{C}$ and $O_{\nabla C}$ represent the computational costs of $C_{\theta}(X)$ and $\nabla C_{\theta}(X)$, respectively.

To establish the convergence properties of Algorithm \ref{cdfsg}, we first make the following assumptions on Algorithm \ref{cdfsg}. 
\begin{assumption}\label{as:algo}
\begin{enumerate}[(1)]
    \item The sampling oracle satisfies $i)$ $\mathbb{E}[M_{\theta_k}]=M$ and $ii)$ $\mathbb{E}[W_{\xi_k,\theta_k}(X,Y)]=W(X,Y)$. 
    \item The matrices $\{M_{\theta_k}\}_{k\ge 0}$ are uniformly bounded almost surely; i.e., there exists a constant $\tau_m>0$ such that $\sup_{k\geq 0} \|M_{\theta_k}\|\le \tau_m$ holds almost surely.
        \item The sequence $\{X_k\}_{k\ge 0}$ is bounded almost surely; i.e., there exists a constant $\tau_x>0$ such that $\sup_{k\geq 0} \norm{X_k} \leq \tau_x$  holds  almost surely.
          \item The sequence $\{D_k\}_{k\ge 0}$ is bounded almost surely; i.e., there exists a constant $\tau_d>0$ such that $\sup_{k\geq 0} \norm{D_k} \leq \tau_d$  holds  almost surely.
            \item The step sizes $\{\alpha_k\}_{k\ge 0}$ and $\{b_k\}_{k\ge 0}$ satisfy $a:=\sup_{k\ge 0}\frac{\alpha_k}{b_k}<\infty$.
\end{enumerate}
\end{assumption} 
\begin{remark}
     In Assumption \ref{as:algo} $(1)$, we adopt the standard unbiasedness assumption of the sampling oracle. Assumption \ref{as:algo} $(2)$ is analogous to the standard uniformly Lipschitz smoothness assumption applied to the constraint functions in nested stochastic optimization (see, e.g., \cite{wang2017stochastic, chen2021solving}). Assumption \ref{as:algo} $(3)$ and $(4)$, which suppose that the generated sequences $\{X_k\}_{k\ge 0}$ and $\{D_k\}_{k\ge 0}$ are uniformly bounded almost surely, are common in the context of stochastic optimization for nonconvex settings (see, e.g., \cite{benaim2005stochastic,le2024nonsmooth,bolte2023subgradient}).  Assumption \ref{as:algo} $(5)$ illustrates great flexibility in choosing the step sizes. 
\end{remark}
\begin{remark}
It follows from Assumption \ref{as:algo} $(2)$ that $\mathbb{E}[\|M_{\theta_k}-M\|^2]\le 2(\tau_m^2+\|M\|^2)$. In the other words, $M_{\theta_k}$ has bounded variance. For brevity, we denote $v_m^2:=\sup_{k\ge0} \mathbb{E}[\|M_{\theta_k}-M\|^2]$.
\end{remark}
The subsequent proposition asserts that under the above assumptions, the sequence $\{Y_k\}_{k\ge 0}$ generated by \eqref{eq:updatey} is almost surely bounded.

\begin{proposition}\label{pro:boundy}
   Suppose Assumptions \ref{as:blanket} and \ref{as:algo} hold, and let $\tau_y:=\tau_x\tau_m(\tau_x+2a\tau_d)$. Then $\sup_{k\geq 0} \norm{Y_k} \leq \tau_y$ holds almost surely. 
\end{proposition}
\begin{proof}
From the update for $Y_{k+1}$ in \eqref{eq:updatey}, we have
\begin{equation}\label{eq:upy1}
     \|Y_{k+1}\|\le(1-b_k)\|Y_k\|+b_k\| C_{\theta_{k+1}}(X_k)\|+\|C_{\theta_{k+1}}(X_{k+1})-C_{\theta_{k+1}}(X_k)\|.
\end{equation}

It follows from the expression of $C_{\theta}(X)$ in \eqref{eq:c} that 
\begin{equation}\label{eq:upy2}
\begin{aligned}
        &\|C_{\theta_{k+1}}(X_{k+1})-C_{\theta_{k+1}}(X_k)\|\\
        \le{}
        &\|C_{\theta_{k+1}}(X_{k+1})-X_{k+1}^\top M_{\theta_{k+1}}X_{k}\|+\|X_{k+1}^\top M_{\theta_{k+1}}X_{k}-C_{\theta_{k+1}}(X_k)\|\\
        \le{}&\|M_{\theta_{k+1}}\|\|X_{k+1}-X_k\|(\|X_{k+1}\|+\|X_k\|)\le 2\tau_x\tau_m\|X_{k+1}-X_{k}\|.
\end{aligned}
\end{equation}

Combining \eqref{eq:upy1} with \eqref{eq:upy2}, we have
    \[\begin{aligned}
        \|Y_{k+1}\|{}&\le(1-b_k)\|Y_k\|+b_k\| C_{\theta_{k+1}}(X_k)\|+ 2\tau_x\tau_m\|X_{k+1}-X_{k}\|\\
        &\overset{(i)}{\le}(1-b_k)\|Y_k\|+b_k\|M_{\theta_{k+1}}\|\|X_k\|^2+ 2\alpha_k\tau_x\tau_m\|D_k\|\\
        &\le(1-b_k)\|Y_k\|+b_k\tau_x^2\tau_m+2\alpha_k\tau_x\tau_m\tau_d\le (1-b_k)\|Y_k\|+b_k\tau_y,
    \end{aligned}\]
    where the inequality $(i)$ directly follows from the update for $X_{k+1}$ in \eqref{eq:updatex}.
     Selecting $Y_0$ such that $\|Y_0\|\le \tau_y$, for instance, setting $Y_0:=X_0^\top M_{\theta_0}X_0$, then by induction, we have $\|Y_{k+1}\|\le \tau_y$. The proof is completed.
\end{proof}

In the rest of this section, we set $\mathcal{F}_k$ be the $\sigma$-algebra generated by $\{X_0,\dots,X_{k+1},Y_0,\dots,Y_k,\\\xi_0,\dots,\xi_{k},\theta_0,\dots,\theta_{k}\}$, and define the following two constants for our analysis,
\[\begin{aligned}
    L_h{}:=\sup_{\|X_1\|,\|X_2\|\le \tau_x}\frac{\nabla h(X_1)-\nabla h(X_2)}{\|X_1-X_2\|}, \;\;\;
    L_W{}:=\sup_{\|X\|\le \tau_x,\|Y_1\|,\|Y_2\|\le \tau_y} \frac{W(X,Y_1)-W(X,Y_2)}{\|Y_1-Y_2\|}.
\end{aligned}\]

The following proposition shows the performance of $\{Y_{k}\}_{k\ge 0}$ for tracking the inner function values $\{C(X_k)\}_{k\ge 0}$.
\begin{proposition}\label{pr:yktrack}
Suppose Assumptions \ref{as:blanket} and \ref{as:algo} hold, then sequence $\{Y_k\}_{k\ge 0}$ generated in Algorithm \ref{cdfsg} satisfies 
\[\mathbb{E}\left[\left.\|C(X_{k+1})-Y_{k+1}\|^2\right|\mathcal{F}_k\right]\le(1-b_k)^2\|C(X_{k})-Y_k\|^2+2b_k^2\tau_x^4v_m^2+4\tau_x^2\tau_m^2\|X_{k+1}-X_{k}\|^2.\]    
\end{proposition}
\begin{proof} 
From the update \eqref{eq:updatey}, we have
    \[\begin{aligned}
    C(X_{k+1})-Y_{k+1}={}&(1-b_k)(C(X_{k})-Y_k)+b_k\underbrace{(C(X_k)-C_{\theta_{k+1}}(X_k))}_{=:I_{1,k}}\\
    &-\underbrace{(C_{\theta_{k+1}}(X_{k+1})
    -C_{\theta_{k+1}}(X_k))}_{=:I_{2,k}}+\underbrace{(C(X_{k+1})-C(X_k))}_{=:I_{3,k}}.
\end{aligned}\] 
Therefore, taking the expectation of both sides of the above equation over $\theta_{k+1}$ conditioned on $\mathcal{F}_k$, we have
\[\begin{aligned}
    &\mathbb{E}\left[\left.\|C(X_{k+1})-Y_{k+1}\|^2\right|\mathcal{F}_k\right]\\
    ={}&(1-b_k)^2\|C(X_{k})-Y_k\|^2+\mathbb{E}\left[\left.\|b_kI_{1,k}-I_{2,k}+I_{3,k}\|^2\right|\mathcal{F}_k\right]\\
    &+2(1-b_k)\inner{C(X_{k})-Y_k,\mathbb{E}\left[\left.b_kI_{1,k}-I_{2,k}+I_{3,k}\right|\mathcal{F}_k\right]}\\
    \overset{(i)}{=}{}&(1-b_k)^2\|C(X_{k})-Y_k\|^2+b_k^2\mathbb{E}\left[\left.\|I_{1,k}\|^2\right|\mathcal{F}_k\right]+ \mathbb{E}\left[\left.\|I_{2,k}\|^2\right|\mathcal{F}_k\right]+\|I_{3,k}\|^2-2b_k\mathbb{E}\left[\left.\inner{I_{1,k},I_{2,k}}\right|\mathcal{F}_k\right]\\
&+2b_k\inner{\mathbb{E}\left[\left.I_{1,k}\right|\mathcal{F}_k\right],I_{3,k}}-2\inner{\mathbb{E}\left[\left.I_{1,k}\right|\mathcal{F}_k\right],I_{3,k}}\\
    \overset{(ii)}{=}{}&(1-b_k)^2\|C(X_{k})-Y_k\|^2+b_k^2\mathbb{E}\left[\left.\|I_{1,k}\|^2\right|\mathcal{F}_k\right]+ \mathbb{E}\left[\left.\|I_{2,k}\|^2\right|\mathcal{F}_k\right]-\|I_{3,k}\|^2-2b_k\mathbb{E}\left[\left.\inner{I_{1,k},I_{2,k}}\right|\mathcal{F}_k\right]\\
    \overset{(iii)}{\le}{}&(1-b_k)^2\|C(X_{k})-Y_k\|^2+2b_k^2 \mathbb{E}\left[\left.\|I_{1,k}\|^2\right|\mathcal{F}_k\right]+2\mathbb{E}\left[\left.\|I_{2,k}\|^2\right|\mathcal{F}_k\right]\\
   \le{}&(1-b_k)^2\|C(X_{k})-Y_k\|^2+2b_k^2\|X_k\|^2\mathbb{E}\left[\left.\|M-M_{\theta_{k+1}}\|^2\right|\mathcal{F}_k\right]+2\mathbb{E}\left[\left.\|I_{2,k}\|^2\right|\mathcal{F}_k\right]\\
    \overset{(iv)}{\le}{}&(1-b_k)^2\|C(X_{k})-Y_k\|^2+2b_k^2\tau_x^4v_m^2+4\tau_x^2\tau_m^2\|X_{k+1}-X_{k}\|^2,
\end{aligned}\]
where the equations $(i)$ and $(ii)$ use the facts that $\mathbb{E}[I_{1,k}|\mathcal{F}_k]=0$ and $\mathbb{E}[I_{2,k}|\mathcal{F}_k]=\mathbb{E}[I_{3,k}|\mathcal{F}_k]=0$, the inequality $(iii)$ follows from the Cauchy-Schwarz inequality
$-2b_k\inner{I_{2,k},I_{3,k}}\le b_k^2\|I_{1,k}\|^2+\|I_{2,k}\|^2$, and the inequality $(iv)$ leaverages the result in \eqref{eq:upy2}. The proof is completed. 
\end{proof}

Then we are ready to prove the results on the sample complexity of Algorithm \ref{cdfsg}.
\begin{theorem}
    Suppose Assumptions \ref{as:blanket} and \ref{as:algo} hold, and the step sizes are chosen as $\alpha_k=s_1b_k=\frac{s_2}{\sqrt{K}}$, where $2\ge s_1>0$ and $s_2>0$. Then the sequence $\{X_k\}_{k\ge 0}$ generated by Algorithm \ref{cdfsg} satisfies
    \[\frac{1}{K}\sum_{k=0}^{K-1}\mathbb{E}\left[\norm{\nabla h(X_{k})}^2\right]\le \frac{2}{\sqrt{K}}\left(s_2^{-1}J_0+\frac{s_2}{2s_1^2}L_W^2\tau_x^4v_m^2+\left(4L_W^2\tau_x^2\tau_m^2+\frac{L_h}{2}\right)s_2\tau_d^2\right),\]
              where  $J_0:=h(X_0)-h(X_*)+L_W^2\|C(X_{0})-Y_0\|^2$  and $X_*\in\arg\min_{X\in\mathbb{R}^{n\times p}} h(X)$.
\end{theorem}
\begin{proof}
    We first define the auxiliary variables $\{J_k\}$ as follows,
    \begin{equation}\label{eq:jkdefi}
        J_k:=h(X_k)-h(X_*)+L_W^2\|C(X_{k})-Y_k\|^2.
    \end{equation}
     Then by the update scheme in Algorithm \ref{cdfsg}, we have, 
    \begin{equation}\label{eq:divi}
        \begin{aligned}
    J_{k+1}-J_k{}&=h(X_{k+1})-h(X_k)+L_W^2\|C(X_{k+1})-Y_{k+1}\|^2-L_W^2\|C(X_{k})-Y_k\|^2\\
    &\le-\alpha_k\inner{\nabla h(X_{k}),D_{k}}+\frac{L_h}{2}\|X_{k+1}-X_k\|^2+L_W^2\|C(X_{k+1})-Y_{k+1}\|^2-L_W^2\|C(X_{k})-Y_k\|^2.
\end{aligned}
    \end{equation}
    Let $\bar{D}_{k}:=W_{\xi_k,\theta_k}(X_k,C(X_{k}))$. Taking the expectation over $\xi_{k}$ and $\theta_{k}$ of the term  $-\alpha_k\langle\nabla h(X_{k}),D_{k}\rangle$ conditioned on $\mathcal{F}_{k-1}$, we have
    \[\begin{aligned}
       \mathbb{E}\left[\left.-\alpha_k\inner{\nabla h(X_{k}),D_{k}}\right|\mathcal{F}_{k-1}\right]
       ={}&\mathbb{E}\left[\left.-\alpha_k\inner{\nabla h(X_{k}),\bar{D}_{k}}-\alpha_k\inner{\nabla h(X_{k}),D_{k}-\bar{D}_{k}}\right|\mathcal{F}_{k-1}\right]\\
       ={}&-\alpha_k\norm{\nabla h(X_{k})}^2-\alpha_k\inner{\nabla h(X_{k}),W(X_k,Y_k)-W(X_k,C(X_{k}))}\\
       \le{}&-\alpha_k\norm{\nabla h(X_{k})}^2+\alpha_kL_W\norm{\nabla h(X_{k})}\|Y_k-C(X_{k})\|\\
       \overset{(i)}{\le}{}&-\alpha_k\mybrace{1-\frac{\alpha_k}{4b_k}}\norm{\nabla h(X_{k})}^2+b_kL_W^2\|Y_k-C(X_{k})\|^2\\
       \le{}&-\alpha_k\mybrace{1-\frac{a}{4}}\norm{\nabla h(X_{k})}^2+b_kL_W^2\|Y_k-C(X_{k})\|^2,
    \end{aligned}\]
    where the inequality $(i)$ is due to Young's inequality $t_1t_2\le\frac{t_1^2}{4b_k}+b_kt_2^2$ with $t_1:=\alpha_k\|\nabla h(X_k)\|$ and $t_2:=L_W\|Y_k-C(X_k)\|$.
    
    Consequently, take expectation over $\xi_{k+1}$ and $\theta_{k+1}$ on both sides of \eqref{eq:divi} conditioned on $\mathcal{F}_k$, then we have
    \[\begin{aligned}
       \mathbb{E}\left [\left.J_{k+1}\right|\mathcal{F}_k\right]-J_k={}&\mathbb{E} \left[\left.\mathbb{E} \left[\left.-\alpha_k\inner{\nabla h(X_{k}),D_{k}}\right|\mathcal{F}_{k-1}\right]\right|\mathcal{F}_k\right]+\frac{L_h}{2}\|X_{k+1}-X_k\|^2\\
       &+L_W^2\mathbb{E} \left[\left.\|C(X_{k+1})-Y_{k+1}\|^2\right|\mathcal{F}_k\right]-L_W^2\|C(X_{k})-Y_k\|^2\\
       \le{}&-\alpha_k\mybrace{1-\frac{a}{4}}\norm{\nabla h(X_{k})}^2+b_kL_W^2\|Y_k-C(X_{k})\|^2+\frac{L_h}{2}\|X_{k+1}-X_k\|^2\\
        &+L_W^2\|C(X_{k+1})-Y_{k+1}\|^2-L_W^2\|C(X_{k})-Y_k\|^2\\
        \overset{(i)}{\le}{}& -\alpha_k\mybrace{1-\frac{a}{4}}\norm{\nabla h(X_{k})}^2+L_W^2((1-b_k)^2-(1-b_k))\|C(X_{k})-Y_k\|^2\\
        &+2b_k^2L_W^2\tau_x^4v_m^2+\alpha^2_k\left(4L_W^2\tau_x^2\tau_m^2+\frac{L_h}{2}\right)\tau_d^2\\
        \le{}& -\alpha_k\mybrace{1-\frac{a}{4}}\norm{\nabla h(X_{k})}^2+2b_k^2L_W^2\tau_x^4v_m^2+\alpha^2_k\left(4L_W^2\tau_x^2\tau_m^2+\frac{L_h}{2}\right)\tau_d^2,
    \end{aligned}\]
    where the inequality $(i)$ directly follows from Proposition \ref{pr:yktrack}.
    Rearranging the terms above and telescoping from $k=0,\dots,K-1$, then we have
    \[\begin{aligned}
         \sum_{k=0}^{K-1}\alpha_k\mybrace{1-\frac{a}{4}}\mathbb{E}\left[\norm{\nabla h(X_{k})}^2\right]{}&\le J_0-\mathbb{E}[J_K]+\sum_{k=0}^K2b_k^2L_W^2\tau_x^4v_m^2+\sum_{k=0}^K\alpha^2_k\left(4L_W^2\tau_x^2\tau_m^2+\frac{L_h}{2}\right)\tau_d^2\\
         &\le J_0+\sum_{k=0}^K2b_k^2L_W^2\tau_x^4v_m^2+\sum_{k=0}^K\alpha^2_k\left(4L_W^2\tau_x^2\tau_m^2+\frac{L_h}{2}\right)\tau_d^2.
    \end{aligned}\]
    Selecting $\alpha_k=s_1b_k=\frac{s_2}{\sqrt{K}}$ with $2\ge s_1>0$ and $s_2>0$, we can finally obtain 
    \[\frac{1}{K}\sum_{k=0}^{K-1}\mathbb{E}\left[\norm{\nabla h(X_{k})}^2\right]\le \frac{2}{\sqrt{K}}\left(s_2^{-1}J_0+\frac{s_2}{2s_1^2}L_W^2\tau_x^4v_m^2+\left(4L_W^2\tau_x^2\tau_m^2+\frac{L_h}{2}\right)s_2\tau_d^2\right).\]
    The proof is completed.
\end{proof}

\subsection{Stochastic gradient method with adaptive step sizes}

 In this subsection, we consider the acceleration of the stochastic gradient algorithm by imposing adaptive step sizes. Considering that
 Adam \cite{kingma2014adam} is a popular adaptive optimization strategy renowned for its effectiveness on a wide range of problems, we propose a variant of Algorithm \ref{cdfsg} following its framework, as illustrated in Algorithm \ref{acdfsg}.
 
\begin{algorithm}[htbp]\label{acdfsg}
\caption{Stochastic gradient method with adaptive step sizes for solving \ref{sogse}.}
\KwData{Initialization $X_0,B_0\in\Rnp,Y_0\in\bb{R}^{p\times p}$, step sizes $\{\alpha_k\}_{k\ge 0},\{b_k\}_{k\ge 0}$, parameters $\eta_1, \eta_2 \in [0,1),$ $ \epsilon>0$.}

\For{$k=0,1,2,\dots,K-1$}
 {Update the main iterate $X_{k+1}$ by
    \begin{equation}\label{eq:updatex}
        X_{k+1}=X_k-\alpha_k(\epsilon+\hat{V}_k)^{\odot(-\frac{1}{2})}\odot B_{k}.
    \end{equation}
    
Randomly  and independently draw $\theta_{k+1}$ from $\Theta$, update the tracking iterate $Y_{k+1}$ by \eqref{eq:updatey}.

   Randomly  and independently draw $\xi_{k+1}$ from $\Xi$, update the direction $D_{k+1}$ by \eqref{eq:updated}.

    Update the first-order moment estimate $B_{k+1}$ by
    \begin{equation}\label{eq:updateb}
        B_{k+1}=\eta_1 B_k+(1-\eta_1)D_{k+1}.
    \end{equation}

    Update the second-order moment estimate $V_{k+1}$ by
    \begin{equation}\label{eq:updatev}
        V_{k+1}=\eta_2 \hat{V}_k+(1-\eta_2)D_{k+1}^{\odot 2}.
    \end{equation}
    
    Update the scaling parameter $V_{k+1}$ by
    \begin{equation}\label{eq:updateav}
        \hat{V}_{k+1}=\max(V_{k+1},\hat{V}_k).
    \end{equation}
 }
 \KwResult{$X_K$.}
\end{algorithm}

To establish the convergence properties of Algorithm \ref{acdfsg}, we first make the following assumptions
based on Assumption \ref{as:blanket} and \ref{as:algo}.
\begin{assumption}\label{as:algoadam}
\begin{enumerate}[(1)]
    \item The function $\nabla f_{\xi}$ is uniformly bounded almost surely when $\|X\|\le \frac{\tau_x(\tau_y+3)}{2}$; i.e., 
    \[\sup_{\|X\|\le \frac{\tau_x(\tau_y+3)}{2}}\|\nabla f_{\xi} (X)\|\le \tau_f<\infty\]
    holds almost surely.
    \item The function $W_{\xi,\theta}$ is uniformly Lipschitz continuous with respect to $Y$ almost surely when $\|X\|\le \tau_x$ and $\|Y\|\le \tau_y$; i.e., there is a constant $\bar{L}_W>0$ such that
    \[\sup_{\|X\|\le \tau_x,\|Y_1\|,\|Y_2\|\le \tau_y}\frac{\|W_{\xi,\theta}(X,Y_1)-W_{\xi,\theta}(X,Y_2)\|}{\|Y_1-Y_2\|}\le \bar{L}_W\]
    holds almost surely.
\end{enumerate}
\end{assumption} 
\begin{remark}
     The conditions in Assumption \ref{as:algoadam} are weaker than the bounded stochastic gradient assumption commonly used in Adam's convergence analysis (e.g., \cite{kingma2014adam, chen2018convergence}), and the global Lipschitz continuity condition frequently imposed in studies of nested stochastic optimization problems (e.g., \cite{wang2017stochastic, chen2021solving}).
\end{remark}

The following proposition shows that the squared Euclidean distance $\|X_{k+1}-X_k\|^2$ between two consecutive elements of the sequence $\{X_k\}_{k\ge 0}$ produced by Algorithm \ref{acdfsg} exhibits the order of $\alpha_k^2$.

\begin{proposition}\label{pr:xkdist}
Suppose Assumptions \ref{as:blanket}, \ref{as:algo} and \ref{as:algoadam} hold, then the sequence $\{X_k\}_{k\ge 0}$ generated by Algorithm \ref{cdfsg} satisfies
    \[\begin{aligned}
   \|X_{k+1}-X_k\|^2\le\alpha_k^2np(1-\eta_2)^{-1}(1-\gamma)^{-1},
\end{aligned}\]
where $\eta_1<\eta_2^{\frac{1}{2}}<1$, and $\gamma=\frac{\eta_1^2}{\eta_2}$.
\end{proposition}
\begin{proof}
Selecting $\eta_1<1$ and setting $\gamma:=\frac{\eta_1^2}{\eta_2}$, for every $i\in\{1,\dots,n\}, j\in \{1,\dots,p\}$, we have
    \[\begin{aligned}
        |B_{k+1}(i,j)|{}&=|\eta_1B_k(i,j)+(1-\eta_1)D_k(i,j)|\le \eta_1|B_k(i,j)|+|D_k(i,j)|
        \le \sum_{l=0}^k\gamma^{\frac{k-l}{2}}\eta_2^{\frac{k-l}{2}}|D_l(i,j)|\\
        &\overset{(i)}{\le} \mybrace{\sum_{l=0}^k\gamma^{k-l}}^{\frac{1}{2}}\mybrace{\sum_{l=0}^k\gamma^{k-l}(D_l(i,j))^2}^{\frac{1}{2}}
        \overset{(ii)}{\le}(1-\gamma)^{-\frac{1}{2}}\mybrace{\sum_{l=0}^k\gamma^{k-l}(D_l(i,j))^2}^{\frac{1}{2}},
    \end{aligned}\]
    where the inequality $(i)$ follows from the Cauchy-Schwartz inequality, and the inequality $(ii)$ uses $\sum_{l=0}^{k}\gamma^{k-l}\le \sum_{l=0}^{k}\gamma^k\le(1-\gamma)^{-1}$.
    Since $\hat{V}_1(i,j)\ge(1-\eta_2)(D_1(i,j))^2$, and $\hat{V}_{k+1}(i,j)\ge\eta_2\hat{V}_k(i,j)+(1-\eta_2)(D_k(i,j))^2$, then by induction we have 
    \[\hat{V}_{k+1}(i,j)\ge(1-\eta_2)\mybrace{\sum_{l=0}^k\gamma^{k-l}(D_l(i,j))^2}.\]
    Consequently, we obtain,
    \[|B_{k+1}(i,j)|^2\le(1-\gamma)^{-1}\mybrace{\sum_{l=0}^k\gamma^{k-l}(D_l(i,j))^2}^{\frac{1}{2}}\le(1-\gamma)^{-1}(1-\eta_2)^{-1}\hat{V}_{k+1}(i,j).\]
    From the update \eqref{eq:updatex}, we have
    \[\|X_{k+1}-X_k\|^2=\alpha^2_k\sum_{i,j}(\epsilon+\hat{V}_{k+1}(i,j))^{-1}|B_{k+1}(i,j)|^2\le\alpha^2_k np(1-\eta_2)^{-1}(1-\gamma)^{-1}.\]
    The proof is completed.
\end{proof}

In the following proposition, we show that the sequences $\{B_k\}_{k\ge 0}$ and $\{\hat{V}_k\}_{k\ge 0}$ are all almost surely bounded.
\begin{proposition}\label{pro:bvbound}
    Suppose Assumptions \ref{as:blanket}, \ref{as:algo} and \ref{as:algoadam} hold, then the sequences $\{B_k\}_{k\ge 0}$ and $\{\hat{V}_k\}_{k\ge 0}$ satisfy \[ \quad\sup_{k\ge 0} \|B_k\|\le \tau_d;\quad \sup_{1\le i\le m,1 \le j \le n, k\ge 0}\hat{V}_k(i,j)\le \tau_d^2.\]
\end{proposition}
\begin{proof}

According to the update \eqref{eq:updateb}, we have
   \[\|B_{k+1}\|\le\eta_1\|B_k\|+(1-\eta_1)\|D_k\|\le\eta_1\|B_k\|+(1-\eta_1)\tau_d.\]
Selecting $B_0$ such that $\|B_0\|\le \tau_d$, for instance, setting $B_0:=X_0^\top M_{\theta_0}X_0$, then by induction, we have $\|B_{k+1}\|\le \tau_d$. Similarly, from the update \eqref{eq:updatev}, we have
\[\begin{aligned}
    \hat{V}_{k+1}(i,j){}&\le\max\left\{\eta_2\hat{V}_k(i,j)+(1-\eta_2)(D_k(i,j))^2\right\}\\
    &\le\max\left\{\eta_2\hat{V}_k(i,j)+(1-\eta_2)\tau_d^2\right\}.
\end{aligned}\]
Since $V_1(i,j)\le \hat{V}_1(i,j)\le \tau_d^2$, by induction, we have $\hat{V}_{k+1}(i,j)\le \tau_d^2$.
\end{proof}

Then we are ready to prove the sample complexity results of Algorithm \ref{acdfsg}.

\begin{theorem} 
Suppose Assumptions \ref{as:blanket}, \ref{as:algo} and \ref{as:algoadam} hold, and the step sizes are chosen as $\alpha_k=s_1b_k=\frac{s_2}{\sqrt{K}}$, where $2\ge s_1>0$ and $s_2>0$. Then the iterates $\{X_k\}_{k\ge 0}$ generated by Algorithm \ref{acdfsg} satisfy 
\[\begin{aligned}
    \frac{1}{K}\sum_{k=1}^{K}\mathbb{E}\left[\left\|\nabla h(X_k)\right\|^2\right]\le{}& \frac{1}{\sqrt{K}}\frac{2(\epsilon+\tau_d^2)^{\frac{1}{2}}}{1-\eta_1}\left(s_2^{-1}J_0+\frac{(\tau_h\tau_d)\epsilon^{-\frac{1}{2}}(np+1)}{s_2(1-\eta_1)\sqrt{K}}+8cs_2\tau_x^2\tau_m^2np(1-\eta_1)\tilde{\eta}\right.\\
    &\left.+4s_1^{-2}s_2c\tau_x^4v_m^2+\frac{3-\eta_1}{2}L_hs_2\tilde{\eta}np\right),
\end{aligned}\]
 where  $J_0:=h(X_0)-h(X_*)+c\|C(X_{0})-Y_0\|^2$, $c:=(1-\eta_1)^{-1}\epsilon^{-\frac{1}{2}}\bar{L}_W^2$,  and $X_*\in\arg\min_{X\in\mathbb{R}^{n\times p}} h(X)$.
\end{theorem}
\begin{proof}
For our analysis, we define the following random variable 
    \begin{equation}\label{eq:jkdefi_1}
        J_k:=h(X_k)-h(X_*)-c_k\inner{\nabla h(X_{k}),(\epsilon+\hat{V}_{k})^{\odot(-\frac{1}{2})}\odot B_{k}}+c\|C(X_{k})-Y_k\|^2,
    \end{equation}
    where $c_k:=\sum_{j=k}^{\infty}\eta_1^{j-k}\alpha_j$, and $c:=(1-\eta_1)^{-1}\epsilon^{-\frac{1}{2}}\bar{L}_W^2$.  By Taylor's expansion, we have, 
    \begin{equation}\label{eq:divij}
        \begin{aligned}
    J_{k+1}-J_k={}&h(X_{k+1})-h(X_k)-c_{k+1}\inner{\nabla h(X_{k+1}),(\epsilon+\hat{V}_{k+1})^{\odot(-\frac{1}{2})}\odot B_{k+1}}\\
    &+c\|C(X_{k+1})-Y_{k+1}\|^2-c_k\inner{\nabla h(X_{k}),(\epsilon+\hat{V}_{k})^{\odot(-\frac{1}{2})}\odot B_{k}}-c\|C(X_{k})-Y_k\|^2\\
    \le{}&-c_{k+1}\inner{\nabla h(X_{k+1}),(\epsilon+\hat{V}_{k+1})^{\odot(-\frac{1}{2})}\odot B_{k+1}}+\frac{L_h}{2}\|X_{k+1}-X_k\|^2\\
    &+c\|C(X_{k+1})-Y_{k+1}\|^2-(\alpha_k-c_k)\inner{\nabla h(X_{k}),(\epsilon+\hat{V}_{k})^{\odot(-\frac{1}{2})}\odot B_{k}}-c\|C(X_{k})-Y_k\|^2.
\end{aligned}
    \end{equation}
    We can decompose the inner product term $-\inner{\nabla h(X_{k}),(\epsilon+\hat{V}_{k})^{\odot(-\frac{1}{2})}\odot B_{k}}$ as follows,
    \begin{equation}\label{eq:decom}
        \begin{aligned}
        &-\inner{\nabla h(X_k), (\epsilon+\hat{V}_k)^{\odot(-\frac{1}{2})}\odot B_{k}}\\
        ={}&\underbrace{-(1-\eta_1)\inner{\nabla h(X_k), (\epsilon+\hat{V}_{k-1})^{\odot(-\frac{1}{2})}\odot D_{k}}}_{=:I_{1,k}}\underbrace{-\eta_1\inner{\nabla h(X_k), (\epsilon+\hat{V}_{k-1})^{\odot(-\frac{1}{2})}\odot B_{k-1}}}_{=:I_{2,k}}\\
        &\underbrace{-\inner{\nabla h(X_k), \left((\epsilon+\hat{V}_{k})^{\odot(-\frac{1}{2})}-(\epsilon+\hat{V}_{k-1})^{\odot(-\frac{1}{2})}\right)\odot B_{k}}}_{=:I_{3,k}}.
    \end{aligned}
    \end{equation}
By defining $\bar{D}_{k}:=W_{\xi_k,\theta_k}(X_k,C(X_{k}))$, we have
    \[I_{1,k}=-(1-\eta_1)\inner{\nabla h(X_k), (\epsilon+\hat{V}_{k-1})^{\odot(-\frac{1}{2})}\odot \bar{D}_{k}}+(1-\eta_1)\inner{\nabla h(X_k), (\epsilon+\hat{V}_{k-1})^{\odot(-\frac{1}{2})}\odot (\bar{D}_{k}-D_{k})}. \]
Taking expectation over $\xi_{k}$ and $\theta_{k}$ of $I_{1,k}$ conditioned on $\mathcal{F}_{k-1}$, we have 
\begin{equation}\label{eq:i1}
    \begin{aligned}
        &\mathbb{E}[I_{1,k}|\mathcal{F}_{k-1}]\\
    \le{}&-(1-\eta_1)\inner{\nabla h(X_k), (\epsilon+\hat{V}_{k-1})^{\odot(-\frac{1}{2})}\odot \mathbb{E}[\bar{D}_{k}|\mathcal{F}_{k-1}]}\\
    &+(1-\eta_1)\mathbb{E}\left[\left.\left\|(\epsilon+\hat{V}_{k-1})^{\odot(-\frac{1}{4})}\odot \nabla h(X_k)\right\|\left\|(\epsilon+\hat{V}_{k-1})^{\odot(-\frac{1}{4})}\odot(\bar{D}_{k}-D_{k})\right\|\right|\mathcal{F}_{k-1}\right]\\
    ={}&-(1-\eta_1)\left\|(\epsilon+\hat{V}_{k-1})^{\odot(-\frac{1}{4})}\odot\nabla h(X_k)\right\|^2\\
    &+(1-\eta_1)\left\|(\epsilon+\hat{V}_{k-1})^{\odot(-\frac{1}{4})}\odot \nabla h(X_k)\right\|\mathbb{E}\left[\left.\left\|(\epsilon+\hat{V}_{k-1})^{\odot(-\frac{1}{4})}\odot(\bar{D}_{k}-D_{k})\right\|\right|\mathcal{F}_{k-1}\right]\\
    \overset{(i)}{\le}{}&-(1-\eta_1)\left(1-\frac{a}{4}\right)\left\|(\epsilon+\hat{V}_{k-1})^{\odot(-\frac{1}{4})}\odot\nabla h(X_k)\right\|^2\\
    &+(1-\eta_1)a^{-1}\mathbb{E}\left[\left.\left\|(\epsilon+\hat{V}_{k-1})^{\odot(-\frac{1}{4})}\odot(\bar{D}_{k}-D_{k})\right\|^2\right|\mathcal{F}_{k-1}\right]\\
    \overset{(ii)}{\le}{}&-(1-\eta_1)\left(1-\frac{a}{4}\right)(\epsilon+\tau_d^2)^{-\frac{1}{2}}\|\nabla h(X_k)\|^2+a^{-1}\epsilon^{-\frac{1}{2}}\bar{L}_W^2\mathbb{E}[\|C(X_{k})-Y_{k}\|^2|\mathcal{F}_{k-1}],
\end{aligned}
\end{equation}
where the inequality $(i)$ is due to Young's inequality $t_1t_2\le\frac{a}{4}t_1^2+a^{-1}t_2^2$ with $t_1:=\|(\epsilon+\hat{V}_{k-1})^{\odot(-\frac{1}{4})}\odot\nabla h(X_k)\|$ and $t_2:=(1-\eta_1)^{\frac{1}{2}}\|(\epsilon+\hat{V}_{k-1})^{\odot(-\frac{1}{4})}\odot(\bar{D}_{k}-D_{k})\|$, and the inequality $(ii)$ follows from the entry-wise bound $(\epsilon+\tau_d^2)^{-\frac{1}{2}}\le(\epsilon+\hat{V}_{k-1}(i,j))^{-\frac{1}{2}}\le \epsilon^{-\frac{1}{2}}$ and the fact that $\|\bar{D}_{k}-D_{k}\|^2\le \bar{L}_W^2\|C(X_{k})-Y_{k}\|^2$. 

Similarly, for $I_{2,k}$, we have 
\begin{equation}\label{eq:i2}
    \begin{aligned}
    \mathbb{E}[I_{2,k}|\mathcal{F}_{k-1}]={}&-\eta_1\inner{\nabla h(X_{k-1}),(\epsilon+\hat{V}_{k-1})^{\odot(-\frac{1}{2})}\odot B_{k-1}}\\
    &-\eta_1\inner{\nabla h(X_k)-\nabla h(X_{k-1}),(\epsilon+\hat{V}_{k-1})^{\odot(-\frac{1}{2})}\odot B_{k-1}}\\
    \le{}&-\eta_1\inner{\nabla h(X_{k-1}),(\epsilon+\hat{V}_{k-1})^{\odot(-\frac{1}{2})}\odot B_{k-1}}+\eta_1L_h\alpha_{k-1}^{-1}\|X_k-X_{k-1}\|^2\\
    \overset{(i)}{\le}{}& \eta_1(I_{1,k-1}+I_{2,k-1}+I_{3,k-1})+\eta_1L_h\alpha_{k-1}np(1-\eta_2)^{-1}(1-\gamma)^{-1},
\end{aligned}
\end{equation}
where the inequality $(i)$ follows from Proposition \ref{pr:xkdist} and the decomposition \eqref{eq:decom}.

Obviously, we have $|\nabla_{i,j} h(X_k)|\le \|\nabla h(X_k)\|$ and  $|B_{k}(i,j)|\le \|B_{k}\|$. Since $\hat{V}_{k}(i,j)=\max\{\cdot,\hat{V}_{k-1}(i,j)\}\ge\hat{V}_{k-1}(i,j)$, it holds that $(\epsilon+\hat{V}_{k-1}(i,j))^{-\frac{1}{2}}\ge(\epsilon+\hat{V}_{k}(i,j))^{-\frac{1}{2}}$, then for $I_{3,k}$, we obtain,

\begin{equation}\label{eq:i3}
    \begin{aligned}
    \mathbb{E}[I_{3,k}|\mathcal{F}_{k-1}]{}&=-\sum_{i,j}\nabla_{i,j} h(X_k)\left((\epsilon+\hat{V}_{k}(i,j))^{-\frac{1}{2}}-(\epsilon+\hat{V}_{k-1}(i,j))^{-\frac{1}{2}}\right)B_{k}(i,j)\\
    &\le\|\nabla h(X_k)\|\|B_{k}\|\sum_{i,j}\left((\epsilon+\hat{V}_{k-1}(i,j))^{-\frac{1}{2}}-(\epsilon+\hat{V}_{k}(i,j))^{-\frac{1}{2}}\right)\\
    &\le \tau_h \tau_d\sum_{i,j}\left((\epsilon+\hat{V}_{k-1}(i,j))^{-\frac{1}{2}}-(\epsilon+\hat{V}_{k}(i,j))^{-\frac{1}{2}}\right).
\end{aligned}
\end{equation}

 Taking expectation over $\xi_{k+1}$ and $\theta_{k+1}$ on both sides of \eqref{eq:divij} conditioned on $\mathcal{F}_k$, we have
\[\begin{aligned}
    \mathbb{E}[J_{k+1}|\mathcal{F}_k]-J_k\le{}& c_{k+1}\mathbb{E}[I_{1,k+1}+I_{2,k+1}+I_{3,k+1}|\mathcal{F}_k]+\frac{L_h}{2}\mathbb{E}[\|X_{k+1}-X_k\|^2|\mathcal{F}_k]\\
    &-(c_k-\alpha_k)(I_{1,k}+I_{2,k}+I_{3,k})+c\mathbb{E}[\|C(X_{k+1})-Y_{k+1}\|^2|\mathcal{F}_k]-c\|C(X_{k})-Y_k\|\\
    \overset{(i)}{\le}{}&-c_{k+1}(1-\eta_1)\left(1-\frac{a}{4}\right)(\epsilon+\tau_d^2)^{-\frac{1}{2}}\|\nabla h(X_{k+1})\|^2\\
    &+(\eta_1c_{k+1}-(c_k-\alpha_k))(I_{1,k}+I_{2,k}+I_{3,k})
    +c_{k+1}\eta_1L_h\alpha_{k}np(1-\eta_2)^{-1}(1-\gamma)^{-1}\\
    &+c_{k+1}\tau_h\tau_d\sum_{i,j}\left((\epsilon+\hat{V}_{k}(i,j))^{-\frac{1}{2}}-(\epsilon+\hat{V}_{k+1}(i,j))^{-\frac{1}{2}}\right)\\
    &+(c+c_{k+1}a^{-1}\epsilon^{-\frac{1}{2}}\bar{L}_W^2)\mathbb{E}[\|C(X_{k+1})-Y_{k+1}\|^2|\mathcal{F}_k]\\
    &+\frac{L_h}{2}\alpha_k^2np(1-\eta_2)^{-1}(1-\gamma)^{-1}-c\|C(X_{k})-Y_k\|^2,
\end{aligned}\]
where the inequality $(i)$ follows from inequalities \eqref{eq:i1}-\eqref{eq:i3}, and Proposition \ref{pr:xkdist}.

By setting $\alpha_{k+1}\le \alpha_k$, we obtain $c_k=\sum_{j=k}^{\infty}\eta_1^{j-k}\alpha_j\le(1-\eta_1)^{-1}\alpha_k$. Then it follows from Proposition \ref{pr:yktrack} that,
    \[\begin{aligned}
        &(c+c_{k+1}a^{-1}\epsilon^{-\frac{1}{2}}\bar{L}_W^2)\mathbb{E}[\|C(X_{k+1})-Y_{k+1}\|^2|\mathcal{F}_k]-c\|C(X_{k})-Y_k\|^2\\
        \le{}& c(1+a^{-1}\alpha_k)\mathbb{E}[\|C(X_{k+1})-Y_{k+1}\|^2|\mathcal{F}_k]-c\|C(X_{k})-Y_k\|^2\\
        \le{}&  c((1+a^{-1}\alpha_k)(1-b_k)^2-1)\|C(X_{k})-Y_k\|^2+4c(1+a^{-1}\alpha_k)\tau_x^2\tau_m^2\|X_{k+1}-X_{k}\|^2\\
        &+2c(1+a^{-1}\alpha_k)b_k^2\tau_x^4v_m^2
        \overset{(i)}{\le}{} 2c(1+a^{-1}\alpha_k)\left(2\tau_x^2\tau_m^2\alpha_{k}^2np(1-\eta_2)^{-1}(1-\gamma)^{-1}+\tau_x^4v_m^2b_k^2\right),
    \end{aligned}\]
where the inequality $(i)$ follows from the inequalities $(1+a^{-1}\alpha_k)(1-b_k)^2\le(1+b_k)(1-b_k)^2 =(1-b_k^2)(1-b_k)\le 1$.

By $c_{k+1}=\sum_{j=k+1}^{\infty}\eta_1^{j-k-1}\alpha_j$, it is easy to verify that $\eta_1c_{k+1}=c_k-\alpha_k$ and $c_{k+1}\ge\alpha_{k+1}$. Then we have
    \[\begin{aligned}
    &\mathbb{E}[J_{k+1}|\mathcal{F}_k]-J_k\\
    \le{}&-\alpha_{k+1}(1-\eta_1)\left(1-\frac{a}{4}\right)(\epsilon+\tau_d^2)^{-\frac{1}{2}}\|\nabla h(X_{k+1})\|^2\\
    &+2c(1+a^{-1}\alpha_k)\left(2\tau_x^2\tau_m^2\alpha_{k}^2np(1-\eta_2)^{-1}(1-\gamma)^{-1}+\tau_x^4v_m^2b_k^2\right)
    +L_h np(1-\eta_1)^{-1}(1-\eta_2)^{-1}(1-\gamma)^{-1}\alpha_k^2\\
    &+c_{k+1}\tau_h\tau_d\sum_{i,j}\left((\epsilon+\hat{V}_{k}(i,j))^{-\frac{1}{2}}-(\epsilon+\hat{V}_{k+1}(i,j))^{-\frac{1}{2}}\right)
    +\frac{L_h}{2}\alpha_k^2np(1-\eta_2)^{-1}(1-\gamma)^{-1}.
\end{aligned}\]
Defining $\tilde{\eta}=(1-\eta_1)^{-1}(1-\eta_2)^{-1}(1-\gamma)^{-1}$, and rearranging terms above and telescoping from $k=0,\dots,K-1$, we have
\[\begin{aligned}
   &\sum_{k=0}^{K-1}\alpha_{k+1} (1-\eta_1)\left(1-\frac{a}{4}\right)(\epsilon+\tau_d^2)^{-\frac{1}{2}}\mathbb{E}[\|\nabla h(X_{k+1})\|^2]\\
   \le{}& J_0-\mathbb{E}[J_K]+\sum_{k=0}^{K-1}2c(1+a^{-1}\alpha_k)(2\tau_x^2\tau_m^2\alpha_{k}^2np(1-\eta_1)\tilde{\eta}+\tau_x^4v_m^2b_k^2)\\
   &+\sum_{k=0}^{K-1}L_h np\tilde{\eta}\frac{3-\eta_1}{2}\alpha_k^2+\sum_{k=0}^{K-1} c_{k+1}\tau_h\tau_d\sum_{i,j}\left((\epsilon+\hat{V}_{k}(i,j))^{-\frac{1}{2}}-(\epsilon+\hat{V}_{k+1}(i,j))^{-\frac{1}{2}}\right)\\
  \overset{(i)}{\le}{}&  J_0+(1-\eta)^{-1}\alpha_K\tau_h\tau_d\epsilon^{-\frac{1}{2}}+\sum_{k=0}^{K-1}2c(1+a^{-1}\alpha_k)(2\tau_x^2\tau_m^2\alpha_{k}^2np(1-\eta_1)\tilde{\eta}+\tau_x^4v_m^2b_k^2)\\
   &+\sum_{k=0}^{K-1}L_h np\tilde{\eta}\frac{3-\eta_1}{2}\alpha_k^2+(1-\eta_1)^{-1}\alpha_0\tau_h\tau_d\sum_{i,j}\left((\epsilon+\hat{V}_{0}(i,j))^{-\frac{1}{2}}-(\epsilon+\hat{V}_{K}(i,j))^{-\frac{1}{2}}\right),
\end{aligned}\]
where the inequality $(i)$ holds since
\[\begin{aligned}
    -\mathbb{E}[J_K]={}&-(h(X_K)-h(X_*))+c_k\inner{\nabla h(X_{K}),(\epsilon+\hat{V}_{K})^{\odot(-\frac{1}{2})}\odot B_{K}}-c\|C(X_{k})-Y_K\|^2\\
    \le{} &c_K\epsilon^{-\frac{1}{2}}\|\nabla h(X_{K})\|\|B_{K}\|\le(1-\eta_1)^{-1}\alpha_K\tau_h\tau_d\epsilon^{-\frac{1}{2}}.
\end{aligned}\]
  By setting $\alpha_k=s_1b_k=\frac{s_2}{\sqrt{K}}$ with $2\ge s_1>0$ and $s_2>0$, we can arrive that  
\[\begin{aligned}
    \frac{1}{K}\sum_{k=1}^{K}\mathbb{E}\left[\left\|\nabla h(X_k)\right\|^2\right]\le{}& \frac{1}{\sqrt{K}}\frac{2(\epsilon+\tau_d^2)^{\frac{1}{2}}}{1-\eta_1}\left(s_2^{-1}J_0+\frac{(\tau_h\tau_d)\epsilon^{-\frac{1}{2}}(np+1)}{s_2(1-\eta_1)\sqrt{K}}+8cs_2\tau_x^2\tau_m^2np(1-\eta_1)\tilde{\eta}\right.\\
    &\left.+4s_1^{-2}s_2c\tau_x^4v_m^2+\frac{3-\eta_1}{2}L_hs_2\tilde{\eta}np\right).
\end{aligned}\]
The proof is completed.
\end{proof}
\section{Numerical Experiments}\label{numerical}
In this section, we investigate the numerical performance of our proposed algorithms on the aforementioned GCCA problem.  All of the numerical experiments in this section are performed in MATLAB R2019b under an Ubuntu 20.04.6 operating system on a workstation with an Intel(R) Xeon(R) Silver 4110 CPU at 2.10GHz and 384GB of RAM.

 \subsection{Basic settings}
We test the GCCA problem formulated by \eqref{p:GCCA} with two choices of merit function $g$ including the identity mapping and the Huber loss function. Our experiments are conduct on three frequently used real datasets in machine learning, including Mediamill \cite{snoek_mediamill_2006}, MNIST \cite{lecun1998mnist}, and noisy-MNIST (n-MNIST) \cite{basu2017learning}.  
The Mediamill dataset contains keyframes of a video, with each annotated with 101 labels and consisting of 120 features. We use its first 43900 samples to evaluate the correlation structure between the labels and features.
The MNIST dataset contains gray-scale images of handwritten digits, each sized at $28\times 28$ pixels. We utilize its full training set including 6000 samples to learn correlated representations between the left and right halves of the images. Moreover, the n-MNIST dataset includes three subsets of data created by adding additive white Gaussian noise, motion blur, and a combination of additive white Gaussian noise and reduced contrast to the MNIST dataset, respectively. We perform GCCA to explore the correlation relationships among the three noisy variants of MNIST.
A short summary of the real data examples is given in Table \ref{ta:data}.

\begin{table}[htbp]
    \centering
    \begin{tabular}{|c|c|c|c|}
    \hline
       Name  &  Feature number & Sample size& Remark \\
       \hline
        Mediamill \cite{snoek_mediamill_2006}& 120,101& 43900 &Vedio keyframes with multi-labels\\
        \hline
     MNIST \cite{lecun1998mnist} & 392,392& 60000 & Hand-written digits in gray-scale\\
     \hline
    n-MNIST \cite{basu2017learning} &784,784,784& 60000&Noisy variants of MNIST\\
     \hline
    \end{tabular}
    \caption{Descriptions of real datasets}
    \label{ta:data}
\end{table}

We compare our developed algorithms CDFSG (Algorithm \ref{cdfsg}) and  CDFSG-Ada (Algorithm \ref{acdfsg}) with the landing algorithm. 
Particularly, for the experiments conducted on Mediamill and MNIST with $g$ being the identity function, we also compare them with the Riemannian stochastic gradient-based (RSG$+$) algorithm \cite{meng2021online}, which is an algorithm  specially proposed for solving \eqref{eq:classcca}. 

If not specified, in CDFSG and CDFSG-Ada, the step sizes are fixed as $\alpha_k=s_1b_k=s_2$. The parameters $s_1$ and $s_2$ are picked by grid-search, with $s_1$ selected from the set $\{2^{-k}, k=1,2,3\}$ and $s_2$ from the set $\{l\times 10^{-k},l=1,5,k=1,2,3\}$. 
The parameters of the other compared algorithms are chosen as their default values.
All of the compared algorithms are performed from the same randomly generated initial points, and terminated after one pass over the data with a fixed batch size $l$. Furthermore, All results shown in this paper are the average of 10 runs from different initial points.

For the experiments conducted on Mediamill and MNIST with $g$ being the identity function, the accuracy for $X=((X^{[1]})^\top,(X^{[2]]})^\top)^\top$ is measured by the Proportion of Correlations Captured (PCC) \cite{ma2015finding,ge2016efficient,meng2021online}, which is defined as 
\[\mathrm{PCC}(X) =\frac{\mathrm{TCC}(Z_1X^{[1]]},Z_2X^{[2]})}{\mathrm{TCC}(Z_1\bar{X}^{[1]},Z_2\bar{X}^{[2]})},\]
where $Z_1$ and $Z_2$ denote the data matrices corresponding to two subsets of features from the Mediamill or MNIST dataset, $\bar{X}^{[1]}\in\mathbb{R}^{n_1\times p}$ and $\bar{X}^{[2]}\in\mathbb{R}^{n_2\times p}$ represent the ground-truth solution, and the Total Correlations Captured (TCC) between two matrices represents the sum of their canonical correlations.
Intuitively, PCC characterizes the proportion of correlations captured by certain algorithms compared with the ground-truth solution. Therefore, a higher PCC corresponds to a better solution of the corresponding algorithms. 

It is worth mentioning that all the algorithms in comparison, except for RSG$+$ are infeasible methods, the iterates of which are usually infeasible.
Conventionally, after obtaining a solution $X$ by applying an infeasible approach for minimizing over the generalized  Stiefel manifold, one step of orthogonalization can be performed as a post-processing to achieve a high accuracy for feasibility \cite{xiao_constraint_2022}. However, when $M$ is expectation-formulated as given in \ref{sogse}, such a post-processing scheme becomes intractable. To address this issue, we observe that in CDFSG and CDFSG-Ada, the sequence $\{Y_k\}_{k\ge 0}$ effectively tracks the value of $\{X_k^\top MX_k\}_{k\ge0}$. Therefore, after the last iteration $K$, we can conduct $X_K(Y_K)^{-\frac{1}{2}}$ as a post-processing step for both CDFSG and CDFSG-Ada. As the landing algorithm does not include such sequence, we implement $X(X^\top\hat{M}X)^{-\frac{1}{2}}$ as an alternative, where $\hat{M}$ is an approximation of $M$ obtained from minibatch data. In the following, we use the subscript ``$\star$'' to represent the results after an orthogonalization post-processing, and use the subscript ``$*$'' to represent the results after our proposed post-processing.

\subsection{Robustness to penalty parameter}

In this subsection, we present the numerical performance of CDFSG and CDFSG-Ada on the MNIST dataset under different choices of the penalty parameter $\beta$. 

Figure \ref{fig:betarobust-i} and Figure \ref{fig:betarobust-h} exhibit the performance of CDFSG and CDFSG-Ada with different values of $\beta$ and different choices of $g$ for the CCA problem \eqref{p:GCCA}, where we fix the dimension $p$ as $5$ and choose the batch size as $100$. From the results in these figures, we can observe that the numerical performance of CDFSG and CDFSG-Ada exhibit consistent under different values of $\beta$. These results further demonstrate the robust performance of both  CDFSG and CDFSG-Ada with respect to the penalty parameter $\beta$. Based on these observations, we set $\beta=0.1$ as the default value for both CDFSG and CDFSG-Ada in subsequent experiments.

\begin{figure}[htbp]
    \centering
    \subfigure[PCC]{\includegraphics[width=0.45\hsize, height=0.30\hsize]{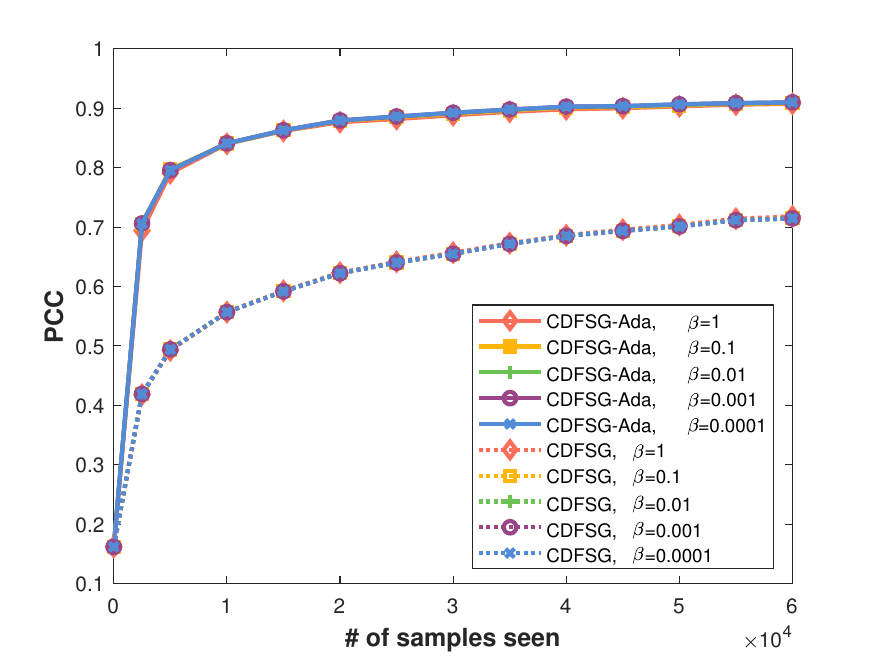}}
    \subfigure[Feasibility violation]{\includegraphics[width=0.45\hsize, height=0.30\hsize]{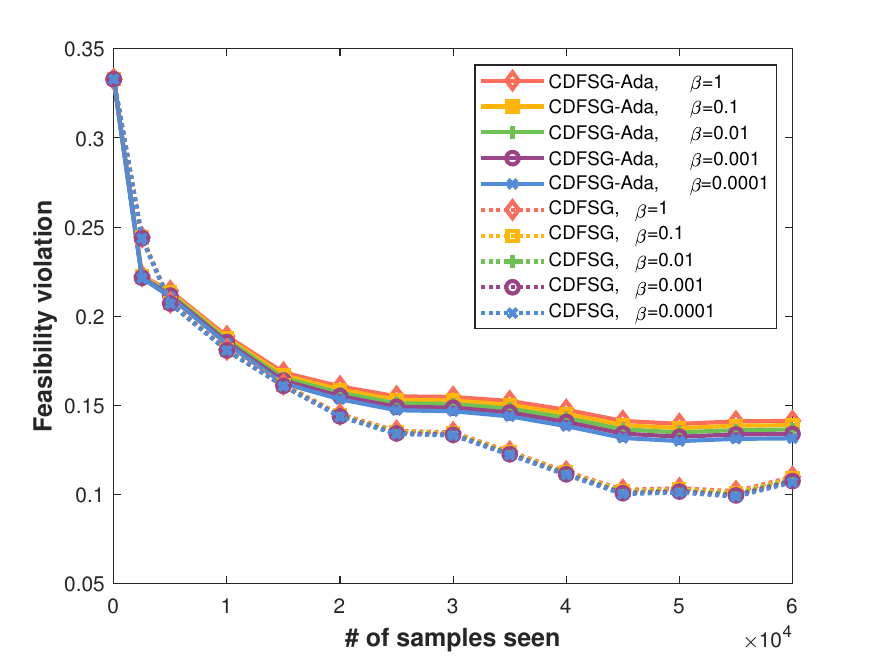}}
    \caption{Performance of CDFSG and CDFSG-Ada with different values of $\beta$ on MNIST for CCA with $g$ as the identity mapping.}
    \label{fig:betarobust-i}
\end{figure}
\begin{figure}[htbp]
    \centering
    \subfigure[Function value]{\includegraphics[width=0.45\hsize, height=0.30\hsize]{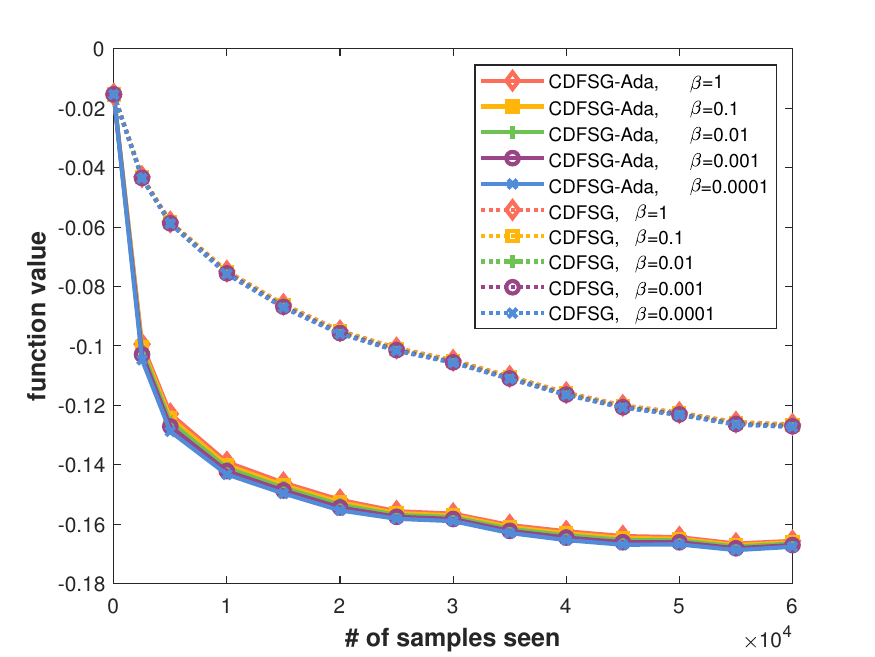}}
    \subfigure[Feasibility violation]{\includegraphics[width=0.45\hsize, height=0.30\hsize]{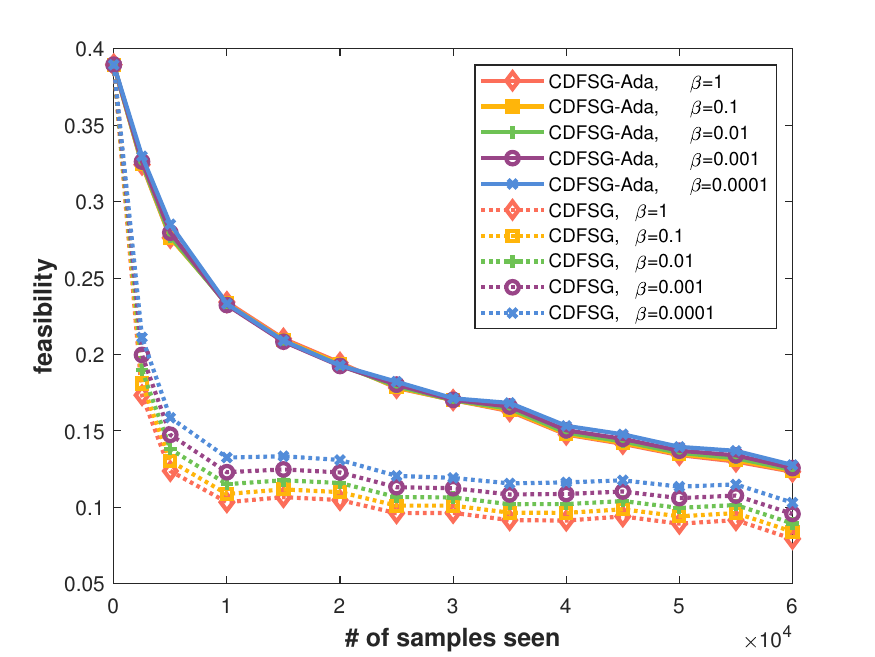}}
    \caption{Performance of CDFSG and CDFSG-Ada with different values of $\beta$ on MNIST for CCA with $g$ as the Huber loss function.}
        \label{fig:betarobust-h}
\end{figure}

\subsection{Numerical results for GCCA with $g$ as the identity mapping}
In this subsection, we compare the numerical performance of CDFSG and CDFSG-Ada with the the landing algorithm and the Riemannian stochastic gradient descent (RSG+) method \cite{meng2021online} on GCCA problems for Mediamill dataset, MNIST dataset and n-MNIST dataset, where the merit function $g$ is chosen as the identity mapping.  The test results on three datasets for different $p$ and batch sizes $l$ are presented in Tables \ref{ta:pcctimemd}-\ref{ta:fvalnmni}, respectively. Additionally, Figures \ref{fig:pcc} and \ref{fig:nmni} illustrate the results with the batch size $l$ fixed at $100$.

From these numerical results, we can conclude that both CDFSG and CDFSG-Ada consistently outperform competing algorithms across test instances, achieving higher accuracy and lower feasibility violations. Additionally, our proposed post-processing procedure successfully improves solution feasibility without compromising accuracy. 
Moreover, as shown in Tables \ref{ta:pcctimemd}-\ref{ta:fvalnmni}, the computational efficiency of CDFSG and CDFSG-Ada is comparable to the Landing algorithm while requiring significantly less CPU time than RSG+. The performance of our proposed algorithms remains robust across varying batch sizes, whereas the Landing algorithm exhibits poor performance and fails to converge with small batch sizes. These results align with our theoretical expectations, as the convergence of the Landing algorithm depends on maintaining iterations within the neighborhood of $\Smnp$.

 \begin{table}[htbp]
    \centering
    \scalebox{0.6}{\begin{tabular}{c?c?c|c|c|c|c|c?c|c|c|c|c|c}
    \Xhline{1.2pt}
     \multirow{2}{*}{Batch size}  &\multirow{2}{*}{Solver}&\multicolumn{6}{c?}{$p=5$} & \multicolumn{6}{c}{$p=10$}\\
    \cline{3-14}
    &&PCC&Fea&Time (s)&PCC*&Fea*&$\text{PCC}^\star$&PCC&Fea&Time (s)&PCC*&Fea*&$\text{PCC}^\star$\\
    \Xhline{1.2pt}
       \multirow{4}{*}{$l=20$}
       &CDFSG     &{0.78}&{{0.17}}&0.75&{0.79}&{{0.14}}&0.78&0.53&{0.37}&0.86&0.53&{0.27}&0.53   \\
    \cline{2-14}
       &CDFSG-Ada&\textbf{{0.80}}&{0.19}&0.82&\textbf{{0.81}}&\textbf{{0.11}}&0.80&\textbf{{0.78}}&{0.36}&0.97&\textbf{{0.78}}&\textbf{{0.20}}&0.78\\
     \cline{2-14}
       &Landing    &0.28&0.29&0.86&0.29&0.79&0.27&0.10&0.52&0.96&0.03&4.97&0.09\\
    \cline{2-14}
       &RSG+       &0.53&-&{75.66}&-&-&-&0.56&-&{265.30}&-&-&-\\
    \Xhline{1.2pt}
            \multirow{4}{*}{$l=50$}
       &CDFSG     &{0.76}&{{0.16}}&0.45&{0.76}&{{0.10}}&0.76&{0.65}&{0.33}&0.51&{0.65}&{0.22}&0.65  \\
    \cline{2-14}
       &CDFSG-Ada&\textbf{{0.85}}&{0.19}&0.47&\textbf{{0.85}}&\textbf{{0.09}}&0.85&\textbf{{0.81}}&{0.30}&0.55&\textbf{{0.81}}&\textbf{{0.15}}&0.81   \\
     \cline{2-14}
       &Landing    &0.38&0.28&0.47&0.39&0.67&0.38&0.32&0.36&0.52&0.31&1.07&0.32 \\
    \cline{2-14}
       &RSG+       &0.62&-&{80.72}&-&-&-&0.60&-&{261.92}&-&-&-\\
    \Xhline{1.2pt}
        \multirow{4}{*}{$l=100$}
       &CDFSG     &{0.74}& {{0.11}}&0.36&{0.74}& {{0.10}}&0.74 &0.55&{0.29}&0.43&0.55&{0.22}&0.55   \\
    \cline{2-14}
       &CDFSG-Ada&\textbf{{0.82}}&{0.13}&0.37&\textbf{{0.82}}&{\textbf{0.08}}&0.82 &\textbf{{0.74}}&{0.28}&0.45&\textbf{{0.74}}&{\textbf{0.13}}&0.74  \\
     \cline{2-14}
       &Landing   &0.48&0.26&0.33&0.47&0.68&0.49 &0.21&0.31&0.37&0.22&1.01&0.21\\
    \cline{2-14}
       &RSG+       &0.61&-&{135.39}&-&-&-&0.60&-&{278.84}&-&-&-\\
    \Xhline{1.2pt}
            \multirow{4}{*}{$l=200$}
       &CDFSG    &{0.68}&{0.09}&0.29&{0.68}&{0.08}&0.68&{0.57}&{0.22}& 0.35&{0.57}&{0.16}&0.57  \\
    \cline{2-14}
       &CDFSG-Ada&\textbf{{0.83}}&{0.12}&0.29&\textbf{{0.83}}&{\textbf{0.06}}&0.83 &\textbf{{0.73}}&{0.20}&0.36&\textbf{{0.73}}&{\textbf{0.10}}&0.73 \\
     \cline{2-14}
       &Landing   &0.37&0.24&0.25&0.37&0.34&0.21&0.20&0.29& 0.28&0.20&0.49&0.20  \\
    \cline{2-14}
       &RSG+       &0.59&-&{150.70}&-&-&-&0.54&-&{292.33}&-&-&-\\
       \Xhline{1.2pt}
    \end{tabular}}
    \caption{PCC, feasibility violation (Fea), and  CPU time of CDFSG, CDFSG-Ada, Landing, and RSG+ with different batch sizes on Mediamill for different $p$.}
    \label{ta:pcctimemd}
\end{table}

 \begin{table}[htbp]
    \centering
    \scalebox{0.6}{\begin{tabular}{c?c?c|c|c|c|c|c?c|c|c|c|c|c}
    \Xhline{1.2pt}
     \multirow{2}{*}{Batch size}  &\multirow{2}{*}{Solver}&\multicolumn{6}{c?}{$p=5$} & \multicolumn{6}{c}{$p=10$}\\
    \cline{3-14}
    &&PCC&Fea&Time (s)&PCC*&Fea*&$\text{PCC}^\star$&PCC&Fea&Time (s)&PCC*&Fea*&$\text{PCC}^\star$\\
    \Xhline{1.2pt}
         \multirow{4}{*}{$l=20$}
         &CDFSG   &{0.90}&{0.23}&2.18&{0.90}&{{0.20}}&0.90&{0.84}&{0.31}&3.41&{0.84}&{0.25}&0.84\\
    \cline{2-14}
       &CDFSG-Ada&\textbf{{0.92}}&{0.29}&2.45&\textbf{{0.92}}&\textbf{{0.15}}&0.92&\textbf{{0.88}}&{0.35}&4.30&\textbf{{0.89}}&\textbf{{0.17}}&0.88 \\
      \cline{2-14}
       &Landing    &{-}   &{-}&{-}&{-}   &{-}& {-}   &{-}    &{-}   &{-}   &{-}    &{-}  & {-}   \\
    \cline{2-14}
       &RSG+       &0.71&-&{40.52}&-&-&-&0.63&-&{28.14}&-&-&-\\
       \Xhline{1.2pt}
    \multirow{4}{*}{$l=50$}
       &CDFSG     &{0.94}& {0.14}&1.77&{0.94}& {{0.12}}&0.94&{0.88}&{{0.29}}&2.42&{0.88}&{{0.20}}&0.88\\
    \cline{2-14}
       &CDFSG-Ada&\textbf{{0.95}}&{0.22}&1.87&\textbf{{0.95}}&\textbf{{0.10}}&0.95&\textbf{{0.91}}&{0.30}&2.83&\textbf{{0.91}}&\textbf{{0.15}}&0.91\\
     \cline{2-14}
       &Landing   &{-}  &{-}&{-}  &{-}&{-}  &{-} &{-}  &{-} &{-}  &{-} &{-}  &{-}  \\
    \cline{2-14}
       &RSG+       &0.73&-&{31.84}&-&-&-&0.67&-&{22.02}&-&-&-\\
    \Xhline{1.2pt}

         \multirow{4}{*}{$l=100$}
         &CDFSG   &{\textbf{0.96}}& {0.14}&1.51&{\textbf{0.96}}& {0.12}&0.96&{\textbf{0.93}}&{{0.20}}&1.69&{\textbf{0.93}}&{{0.16}}&0.93\\
    \cline{2-14}
       &CDFSG-Ada&{\textbf{0.96}}&0.19&1.60&{\textbf{0.96}}&\textbf{{0.09}}&0.96&\textbf{{0.93}}&{0.25}&1.82&\textbf{{0.93}}&{\textbf{0.12}}&0.93 \\
      \cline{2-14}
       &Landing  &0.87&0.14&1.30&0.87&0.45&0.87&0.69&0.23&1.71&0.68&0.82&0.70\\
    \cline{2-14}
       &RSG+ &0.73&-&{28.57}&-&-&- &0.61&-&{135.39}&-&-&-     \\
       \Xhline{1.2pt}
    \multirow{4}{*}{$l=200$}
       &CDFSG     &{0.92}&{0.09}&1.55&{0.92}&\textbf{{0.07}}&0.92&{0.88}&{{0.17}}&1.56&{0.88}&{{0.12}}&0.88\\
    \cline{2-14}
       &CDFSG-Ada&{\textbf{0.96}}&0.17&1.59&{\textbf{0.96}}&0.08&0.96&\textbf{{0.94}}&{0.22}&1.64&\textbf{{0.94}}&{\textbf{0.10}}&0.94\\
     \cline{2-14}
       &Landing   &0.92&0.07&1.42&0.92&0.28&0.92&0.81&0.28&1.54&0.80&0.88&0.81\\
    \cline{2-14} 
    &RSG+      &0.73&-&{27.57}&-&-&-&0.59&-&{150.70}&-&-&-\\
       \Xhline{1.2pt}
    \end{tabular}}
    \caption{PCC, feasibility violation (Fea), and  CPU time of CDFSG, CDFSG-Ada, Landing, and RSG+ with different batch sizes on MNIST for different $p$.}
    \label{ta:pcctimemn}
\end{table}

 \begin{table}[htbp]
    \centering
    \scalebox{0.6}{\begin{tabular}{c?c?c|c|c|c|c|c?c|c|c|c|c|c}
    \Xhline{1.2pt}
     \multirow{2}{*}{Batch size}  &\multirow{2}{*}{Solver}&\multicolumn{6}{c?}{$p=5$} & \multicolumn{6}{c}{$p=10$}\\
    \cline{3-14}
    &&Fval&Fea&Time (s)&Fval*&Fea*&$\text{Fval}^\star$&Fval&Fea&Time (s)&Fval*&Fea*&$\text{Fval}^\star$\\
    \Xhline{1.2pt}
            \multirow{3}{*}{$l=20$}
       &CDFSG     &{-9.51}&{0.20}&194.27&{-9.51}& {0.14} &-9.52 &{-18.65}&  {0.29}&201.33&{-18.65}&  {0.15} &-18.66   \\
    \cline{2-14}
       &CDFSG-Ada&\textbf{{-9.59}}&{0.13}&192.79&\textbf{{-9.59}}&\textbf{{0.07}} &-9.59  &{-18.70}& {0.21}&202.66&\textbf{{-18.71}}& \textbf{{0.11}}&-18.71    \\
     \cline{2-14}
       &Landing  &-&-&-&- &-& - &-&  - &-& -& -& -    \\
    \Xhline{1.2pt}
            \multirow{3}{*}{$l=50$}
       &CDFSG     &{-9.55}& {0.15}&82.99 &{-9.55}& {0.11}&-9.55&\textbf{{-18.67}}&  {0.24}&90.01 &\textbf{{-18.67}}&  {0.15}&-18.67  \\
    \cline{2-14}
       &CDFSG-Ada&\textbf{{-9.61}}&{{0.09}}&85.15&\textbf{{-9.61}}&{\textbf{0.07}} &-9.61  &\textbf{{-18.67}}& {0.14}&85.75&\textbf{{-18.67}}& \textbf{{0.08}}&-18.67      \\
     \cline{2-14}
       &Landing    &-& -  &-& - &-& - &- & -&-&-&- &-  \\
    \Xhline{1.2pt}
        \multirow{3}{*}{$l=100$}
       &CDFSG        &{-9.00}& 0.10 &44.35 &{-9.00}& 0.09&-9.00  &{-17.24}& {0.20}&50.37&{-17.25}& {0.12}&-17.25  \\
    \cline{2-14}
       &CDFSG-Ada&\textbf{-9.61}&{{0.14}}&46.41&\textbf{{-9.61}}&{\textbf{0.07}}&-9.61&\textbf{{-18.79}}&{0.16}&48.52&\textbf{{-18.79}}&\textbf{{0.08}} &-18.79  \\
     \cline{2-14}
       &Landing &  {-8.85}  &0.12 &66.68& -11.15& 1.75 & -8.73 &{-14.92}& 0.19&62.22&-18.90&2.51&-14.69\\
    \Xhline{1.2pt}
            \multirow{3}{*}{$l=200$}
       &CDFSG     &{-8.09}& {0.13}&  27.81&{-8.09}& {0.10}&-8.09& {-15.50}& {{0.24}}&30.78 &{-15.52}& {{0.17}}&-15.53 \\
    \cline{2-14}
       &CDFSG-Ada&\textbf{{-9.46}}& {0.07}&25.85&\textbf{{-9.46}}& {\textbf{0.06}} &-9.46 &\textbf{{-18.62}}&{{0.12}}&28.98&\textbf{{-18.62}}&{\textbf{0.10}}&-18.62    \\
     \cline{2-14}
       &Landing     &{-6.70}&  0.05 &34.94& -7.18 & 0.59&  -6.71 &-12.01&0.10 &36.34&-13.61&1.21&-11.96   \\
    \Xhline{1.2pt}
    \end{tabular}}
    \caption{Function value (Fval), feasibility violation (Fea), and  CPU time of CDFSG, CDFSG-Ada, and Landing with different batch sizes on n-MNIST for different $p$.}
    \label{ta:fvalnmni}
\end{table}

\begin{figure}[htbp]
    \centering
         \subfigure[Mediamill, $p=5$]{\includegraphics[width=0.22\hsize, height=0.2\hsize]{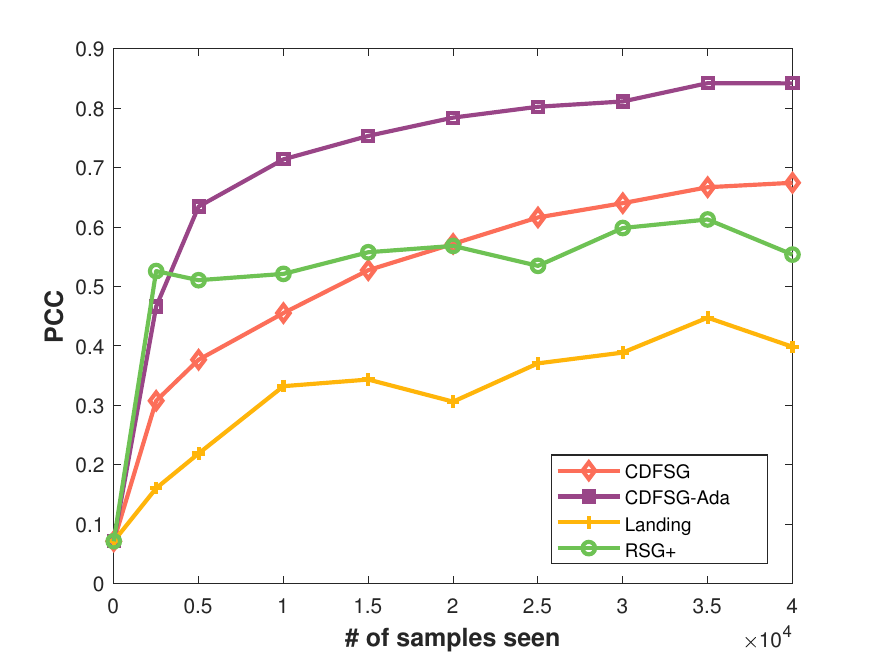}}
      \subfigure[Mediamill, $p=5$]{\includegraphics[width=0.22\hsize, height=0.2\hsize]{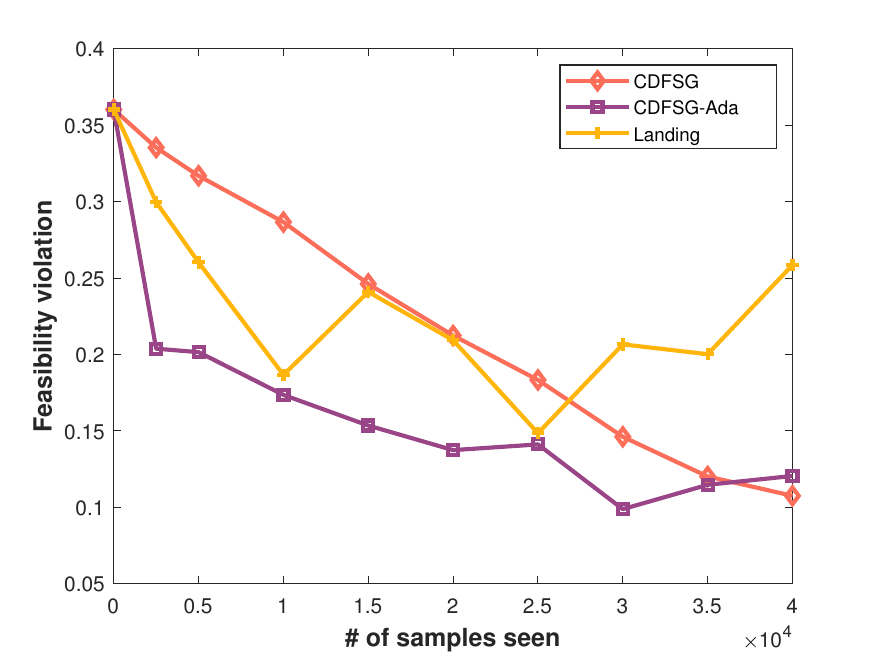}}
    \subfigure[Mediamill, $p=10$]{\includegraphics[width=0.22\hsize, height=0.2\hsize]{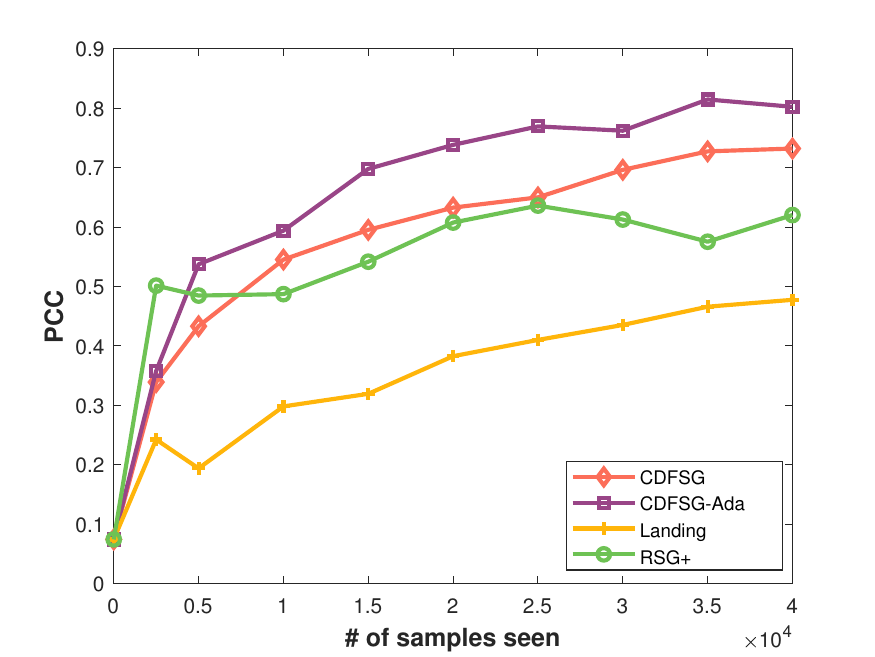}}
     \subfigure[Mediamill, $p=10$]{\includegraphics[width=0.22\hsize, height=0.2\hsize]{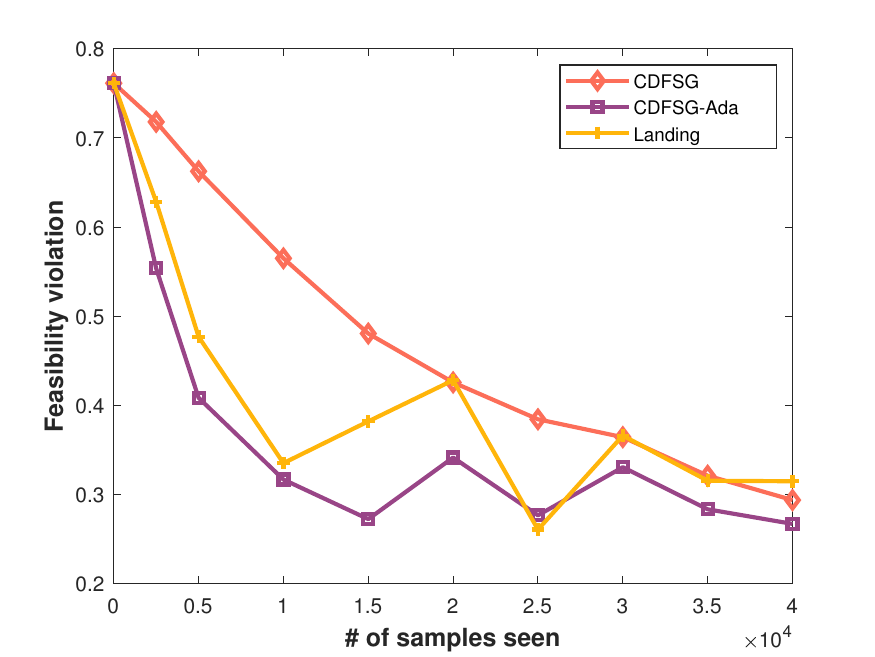}}\\
    \subfigure[MNIST, $p=5$]{\includegraphics[width=0.22\hsize, height=0.2\hsize]{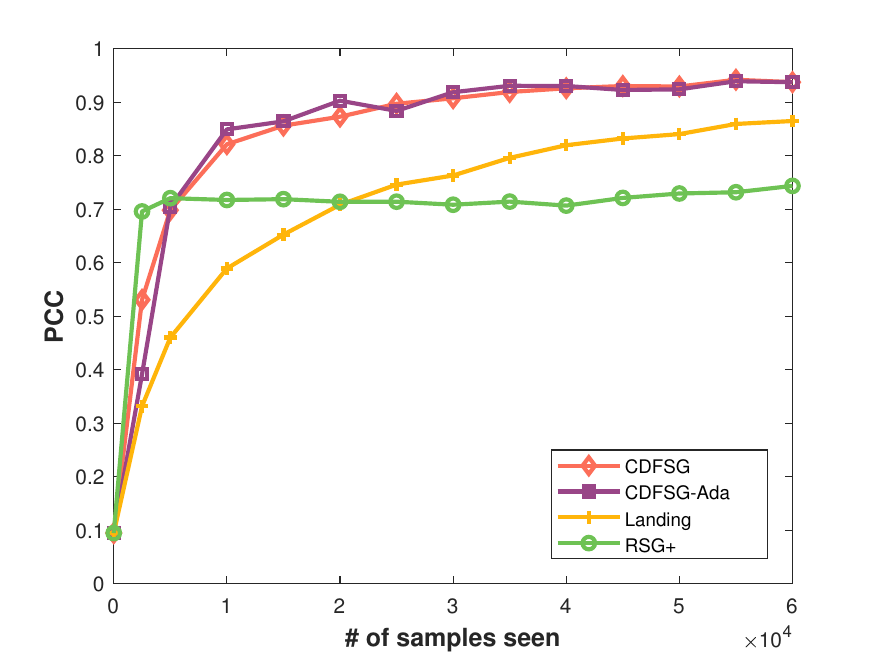}}
    \subfigure[MNIST, $p=5$]{\includegraphics[width=0.22\hsize, height=0.2\hsize]{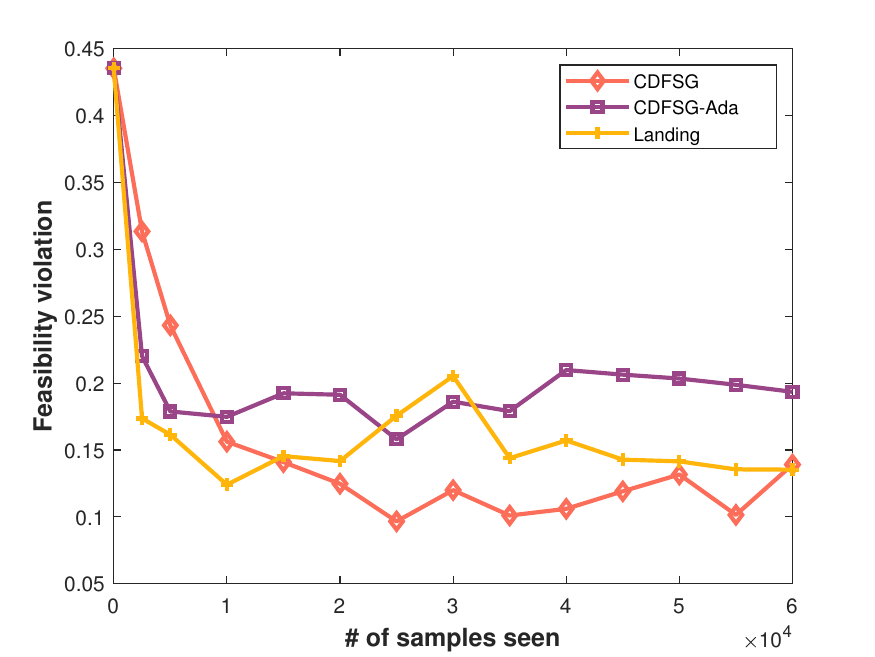}}
    \subfigure[MNIST, $p=10$]{\includegraphics[width=0.22\hsize, height=0.2\hsize]{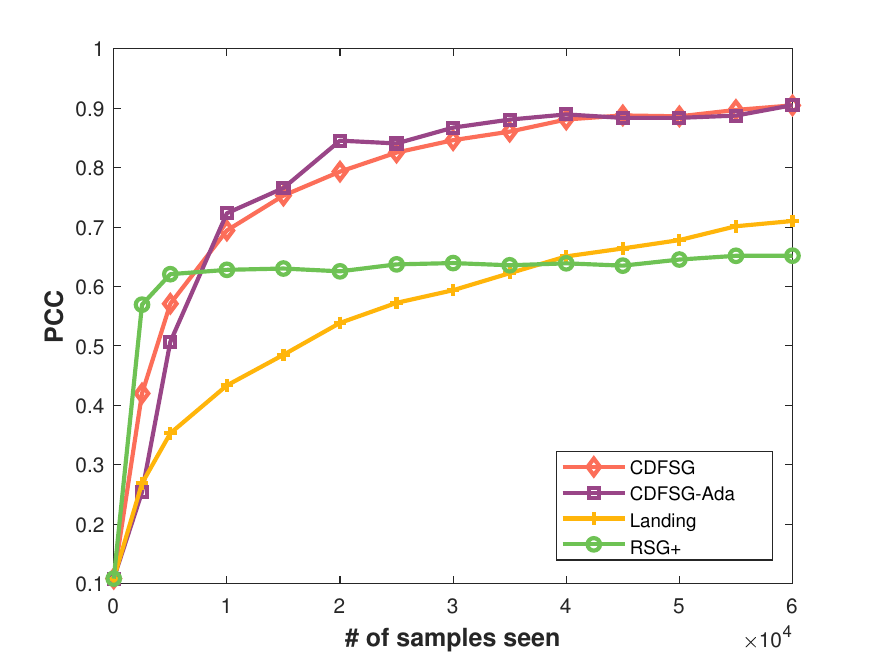}}
     \subfigure[MNIST, $p=10$]{\includegraphics[width=0.22\hsize, height=0.2\hsize]{mn_10_fva_i.pdf}}
    \caption{PCC and feasibility variation of CDFSG, CDFSG-Ada, Landing and RSG+ with batch size fixed at $100$ on Mediamill and MNIST for different $p$.}
    \label{fig:pcc}
\end{figure}

\begin{figure}[htbp]
    \centering
    \subfigure[n-MNIST, $p=5$]{\includegraphics[width=0.22\hsize, height=0.2\hsize]{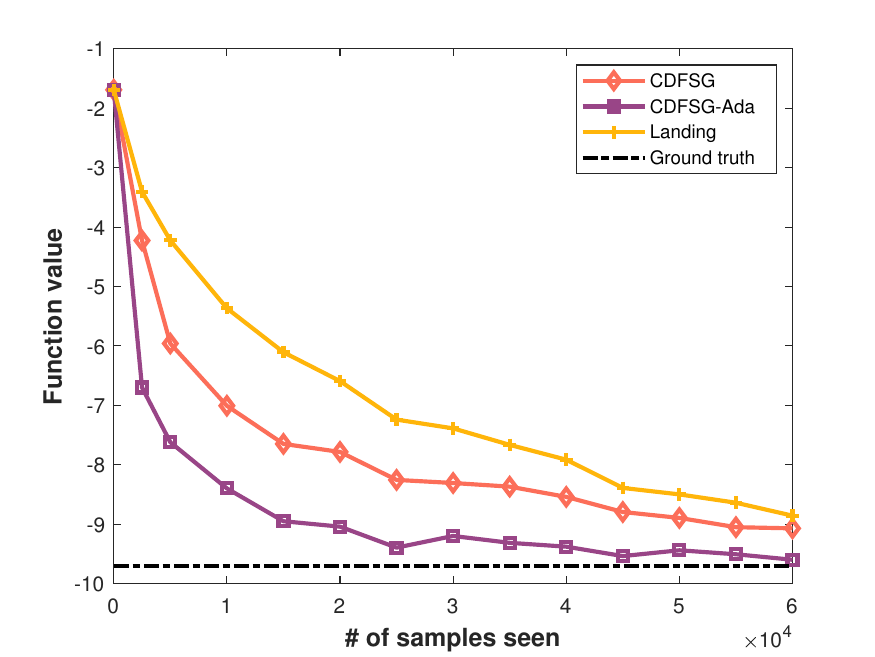}}
    \subfigure[n-MNIST, $p=5$]{\includegraphics[width=0.22\hsize, height=0.2\hsize]{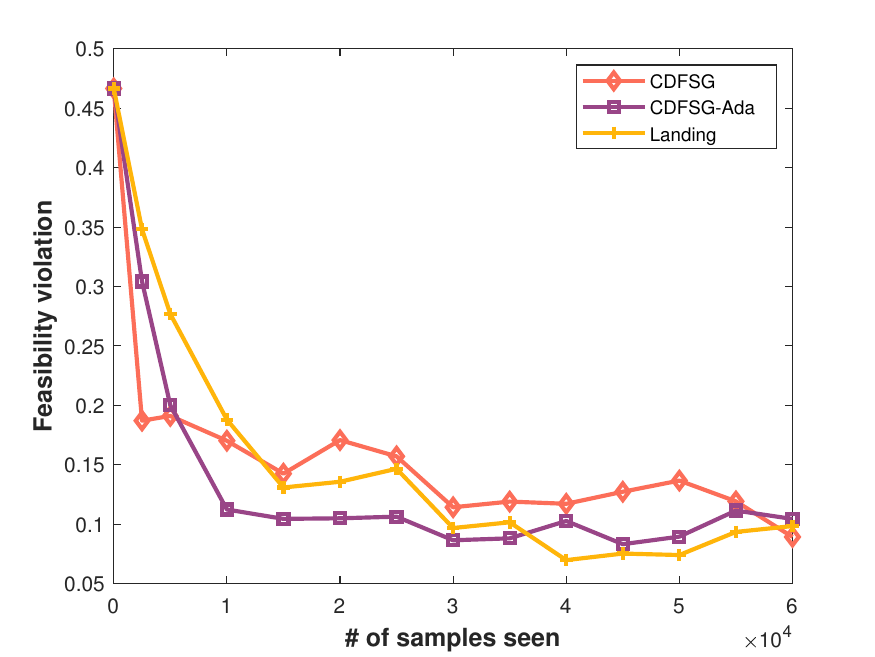}}
    \subfigure[n-MNIST, $p=10$]{\includegraphics[width=0.22\hsize, height=0.2\hsize]{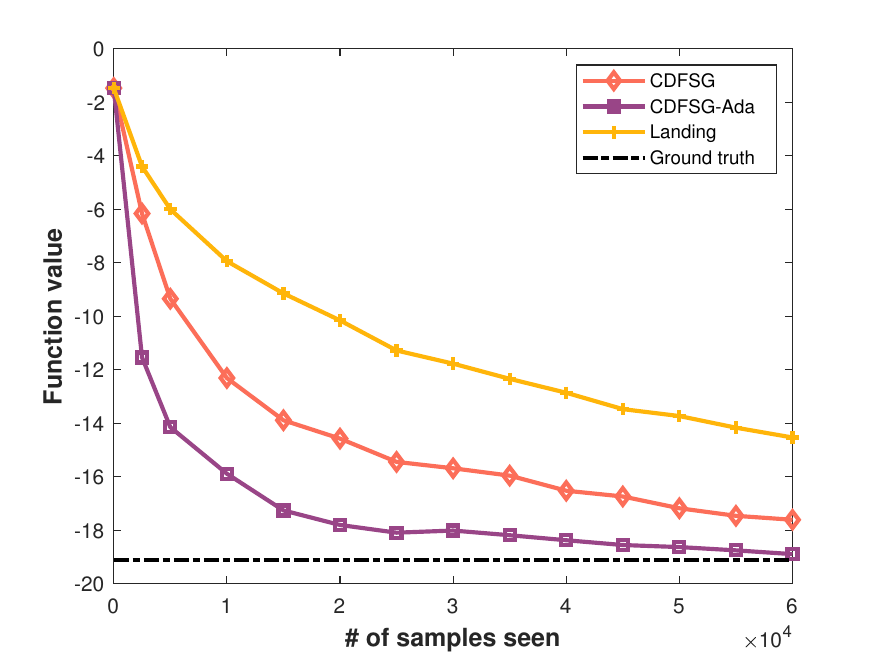}}
     \subfigure[n-MNIST, $p=10$]{\includegraphics[width=0.22\hsize, height=0.2\hsize]{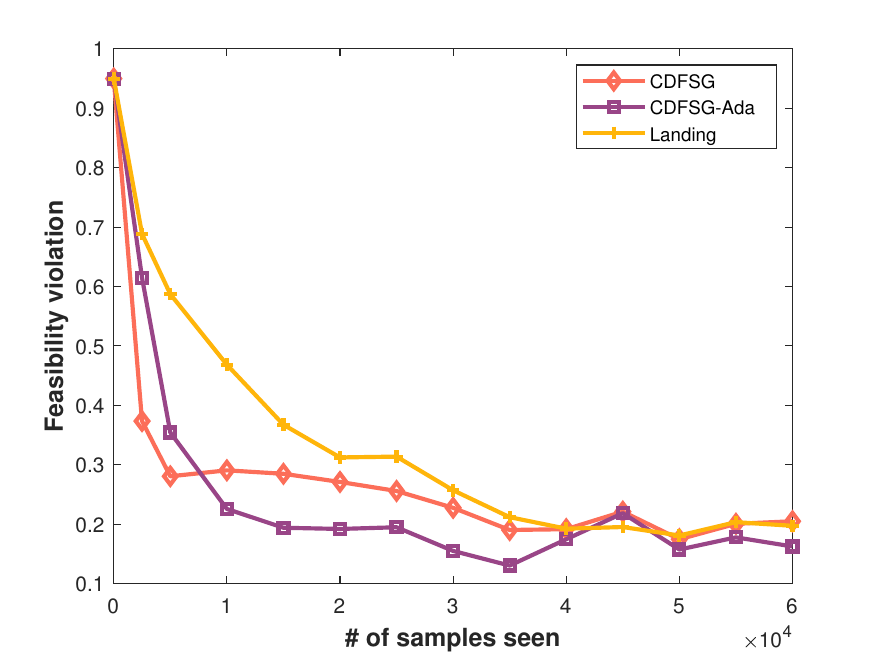}}
    \caption{Function value and feasibility variation of CDFSG, CDFSG-Ada, Landing with batch size fixed at $100$ on n-MNIST for different $p$.}
    \label{fig:nmni}
\end{figure}

\subsection{Numerical results for GCCA with $g$ as the Huber loss function}
In this subsection, we compare the numerical performance of CDFSG and CDFSG-Ada with the the landing algorithm and the Riemannian stochastic gradient descent (RSG+) method \cite{meng2021online} on GCCA problems for Mediamill dataset, MNIST dataset and n-MNIST dataset, where the merit function $g$ is chosen as the Huber loss. Tables \ref{ta:fvalmd}-\ref{ta:fvalnmnh} present the average test results over 10 runs on three datasets for $p=5$ and $p=10$ with different batch sizes $l$, respectively. 
Additionally, Figure \ref{fig:fafe} illustrates the test results with the batch size $l$ fixed at $100$.

We can conclude from Tables  \ref{ta:fvalmn}-\ref{ta:fvalnmnh}  and Figure \ref{fig:fafe} that,  CDFSG and CDFSG-Ada achieve notably lower function values and feasibility violations than the landing algorithm, highlighting their robustness and efficiency in solving \ref{sogse}. Furthermore, our proposed post-processing for CDFSG and CDFSG-Ada does improve the feasibility while maintaining the function value.

 \begin{table}[htbp]
    \centering
    \scalebox{0.6}{\begin{tabular}{c?c?c|c|c|c|c|c?c|c|c|c|c|c}
    \Xhline{1.2pt}
     \multirow{2}{*}{Batch size}  &\multirow{2}{*}{Solver}&\multicolumn{6}{c?}{$p=5$} & \multicolumn{6}{c}{$p=10$}\\
    \cline{3-14}
    &&Fval&Fea&Time (s)&Fval*&Fea*&$\text{Fval}^\star$&Fval&Fea&Time (s)&Fval*&Fea*&$\text{Fval}^\star$\\
    \Xhline{1.2pt}
            \multirow{3}{*}{$l=20$}
       &CDFSG     &{-0.07}   & {0.14}&0.69&{-0.07}   & {0.11} &-0.07 &{-0.05}&  {0.32}&0.81&{-0.05}&  {0.25} &-0.05    \\
    \cline{2-14}
       &CDFSG-Ada&\textbf{{-0.10}}&{0.16}&0.71&\textbf{{-0.10}}&\textbf{{0.07}} &-0.10  &\textbf{{-0.11}}& {0.36}&0.84&\textbf{{-0.11}}& \textbf{{0.19}}&-0.11    \\
     \cline{2-14}
       &Landing  &{-0.01}  &0.20&0.79& -0.01 &  -1.70 &  -0.01 &{-0.01}&0.40&0.93&-0.02&4.42 &-0.01    \\
    \Xhline{1.2pt}
            \multirow{3}{*}{$l=50$}
       &CDFSG     &{-0.06}& {0.10}&0.40 &{-0.06}& {0.09}&-0.06&{-0.05}&  {0.21}&0.45 &{-0.05}&  {0.17}&-0.05  \\
    \cline{2-14}
       &CDFSG-Ada&\textbf{{-0.11}}&{{0.07}}&0.41&\textbf{{-0.11}}&{\textbf{0.03}} &-0.11  &\textbf{{-0.13}}& {0.14}&0.48 &\textbf{{-0.13}}& \textbf{{0.08}}&-0.13     \\
     \cline{2-14}
       &Landing    & {-0.01}   &0.15&0.44& -0.01 & 0.53&  -0.01&{-0.01}&  0.38&0.48&-0.01&1.59&-0.01  \\
    \Xhline{1.2pt}
        \multirow{3}{*}{$l=100$}
       &CDFSG        &{-0.04}& 0.13 &0.29 &{-0.04}& 0.11&-0.04  &{-0.04}& {0.19}&0.35&{-0.04}& {0.16}&-0.04  \\
    \cline{2-14}
       &CDFSG-Ada&\textbf{{-0.11}}&{{0.13}}&0.30&\textbf{{-0.11}}&{\textbf{0.06}}&-0.11&\textbf{{-0.14}}&{{0.17}}&0.34&\textbf{{-0.14}}&\textbf{{0.10}} &-0.14  \\
     \cline{2-14}
       &Landing &  {0.00}  &0.16 &0.31& 0.00& 0.58 & 0.00 &{-0.01}& 0.31&0.34&-0.01&-0.89&-0.01  \\
    \Xhline{1.2pt}
            \multirow{3}{*}{$l=200$}
       &CDFSG     &{-0.03}& {0.06}&  0.24&{-0.03}& {0.05}&-0.03& {-0.02}& {{0.18}}&0.26 &{-0.02}& {{0.13}}&-0.02 \\
    \cline{2-14}
       &CDFSG-Ada&\textbf{{-0.10}}& {0.08}&0.25&\textbf{{-0.10}}& {\textbf{0.03}} &-0.10 &\textbf{{-0.13}}&{{0.18}}&0.26&\textbf{{-0.13}}&{\textbf{0.08}}&-0.13    \\
     \cline{2-14}
       &Landing     &{-0.01}&  0.14 &0.22& -0.01 & 0.19&  -0.01 &-0.01&0.39 &0.24&-0.01&0.50&-0.01   \\
    \Xhline{1.2pt}
    \end{tabular}}
    \caption{Function value (Fval), feasibility violation (Fea), and  CPU time of CDFSG, CDFSG-Ada, and Landing with different batch sizes on Mediamill for different $p$.}
    \label{ta:fvalmd}
\end{table}

 \begin{table}[htbp]
    \centering
    \scalebox{0.6}{\begin{tabular}{c?c?c|c|c|c|c|c?c|c|c|c|c|c}
    \Xhline{1.2pt}
     \multirow{2}{*}{Batch size}  &\multirow{2}{*}{Solver}&\multicolumn{6}{c?}{$p=5$} & \multicolumn{6}{c}{$p=10$}\\
    \cline{3-14}
    &&Fval&Fea&Time (s)&Fval*&Fea*&$\text{Fval}^\star$&Fval&Fea&Time (s)&Fval*&Fea*&$\text{Fval}^\star$\\
    \Xhline{1.2pt}
          \multirow{3}{*}{$l=20$}
         &CDFSG   &\textbf{{-0.20}}&{0.10}&2.12&\textbf{{-0.20}}&{0.09}&-0.20&{-0.32}&0.18&3.19&{-0.32}&{0.14}&-0.32\\
    \cline{2-14}
       &CDFSG-Ada&\textbf{{-0.20}}&{{0.08}}&2.21&\textbf{{-0.20}}&\textbf{{0.05}}&-0.20&\textbf{{-0.37}}&{{0.16}}&3.40&\textbf{{-0.37}}&{\textbf{0.11}}&-0.37\\
      \cline{2-14}
       &Landing    &{-0.05}   &0.09&2.45&-0.05&0.79&0.05   &{-0.11}  &0.06&1.43&-0.11&0.74&-0.11    \\
       \Xhline{1.2pt}
    \multirow{3}{*}{$l=50$}
       &CDFSG     &{-0.19}&0.08&1.39&{-0.19}&0.06&-0.19 &{-0.36}&0.16&2.03&{-0.36}&0.12&-0.36\\
    \cline{2-14}
       &CDFSG-Ada&\textbf{{-0.21}}&{0.06}&1.44&\textbf{{-0.21}}&\textbf{{0.03}}&-0.21&\textbf{{-0.39}}&{0.13}&2.10&\textbf{{-0.39}}&{\textbf{0.10}}&-0.39\\
     \cline{2-14}
       &Landing    &{- 0.04}  &0.16&3.47&-0.07&3.96 &-0.04   &{-0.07}   &0.10&1.89&-0.07&1.46&-0.07 \\
    \Xhline{1.2pt}
           \multirow{3}{*}{$l=100$}
         &CDFSG   &{-0.19}&0.07&1.34&{-0.19}&0.06&-0.19&{-0.35}&0.14&1.31&{-0.35}&0.12&-0.35 \\
    \cline{2-14}
       &CDFSG-Ada&\textbf{{-0.20}}&0.06&1.23&\textbf{{-0.20}}&\textbf{{0.03}}&-0.20&\textbf{{-0.37}}&{{0.10}}&1.38&\textbf{{-0.37}}&\textbf{{0.05}}&-0.37 \\
      \cline{2-14}
       &Landing     &{-0.07}  &0.04&1.20&-0.07& 0.41 & -0.07&{-0.07}& 0.10&  1.47&-0.07&0.84&-0.07   \\
       \Xhline{1.2pt}
    \multirow{3}{*}{$l=200$}
       &CDFSG     &\textbf{{-0.19}} & {{0.10}}&1.19&\textbf{{-0.19}} & {{0.08}}&-0.19 &{-0.35}&0.15&1.40&{-0.35}&0.13&-0.35 \\
    \cline{2-14}
       &CDFSG-Ada& \textbf{{-0.19}}& {{0.10}}&1.17&\textbf{{-0.19}}& \textbf{{0.05}} &-0.19 &\textbf{{-0.39}}&{0.14}&1.37&\textbf{{-0.39}}&{\textbf{0.05}}&-0.39\\
     \cline{2-14}
       &Landing     & {-0.02} & 0.08&1.24&-0.02& 0.35& -0.02&{-0.06}& 0.10&1.57 &-0.07&0.46&-0.06  \\
       \Xhline{1.2pt}
    \end{tabular}}
    \caption{Function value (Fval), feasibility violation (Fea), and  CPU time of CDFSG, CDFSG-Ada, and Landing with different batch sizes on MNIST for different $p$.}
    \label{ta:fvalmn}
\end{table}

 \begin{table}[htbp]
    \centering
    \scalebox{0.6}{\begin{tabular}{c?c?c|c|c|c|c|c?c|c|c|c|c|c}
    \Xhline{1.2pt}
     \multirow{2}{*}{Batch size}  &\multirow{2}{*}{Solver}&\multicolumn{6}{c?}{$p=5$} & \multicolumn{6}{c}{$p=10$}\\
    \cline{3-14}
    &&Fval&Fea&Time (s)&Fval*&Fea*&$\text{Fval}^\star$&Fval&Fea&Time (s)&Fval*&Fea*&$\text{Fval}^\star$\\
    \Xhline{1.2pt}
            \multirow{3}{*}{$l=20$}
       &CDFSG     &\textbf{{-0.41}}&{0.26}&243.53&\textbf{{-0.41}}& {0.20} &-0.40 &{-0.76}&{0.28}&232.69&{-0.76}& {0.14} &-0.76    \\
    \cline{2-14}
       &CDFSG-Ada&{\textbf{-0.41}}&{0.19}&245.12&\textbf{{-0.41}}&\textbf{{0.10}} &-0.41  &\textbf{{-0.78}}&{0.23}&229.43&\textbf{{-0.78}}&\textbf{{0.11}} &-0.78     \\
     \cline{2-14}
       &Landing  &-&-&-&- &-& - &-&  - &-& -& -& -    \\
    \Xhline{1.2pt}
            \multirow{3}{*}{$l=50$}
       &CDFSG        &{-0.39}& 0.11 &104.50 &{-0.39}& \textbf{0.08}&-0.39  &{-0.75}& 0.23 &95.43 &{-0.74}& 0.18&-0.74 \\
    \cline{2-14}
       &CDFSG-Ada&-0.40&{{0.15}}&102.31&\textbf{{-0.40}}&{\textbf{0.08}}&-0.40 &{\textbf{-0.78}}& 0.20 &95.09 &{\textbf{-0.78}}& \textbf{0.11}&-0.78   \\
     \cline{2-14}
       &Landing    &-& -  &-& - &-& - &- & -&-&-&- &-  \\
    \Xhline{1.2pt}
        \multirow{3}{*}{$l=100$}
       &CDFSG        &{-0.37}& 0.11 &51.80 &{-0.37}& 0.06&-0.37  &{-0.70}& 0.15 &51.90 &{-0.70}& 0.11&-0.70    \\
    \cline{2-14}
       &CDFSG-Ada&{\textbf{-0.39}}& 0.05 &51.42 &{\textbf{-0.39}}& \textbf{0.04}&-0.36 &{\textbf{-0.77}}& 0.08 &50.47 &{\textbf{-0.77}}& \textbf{0.06}&-0.77    \\
     \cline{2-14}
       &Landing &{-0.18}& 0.45 &71.05 &{-0.14}& 0.80&-0.12 &{-0.22}& 0.70 &72.24 &{-0.19}& 1.08&-0.18  \\
    \Xhline{1.2pt}
            \multirow{3}{*}{$l=200$}
       &CDFSG     &{-0.37}& {0.28}&  31.56&{-0.37}& {0.14}&-0.37&{-0.67}& 0.30 &31.42 &{-0.67}& 0.17&-0.67 \\
    \cline{2-14}
       &CDFSG-Ada&\textbf{{-0.39}}& {0.11}&31.65&\textbf{{-0.39}}& {\textbf{0.06}} &-0.39 &{\textbf{-0.76}}& 0.16 &32.33 &{\textbf{-0.76}}& \textbf{0.11}&-0.76    \\
     \cline{2-14}
       &Landing     &{-0.09}&  0.20 &38.63& -0.09 & 0.40&  -0.08 &{-0.11}& 0.50 &37.75 &{-0.10}& 0.77&-0.09   \\
    \Xhline{1.2pt}
    \end{tabular}}
    \caption{Function value (Fval), feasibility violation (Fea), and  CPU Time of CDFSG, CDFSG-Ada, and Landing with different batch sizes on n-MNIST for different $p$.}
    \label{ta:fvalnmnh}
\end{table}

\begin{figure}[htbp]
    \centering
        \subfigure[Mediamill, $p=5$]{\includegraphics[width=0.22\hsize, height=0.2\hsize]{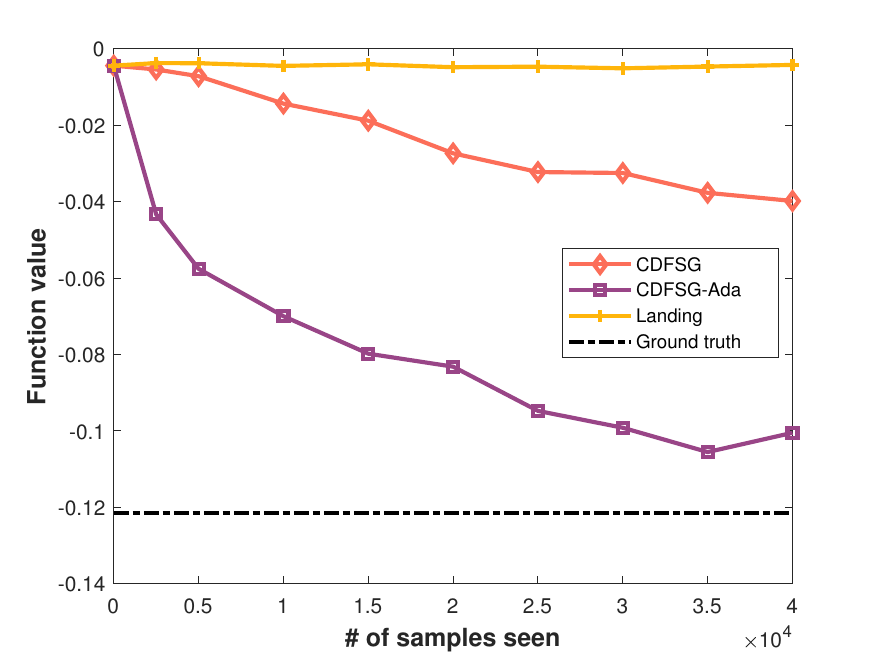}}
    \subfigure[Mediamill, $p=5$]{\includegraphics[width=0.22\hsize, height=0.2\hsize]{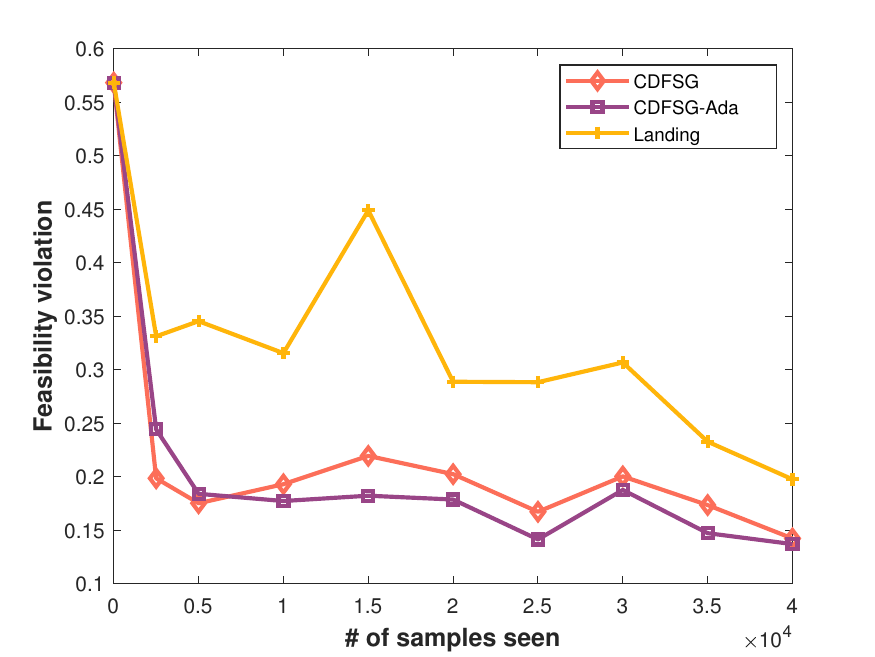}}
     \subfigure[Mediamill, $p=10$]{\includegraphics[width=0.22\hsize, height=0.2\hsize]{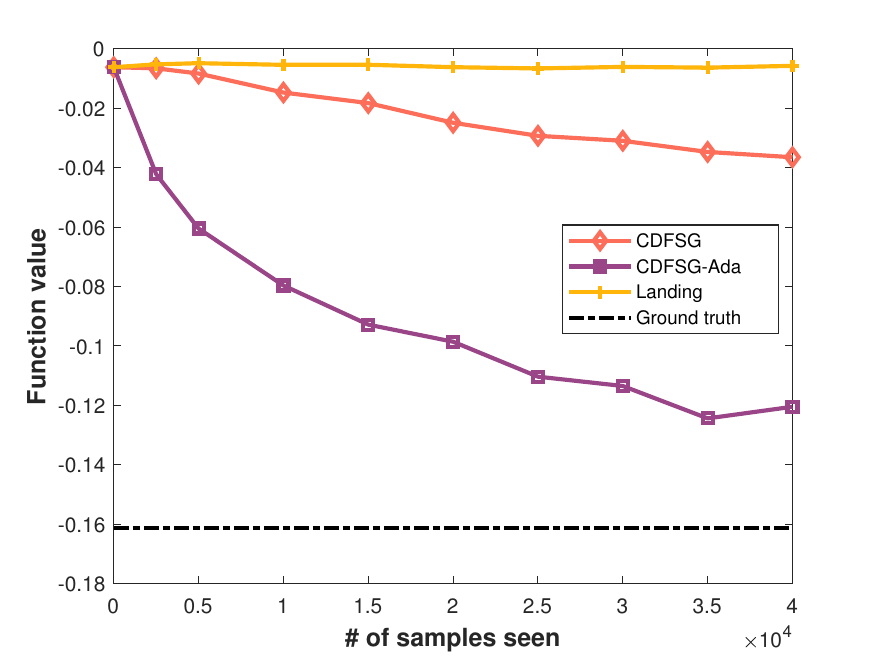}}
    \subfigure[Mediamill, $p=10$]{\includegraphics[width=0.22\hsize, height=0.2\hsize]{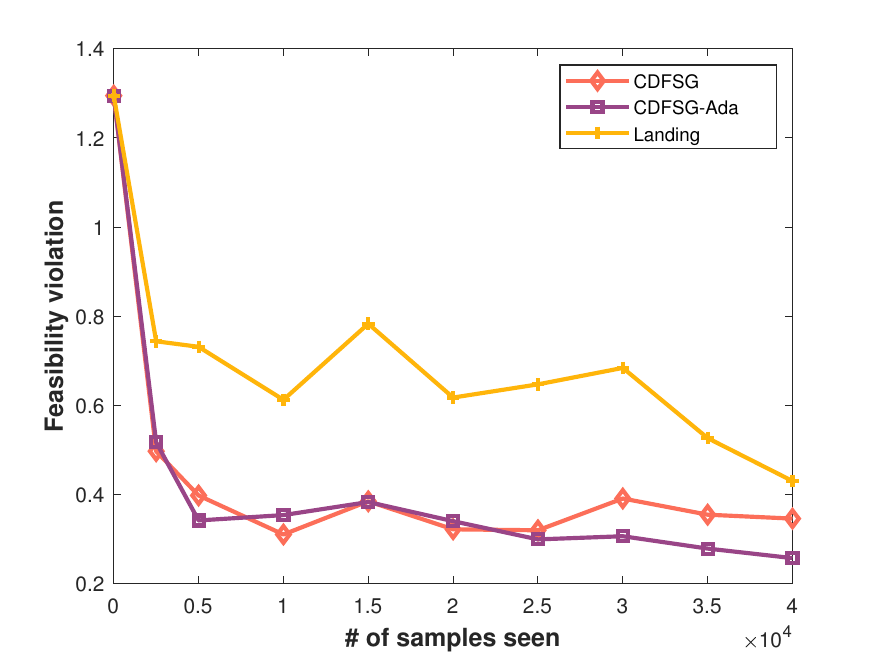}}\\
    \subfigure[MNIST, $p=5$]{\includegraphics[width=0.22\hsize, height=0.2\hsize]{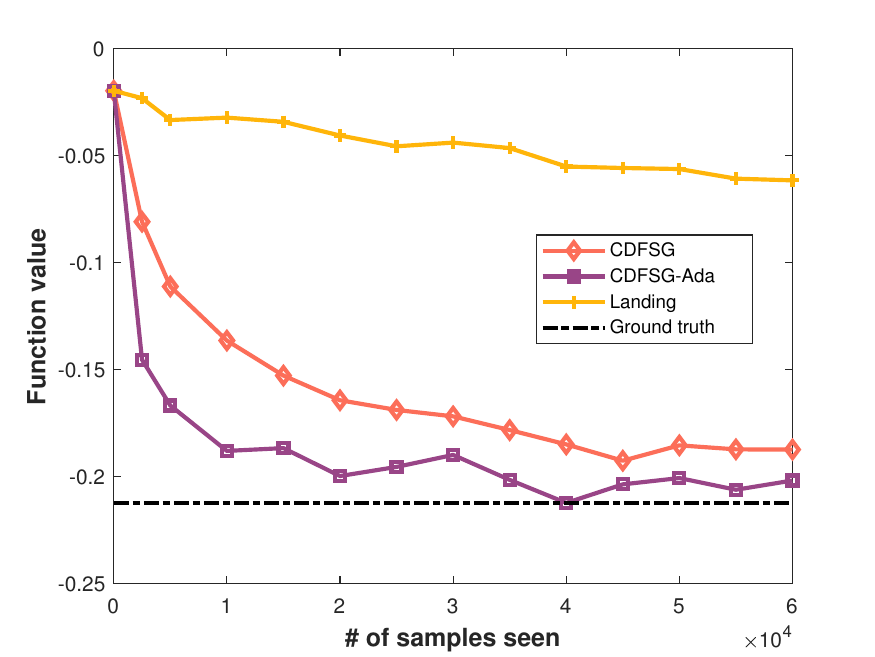}}
    \subfigure[MNIST, $p=5$]{\includegraphics[width=0.22\hsize, height=0.2\hsize]{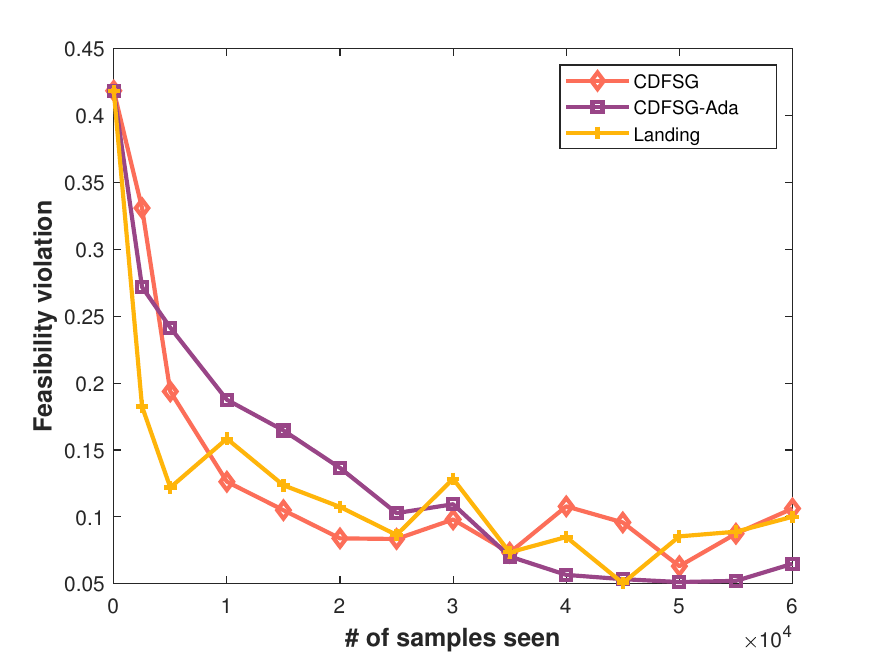}}
     \subfigure[MNIST, $p=10$]{\includegraphics[width=0.22\hsize, height=0.2\hsize]{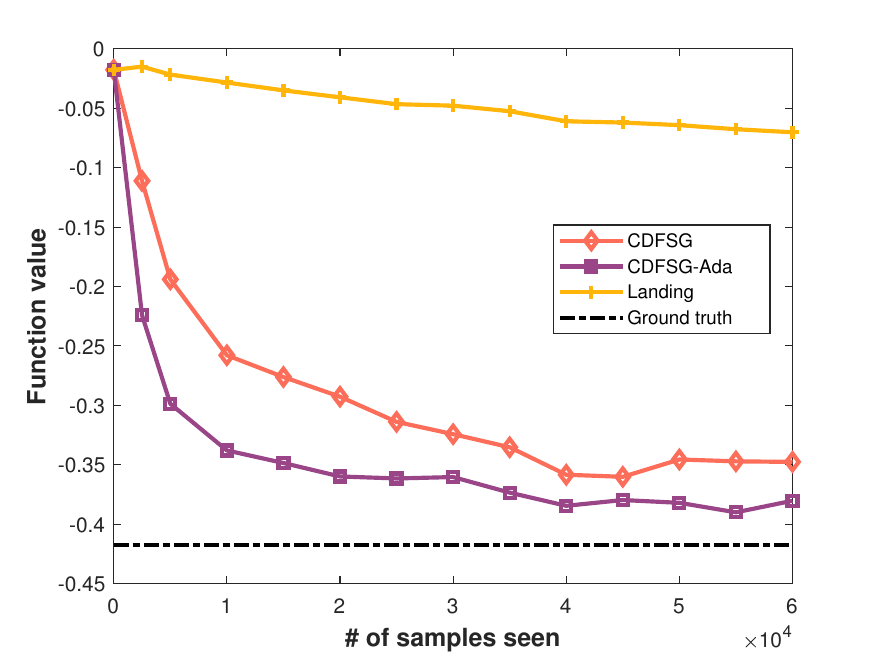}}
    \subfigure[MNIST, $p=10$]{\includegraphics[width=0.22\hsize, height=0.2\hsize]{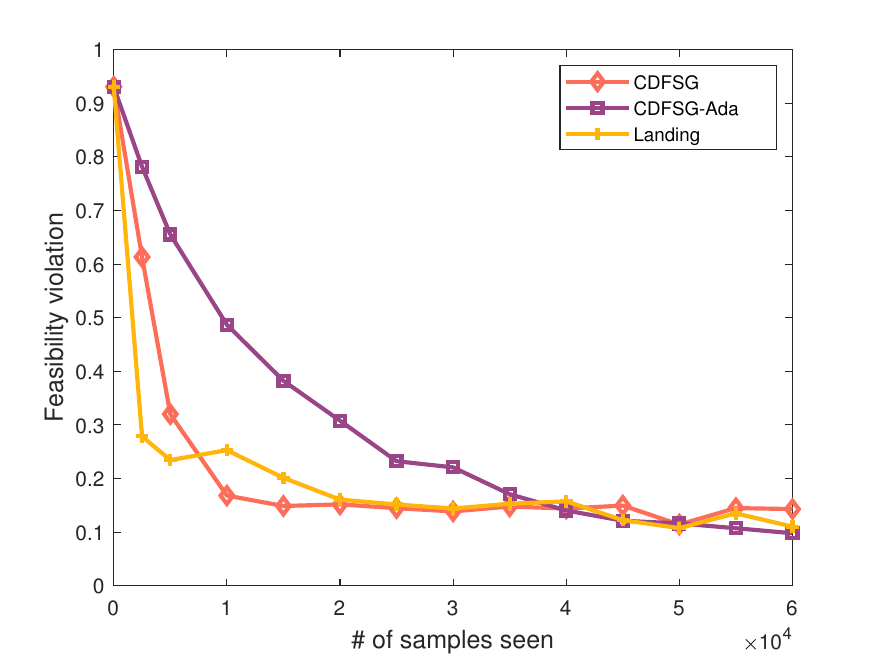}}
    \\
        \subfigure[n-MNIST, $p=5$]{\includegraphics[width=0.22\hsize, height=0.2\hsize]{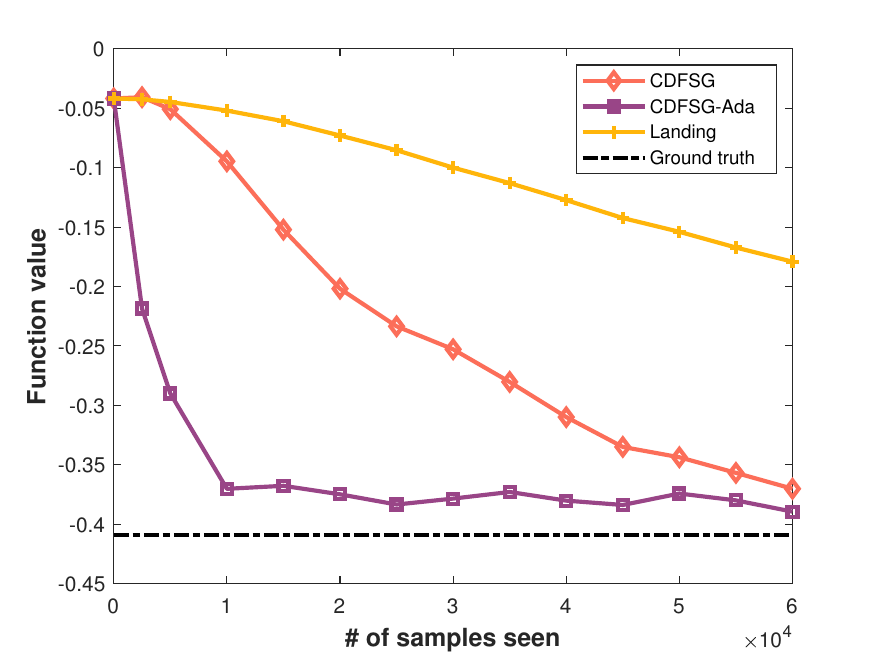}}
    \subfigure[n-MNIST, $p=5$]{\includegraphics[width=0.22\hsize, height=0.2\hsize]{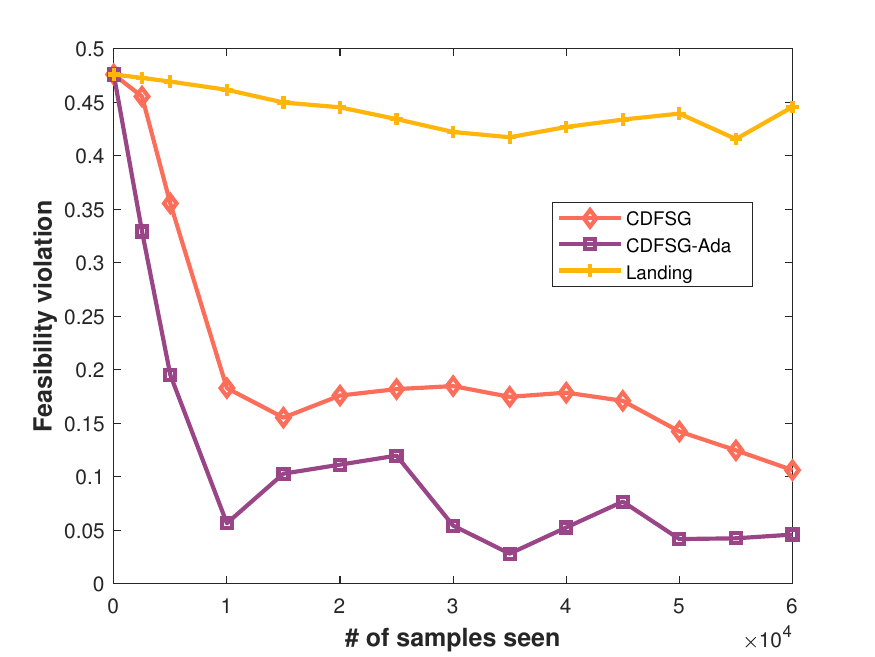}}
    \subfigure[n-MNIST, $p=10$]{\includegraphics[width=0.22\hsize, height=0.2\hsize]{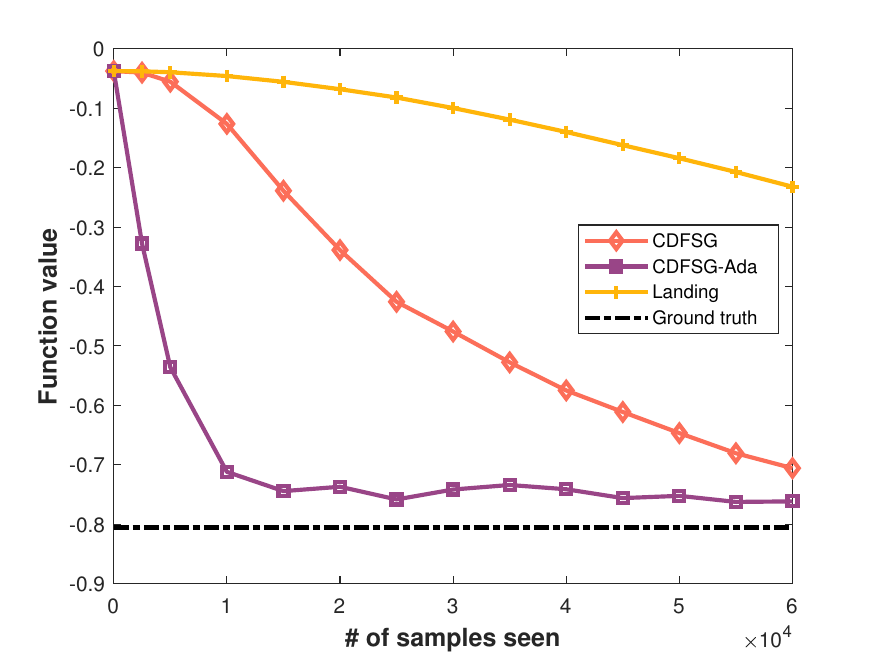}}
     \subfigure[n-MNIST, $p=10$]{\includegraphics[width=0.22\hsize, height=0.2\hsize]{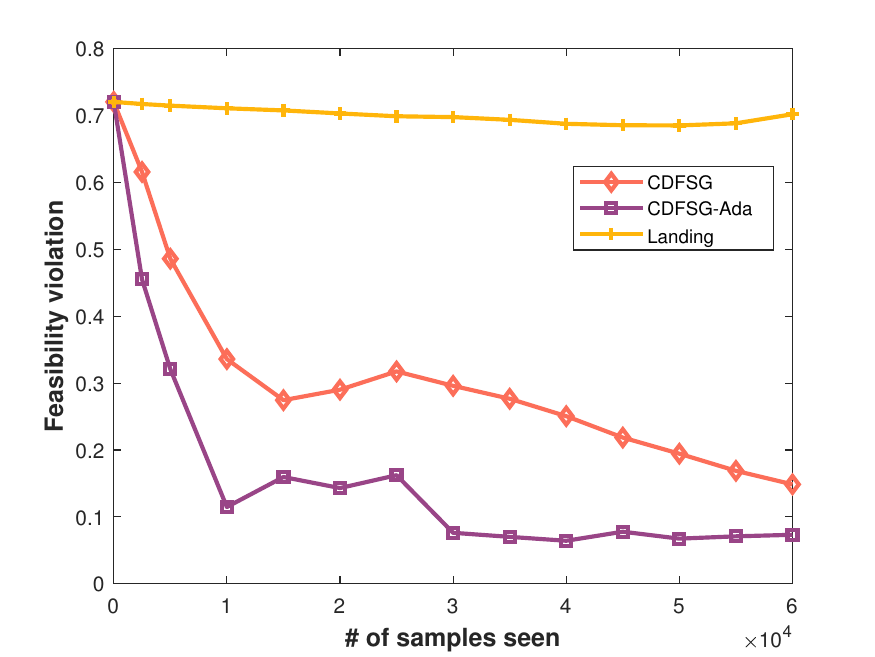}}
    \caption{Function value and feasibility violation variations  of CDFSG, CDFSG-Ada, and Landing with batch size fixed at $100$ on different datasets for CCA with different $p$.}
    \label{fig:fafe}
\end{figure}

\section{Conclusion}
In this paper, we focus on the stochastic optimization problem over the expectation-formulated generalized Stiefel manifold, described by \ref{sogse}, which has important real-world applications in various areas. The presence of stochasticity within the nonconvex constraints brings significant challenges in designing efficient optimization algorithms for solving \ref{sogse}.  To address these challenges, we  propose a constraint dissolving penalty function named \ref{cdfcp} with a novel sixth-order penalty term. The unconstrained minimization of \ref{cdfcp} is proved to be globally equivalent to \ref{sogse}. 

However, the nested expectation structure of \ref{cdfcp} precludes the direct application of classical stochastic gradient descent methods. To overcome this limitation, we employ inner function tracking techniques to develop a stochastic gradient algorithm and its accelerated variant with adaptive step-size strategy. These algorithms achieve a sample complexity of $\mathcal{O}(\varepsilon^{-4})$ for finding an $\varepsilon$-stationary point of \ref{cdfcp}.
Numerical experiments on the GCCA problem illustrate that our proposed algorithms achieve superior performance compared to existing efficient algorithms.  

Several aspects of this work deserve further studies, including investigating  the local convergence rate of our proposed algorithms and designing higher-order approaches by our proposed penalty function. In addition, the extension of our methodologies to nonsmooth optimization problems is of particularly interest as well.

\paragraph{Acknowledgement}
The work of Linshuo Jiang and Xin Liu was supported by the National Key R\&D Program of China (2023YFA1009300). The work of Xin Liu was supported in part by the National Natural Science Foundation of China (12125108, 12226008, 11991021, 11991020, 12021001, 12288201), Key Research Program of Frontier Sciences, Chinese Academy of Sciences (ZDBS-LY-7022), and CAS-Croucher Funding Scheme for Joint Laboratories ``CAS AMSS-PolyU Joint Laboratory of Applied Mathematics: Nonlinear Optimization Theory, Algorithms and Applications''.

\bibliography{SOGSEv0.1}
\end{document}